\documentclass[a4paper,reqno,
11pt
]{amsart}
\usepackage{
amssymb,
amsmath,
amsthm,
eucal,
empheq,
cases,
dsfont,
multicol,
mathrsfs,
tikz,
graphicx,
hyperref,
esvect
}
\usepackage{pgfplots}
\pgfplotsset{compat=1.17}

\usepackage{lipsum}
\makeatletter
\renewcommand*{\eqref}[1]{%
\hyperref[{#1}]{\textup{\tagform@{\!\!\ref*{#1}}}}%
}

\makeatletter
 
  \@addtoreset{equation}{section}
 \makeatother
\makeatletter

\newcommand{\opnorm}{\@ifstar\@opnorms\@opnorm}
\newcommand{\@opnorms}[1]{%
  \left|\mkern-1.5mu\left|\mkern-1.5mu\left|
   #1
  \right|\mkern-1.5mu\right|\mkern-1.5mu\right|
}
\newcommand{\@opnorm}[2][]{%
  \mathopen{#1|\mkern-1.5mu#1|\mkern-1.5mu#1|}
  #2
  \mathclose{#1|\mkern-1.5mu#1|\mkern-1.5mu#1|}
}

\makeatother\theoremstyle{plain}
\newtheorem{theorem}{Theorem}[section]
\newtheorem{lemma}[theorem]{Lemma}
\newtheorem{proposition}[theorem]{Proposition}

\theoremstyle{definition}

\newtheorem{remark}[theorem]{Remark}

\newcommand{\bvec}[1]{\mbox{\boldmath $#1$}}

\def\Re{\mathop{\mathrm{Re}}\nolimits}
\def\Im{\mathop{\mathrm{Im}}\nolimits}

\def\R{{\mathbb{R}}}

\def\N{{\mathbb{N}}}
\def\C{{\mathbb{C}}}

\def\F{{\mathcal{F}}}

\def\<{{\langle}}
\def\>{{\rangle}}

\def\ep{{\varepsilon}}
\def\ds{\displaystyle}

\title[Modified wave operators for subcritical Long-range NLS]{Large-data modified wave operators for the defocusing  nonlinear Schr\"odinger equation in one space dimension with subcritical long-range nonlinearity}

\author{Masaki Kawamoto}\address[M. Kawamoto]{Research Institute for Interdisciplinary Science, Okayama University, 3-1-1, Tsushimanaka, Kita-ku, Okayama City, Okayama, 700-8530, Japan}\email{kawamoto.masaki@okayama-u.ac.jp}
\author{Haruya Mizutani}\address[H. Mizutani]{Department of Mathematics, Graduate School of Science, The University of Osaka, Toyonaka, Osaka 560-0043, Japan}\email{haruya@math.sci.osaka-u.ac.jp}

\makeatletter
\@namedef{subjclassname@2020}{%
	\textup{2020} Mathematics Subject Classification}
\makeatother


\begin{document}

%
\begin{abstract}
We study long-time behavior of the solutions to the final state problem for the defocusing nonlinear Schr\"odinger equation (NLS) in one space dimension with the power nonlinearity $|u|^{2\sigma}u$ in the subcritical long-range regime $\frac{2}{\sqrt{7}}<\sigma<1$. Given a prescribed asymptotic profile in a weighted $L^2$-space, without size restriction, obtained by modifying the free solution with a nonlinear polynomial phase correction, we construct a unique global solution of the NLS that scatters to this profile, thereby proving the existence of modified wave operators. The proof relies on two new ingredients. Extending our previous work for the cubic case, we incorporate 
the leading part of the nonlinear term into the linear part as a linear potential by linearizing the NLS around the asymptotic profile and prove a global modified energy estimate for the linearized equation. We also exploit a specific structure of the nonlinearity arising from the linearization, which gives rise to a crucial cancellation when estimating the nonlinear terms in the modified energy space and enables us to control the polynomial growth of the nonlinear phase correction in the subcritical case.


%
\end{abstract}

\maketitle

\section{Introduction}
\label{introduction}
The present paper is a continuation of our previous study \cite{KaMi} on modified scattering theory for  the following nonlinear Schr\"odinger equation (NLS) with power nonlinearity: 
\begin{align}
\label{NLS}
i\partial_t u+\frac12\Delta u=\lambda |u|^{2\sigma}u,\quad x\in \R^d,\quad t\in \R, 
\end{align}
where $u=u(t,x)$ is a $\C$-valued unknown function,  $\Delta=\partial_{x_1}^2+\cdots+\partial_{x_d}^2$, $0<\sigma\le 1/d$ and $\lambda\in \R$.  

The main result of \cite{KaMi} is existence of modified wave operators on a weighted $L^2$-space for the associated final state problem without size restriction on scattering data in the one-dimensional defocusing cubic case: $d=\sigma=1$ and $\lambda>0$. The aim of this paper is to extend this result to the following subcritical case (again in the one-dimensional defocusing setting): $$\frac{2}{\sqrt{7}}<\sigma<1.$$  
For the case with large scattering data, this provides the first construction of modified wave operators for the NLS with the subcritical power nonlinearity. 


Before stating the main result precisely, we first recall the basic concept of the final state problem for \eqref{NLS}. The primary focus  is to construct a unique global solution $u$ to \eqref{NLS} satisfying the following prescribed asymptotic condition at infinity: 
\begin{align}
\label{scattering}
\|e^{-it\Delta/2}\{u(t)-u_{\mathrm p,+}(t)\}\|_{X}\to 0,\quad t\to +\infty,
\end{align}
where  $u_{\mathrm p,+}=u_{\mathrm p,+}(t,x)$ is a given asymptotic profile and $X\subset L^2(\R^d)$ is a suitable function space on $\R^d$. We say that $u$ scatters to $u_{\mathrm p,+}$ in $X$ if \eqref{scattering} holds. Note that \eqref{scattering} also implies $\|u(t)-u_{\mathrm p,+}(t)\|_{L^2}\to 0$. 
A possible canonical  choice of the profile $u_{\mathrm p,+}$ is the free solution $e^{it\Delta/2}u_+$ with the scattering datum $u_+$. In such a case, the wave operator $W_+$ is defined by
\begin{align}
\label{wave_operator}
W_+(u_+):=u(0).
\end{align}
Similarly, one can also consider the final state problem in the negative time direction $t\to -\infty$ and the associated wave operator $W_-(u_-):=u(0)$. Constructing the wave operators  is a crucial step to construct the scattering operator $S=W_+^{-1}\circ W_-$, which is an important object in time-dependent scattering theory; see \cite[Section 7]{Cazenave} for more details. 
However, the wave operators $W_\pm$ do not always exist. In other words, the free solution is not always a relevant choice and the nonlinear effect must be taken into account. Indeed, it is known that no non-trivial solution scatters to a free solution if $\sigma\le 1/d$ (\cite{Strauss,Barab}). 
On the other hand, for $\max\{1/d,2/(d+2)\}<\sigma<2/(d-2)_+$, the wave operators $W_\pm$ exist in $X=H^1\cap \F H^1$; see e.g. \cite{GOV_1994} and \cite[Section 7]{Cazenave}. 


In this paper we are interested in the long-range case $\sigma\le1/d$. To introduce the asymptotic profile for the long-range case, we recall the Dollard decomposition of the free propagator: 
\begin{align}
\label{Dollard}
e^{it\Delta/2}=\mathcal M(t)\mathcal D(t)\mathcal F\mathcal M(t), 
\end{align}
where $\mathcal F f=\widehat f$ denotes the Fourier transform of $f$, $\mathcal M(t)=e^{\frac{ix^2}{2t}}$ and $\mathcal D(t)f(x)=(it)^{-\frac d2}f(x/t)$. This decomposition can be verified by factorizing  the integral kernel of $e^{it\Delta/2}$ as
$$
(2\pi it)^{-\frac d2}e^{\frac{i|x-y|^2}{2t}}=e^{\frac{i|x|^2}{2t}}( it)^{-\frac d2}(2\pi)^{-\frac d2}e^{-i\frac{x}{t}\cdot y}e^{\frac{i|y|^2}{2t}}.
$$
Since $\mathcal M(t)\to 1$ as $t\to\pm\infty$, the leading term of $e^{it\Delta/2}f$ is given by $\mathcal M(t)\mathcal D(t)\widehat f$, namely 
$$
e^{it\Delta/2}f=\mathcal M(t)\mathcal D(t)\widehat f+o_{L^2}(1),\quad t\to \pm\infty.
$$
Then, for a given scattering datum $u_\pm$,  the asymptotic profile  $u_{\mathrm p,\pm}$ for the long-range case is defined by adding a nonlinear phase correction to this linear profile function (\cite{ZaMa,Ozawa_1991,HKN}): 
\begin{align}
\label{u_p}
u_{\mathrm{p},\pm}(t,x)=[\mathcal M(t) \mathcal D(t) w_{\mathrm{p},\pm}](t,x)=(it)^{-\frac d2}e^{\frac{ix^2}{2t}}\widehat{u_\pm}\left(\frac xt\right)\exp\left(\mp i\gamma(t) \left|\widehat{u_\pm}\left(\frac xt\right)\right|^{2\sigma}\right)
\end{align}
where $w_{\mathrm{p},\pm}(t,x)=\widehat{u_\pm}(x)\exp\left(\mp i\gamma(t) |\widehat{u_\pm}(x)|^{2\sigma}\right)$ and
\begin{align}
\label{w_p}
\gamma(t)=\begin{cases}\ds 
\frac{\lambda |t|^{1-d\sigma}}{1-d\sigma}&\text{if}\quad 0<\sigma<1/d,\bigskip\\ 
 \lambda \log|t|&\text{if}\quad \sigma=1/d.\end{cases}
\end{align}
It is easy to see that $w_{\mathrm{p},\pm}$ satisfies the following ODE (often called the reduced ODE): 
\begin{align}
\label{NLS_w_p}
i\partial_t w_{\mathrm{p},\pm}=\lambda t^{-d\sigma} |w_{\mathrm{p},\pm}|^{2\sigma}w_{\mathrm{p},\pm},\quad \pm t>0,\quad x\in \R^d.
\end{align}
Note that the results on modified scattering for the linear Schr\"odinger equation with potentials and for the Hartree equations suggest that, for sufficiently small $\sigma$, this phase correction is not sufficient and higher-order corrections would be required; see also Remarks \ref{remark_2} and \ref{remark_SeWu} below. We also emphasize the polynomial growth of $\gamma(t)$ in the subcritical case $\sigma<1/d$, whereas it is only logarithmic in the critical case $\sigma=1/d$. In addition to the slow time-decay rate $t^{-d\sigma}$ in \eqref{NLS_w_p}, this polynomial growth constitutes a major obstacle to treating the subcritical case in the weighted space 
$$
X=\F H^\nu(\R^d)=\{f\in L^2(\R^d)\ |\ \widehat f\in H^\nu(\R^d)\}=\{f\in L^2(\R^d)\ |\ |x|^\nu f\in L^2(\R^d)\}.
$$ Indeed, since the weight $|x|^\nu$ corresponds to the derivative $|D|^\nu$ under the Fourier transform and also under the pseudo-conformal transform \eqref{v} used in this paper (see the formula \eqref{lemma_PC_2}), we have to deal with terms that grow polynomially in time when estimating the asymptotic profile in weighted spaces. 

Next we review existing literature on the long-range case $\sigma\le 1/d$. For the critical case $\sigma=1/d$, small-data modified scattering for the final state problem was established by Ozawa \cite{Ozawa_1991} for $d=1$ and Ginibre--Ozawa \cite{Ginibre_Ozawa_1993} for $d=2,3$. Precisely, for any $u_+\in \F H^2(\R^d)$ and $\lambda\in \R$ satisfying $|\lambda|\|\widehat{u_+}\|_{L^\infty}^2\ll1$ there exists a unique global solution $u\in C(\R,L^2(\R^d))$ to \eqref{NLS} that scatters to $u_{\mathrm p,+}$ in $L^2(\R^d)$.  Note that $\|\widehat u\|_{L^\infty}$ is an invariant norm with respect to the scaling $(0,\infty)\ni\rho\mapsto \rho^{d}u(\rho^2t,\rho x)$ that leaves the NLS \eqref{NLS} with $\sigma=1/d$ invariant. Thus, this smallness condition cannot be eliminated simply by scaling alone. Later, the conditions on $u_+$ and the topology $X$ of scattering were improved by Hayashi--Naumkin \cite{Hayashi_Naumkin_2006} to $u_+\in \F H^\alpha(\R^d)$ and $X=\F H^\beta(\R^d)$ with $d/2<\beta<\alpha<\min\{d,2,1+2/d\}$, though the same smallness condition was still assumed. Here the restriction $d\le3$ comes from the condition $d/2<1+2/d$ related with the Sobolev embedding $H^{d/2+\ep}(\R^d)\subset L^\infty(\R^d)$ and the regularity of the nonlinearity $|u|^{2/d}u$. We will provide briefly a proof strategy by \cite{Hayashi_Naumkin_2006} in the beginning of Section \ref{subsection_energy}. It follows from these results that the map \eqref{wave_operator} is still well defined. In this case,  $W_+$ is called the modified wave operator as $u_{\mathrm p}$ has the nonlinear phase modification. For the properties of the modified wave and modified scattering operators (in the small-data regime), we refer to \cite{Carles_2001,Hayashi_Naumkin_2006, Carles_2024}. 

Small data modified scattering for the Cauchy problem of \eqref{NLS}  has also been  extensively studied in the critical case $\sigma=1/d$ with $d\le3$. Since the present paper focuses on the final state problem, we only refer to a seminal paper \cite{Hayashi_Naumkin_1998} by Hayashi--Naumkin and subsequent works \cite{Lindblad_Soffer_2006,KaPu,IfTa} for alternative approaches in one space dimension. 

Compared with the small-data results in the critical case, the large-data or subcritical cases are much less understood. Besides, as shown below, the existing results rely heavily upon either specific features of the cubic nonlinearity in $d=1$ or very restrictive conditions on given data. 

As for the large-data problem, modified scattering for the final state problem has been established by \cite{Ginibre_Velo_2001} for  $d=\sigma=1$ and $\lambda\in \R$ within the framework of an analytic function space (a kind of weighted Gevrey classes of order 1). 
\cite{Deift_Zhou_2002,Deift_Zhou_2003} established the large-data modified scattering in the weighted Sobolev space $\Sigma=H^1(\R)\cap \F H^1(\R)$ for the Cauchy problem of \eqref{NLS} with  $d=\sigma=1$ and $\lambda>0$ by using the inverse scattering methodology. As the real analyticity of the cubic nonlinearity or the complete integrability of the cubic NLS played essential roles in these results, it is difficult to apply these methods to the subcritical cases. The large-data modified scattering for the Cauchy problem was also established by \cite{Cazenave_Naumkin_2018} in the critical case $\sigma=1/d$ for any $d\ge1$. However, they only treated non-vanishing, highly oscillating initial data of the form $u_0=e^{ib|x|^2}v_0$ with sufficiently large $b$, where $v_0$ belongs to a suitable weighted Sobolev space and satisfies $\ds \inf_{x\in \R^d} (1+|x|)^N|v_0(x)|>0$ with large $N$. More recently, \cite{Georgiev_Ozawa} established the large-data modified scattering in $L^2$ for the Cauchy problem in the defocusing critical case $\sigma=1/d$ and $\lambda>0$ with $d=1,2$, under the assumption that the solution $u$ satisfies the same $L^\infty$ decay estimate as that of the free solution, namely $\|u(t)\|_{L^\infty}\lesssim t^{-d/2}$. 

As for the subcritical case, the literature is much more sparse. In \cite{HKN}, the authors proved modified scattering for the Cauchy problem of \eqref{NLS} in $L^2(\R)\cap L^\infty(\R)$ for $1/2<\sigma<1$ and $d=1$ with small initial data $u_0$ which belongs to a suitable analytic function space and satisfies the non-vanishing condition $\ds \inf_{x\in \R^d}(1+|x|)^N|u_0(x)|>0$ with sufficiently large $N$. Later, \cite{Han} extended the result by \cite{HKN}  to the multi-dimensional case $d\ge2$ and ${2}/{(d+\sqrt{d^2+8d})}<\sigma<1/d$. As for the final state problem of \eqref{NLS}, very recently, \cite{SeWu} constructed the modified wave operators for all $d\ge1$, $0<\sigma<1/d$ and $\lambda\in \R$, and for sufficiently small, non-vanishing scattering data $u_+$ in a suitable weighted analytic function space; see Remark \ref{remark_SeWu} for a further discussion. As far as we know, there is no existing literature on  modified scattering for arbitrarily large scattering data in the subcritical case. 


In the previous work \cite{KaMi}, we introduced an approach to the final state problem with large scattering data for the defocusing critical case, based on linearization around the asymptotic profile and a modified energy estimate for the resulting linearized equation. In the present paper, we further develop this approach to treat the subcritical case for arbitrarily large scattering data in weighted $L^2$-spaces without either structural conditions or the analyticity. In addition to the modified energy method developed in \cite{KaMi}, the main new ingredient is the use of a specific structure of the nonlinearity arising from linearization. This structure yields a crucial cancellation in the modified energy estimates for the nonlinear and error terms, enabling us to handle difficulties specific to the subcritical case such as the polynomial growth of the phase correction mentioned above and the non-smoothness of the nonlinearity. 

Finally, it should be mentioned that, for the Hartree type equation (\eqref{NLS} with the nonlinearity replaced by $\lambda (|x|^{-\sigma}*|u|^2)u$), the existence of modified wave operators has been extensively studied for $0<\sigma\le 1$ and $d\ge2$, and much stronger results than those available for \eqref{NLS} have been obtained; see \cite{Ginibre_Velo_2000_1,Ginibre_Velo_2000_2,Ginibre_Velo_2001,Nakanishi_1,Nakanishi_2,Ginibre_Velo_2014,Ginibre_Velo_2015}. However, it is difficult to apply the methods developed in these works to \eqref{NLS} since the smoothing property of the convolution operator in the high frequency region, namely $|x|^{-\sigma}*=c|\nabla|^{\sigma-d}:\dot H^\alpha\to \dot H^{\alpha+d-\sigma}$, plays an essential role in these works. We refer to \cite[Subsection 1.4]{KaMi} for more detailed explanation on the comparison of our approach and those developed for the Hartree equations.


\subsection{Main result}
Now we state the main result in this paper. 

\begin{theorem}
\label{theorem_1} Let $d=1$, $\lambda>0$ and $\frac{2}{\sqrt{7}}<\sigma<1$. Define
$$
\nu_*(\sigma)=
\begin{cases}
\frac{4-3\sigma^2}{2\sigma(2\sigma-1)}&\text{if}\quad \frac{2}{\sqrt7}\le\sigma\le\frac{2+4\sqrt3}{11},\\
2 &\text{if}\quad \frac{2+4\sqrt3}{11}\le\sigma\le\frac 67,\\
\frac{4-3\sigma}{2\sigma-1}&\text{if}\quad \frac 67\le \sigma\le1.
\end{cases}
$$
Then, for any $u_+ \in  \mathcal FH^{\nu}(\R)$ with $\nu_*(\sigma)<\nu<\frac{\sigma}{\sigma-1/2}$, there exists a unique global solution $u\in C(\R, L^2(\R))$ to \eqref{NLS} satisfying $e^{-it\Delta/2}u\in C(\R, \mathcal F H^1(\R))$ and, for any $\delta>0$,
\begin{align}
&\|xe^{-it\Delta/2}\{u(t)-u_{\mathrm p,+}(t)\}\|_{L^2}+t^{\sigma/2-\delta/2}\left\|u(t)-u_{\mathrm p,+}(t)\right\|_{L^2}\nonumber\\
&\quad +t^{-\sigma/2+3/2}\left\|\left|\widehat{u_+}\left(\frac xt\right)\right|^{\sigma-1}\Re\left[\overline{u_{\mathrm p,+}(t)}\{u(t)-{u_{\mathrm p,+}(t)}\}\right]\right\|_{L^2}
\lesssim t^{-\beta+\delta}
\label{theorem_1_1}
\end{align}
uniformly with respect to $t\ge1$, where
\begin{align}
\beta=\beta(\sigma,\nu)=\begin{cases}
\nu(\sigma-1/2)+{3\sigma}/{2}-2&\text{if}\quad 1\le \nu<2,\\
\nu(\sigma-1/2)+\sigma/2-1&\text{if}\quad 2\le \nu<\frac{\sigma}{\sigma-1/2}.
\end{cases}
\label{beta}
\end{align}
As a consequence, the following modified wave operator is well-defined: $$W_{\sigma,+}:\F H^{\nu}(\R)\ni u_+\mapsto u(0)\in \mathcal F H^1(\R).$$ 
\end{theorem}

\begin{remark}
While the critical case $\sigma=1$ is excluded from the statement of the theorem as it was already proved in our previous paper \cite{KaMi}, a minor modification of the proof shows that the theorem also holds for $\sigma=1$. This provides a slightly different proof of the main result of \cite{KaMi}, where the contraction mapping theorem was used, whereas here we use an energy method based on a weak compactness argument to construct the solution.
\end{remark}

\begin{remark}
\label{remark_1}
By \eqref{theorem_1_1}, the solution $u$ satisfies
\begin{align}
\|xe^{-it\Delta/2}\{u(t)-u_{\mathrm p,+}(t)\}\|_{L^2}&\lesssim t^{-\beta+\delta},\nonumber\\
\left\|u(t)-u_{\mathrm p,+}(t)\right\|_{L^2}&\lesssim t^{-\beta-\sigma/2+3\delta/2},\nonumber\\
\left\|\left|\widehat{u_+}\left(\frac xt\right)\right|^{\sigma-1}\Re\left[\overline{u_{\mathrm p,+}(t)}\{u(t)-{u_{\mathrm p,+}(t)}\}\right]\right\|_{L^2}&\lesssim t^{-\beta-(3-\sigma)/2+\delta}. \label{remark_1_1}
\end{align}
Combining with the estimates $\|e^{it\Delta/2}(1+|x|)^{-1}\|_{L^2\to L^\infty}\lesssim t^{-1/2}$ and $\|u_{\mathrm p,+}(t)\|_{L^\infty}\lesssim t^{-1/2}$, we also have
\begin{align}
\label{remark_1_2}
\|u(t)-u_{\mathrm p,+}(t)\|_{L^\infty}\lesssim t^{-1/2-\beta+\delta},\quad \|u(t)\|_{L^\infty}\lesssim t^{-1/2}. 
\end{align}
These decay estimates, except for \eqref{remark_1_1}, have been usually obtained in the small-data scattering result for the critical case $\sigma=1$; see e.g. \cite{Ozawa_1991,Ginibre_Ozawa_1993,Hayashi_Naumkin_2006}. \eqref{remark_1_1} was implicitly obtained in the previous work \cite{KaMi} again in the critical case $\sigma=1$. This additional information is not only a consequence, but also important for proving Theorem \ref{theorem_1} itself. More precisely, the last term of the LHS of \eqref{theorem_1_1} plays a crucial role in our argument. 
\end{remark}

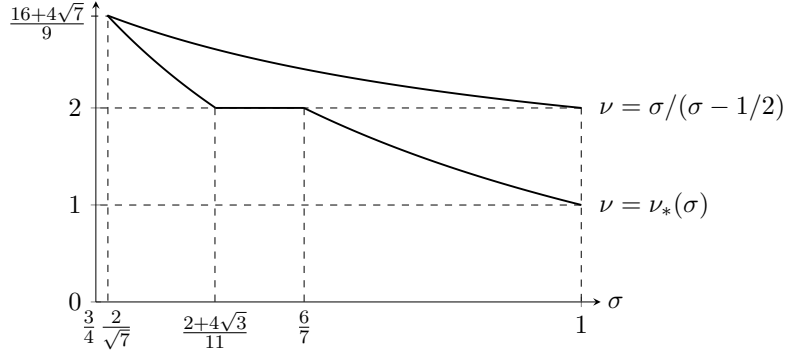
\begin{figure}[htbp]
\label{figure_1}
\begin{center}
\scalebox{0.9}[0.9]{
\begin{tikzpicture}
  \begin{axis}[
    width=9cm,
    height=6cm,
    xmin=0.75, xmax=1.01,
    ymin=0, ymax=3.1,
    xtick={0,1},
    ytick={0,1,2},
    axis lines=left,
    samples=200,
  ]

  \addplot[thick, domain={2/sqrt(7)}:{(2+4*sqrt(3))/11}]
    {(4-3*x^2)/(2*x*(2*x-1))};

  \addplot[thick, domain={(2+4*sqrt(3))/11}:{6/7}]
    {2};

  \addplot[thick, domain={6/7}:1]
    {(4-3*x)/(2*x-1)};

  \addplot[thick, domain={2/sqrt(7)}:{1}]
    {x/(x-1/2)};

  \end{axis}
\draw (7.132,2.86) node[right] {$\ \nu=\sigma/(\sigma-1/2)$};
\draw (7.132,1.424) node[right] {$\ \nu=\nu_*(\sigma)$};
\draw (7.132,0) node[right] {$\ \ \sigma$};
\draw[dashed] (7.132,2.86) -- (7.132,0);
\draw[dashed] (7.132,1.424) -- (0,1.424);
\draw[dashed] (7.132,2.85) -- (0,2.85);
\draw[dashed] (3.06,2.85) -- (3.06,0);
\draw (3.06,0) node[below] {$\frac67$};
\draw[dashed] (1.75,2.85) -- (1.75,0);
\draw (1.75,0) node[below] {$\frac{2+4\sqrt3}{11}$};
\draw[dashed] (0.177,4.13) -- (0.177,0);
\draw (0.3,0) node[below] {$\frac{2}{\sqrt7}$};
\draw (-0.1,0) node[below] {$\frac34$};
\draw[dashed] (-0.07,4.2) -- (0.177,4.2);
\draw (0,4.13) node[left] {$\frac{16+4\sqrt7}{9}$};
\end{tikzpicture}}
\end{center}
\caption{The admissible region of $(\sigma,\nu)$. We take the line $\sigma=\frac34$ as the $\nu$-axis.}
\end{figure}

\begin{remark}
\label{remark_2}
It is easy to check that $\nu_*(\sigma)$ is continuous and non-increasing, $\nu_*(1)=1$ and $$\nu_*(2/\sqrt{7})=\sigma/(\sigma-1/2)|_{\sigma=2/\sqrt7}=(16+4\sqrt{7})/9<3.$$ Thus, as a direct corollary of Theorem \ref{theorem_1}, the modified wave operator $W_{\sigma,+}:\F H^3(\R)\to \F H^1(\R)$ exists for any $2/\sqrt7<\sigma<1$. 

$\nu_*(\sigma)$ and the lower bound $2/\sqrt7$ of the admissible range of $\sigma$  are determined as follows; see Section \ref{subsection_existence_3} for the details. Define $\mu(\sigma,\nu):=\min\{\mu_1,\mu_2,\mu_3\}$, where
\begin{align*}
\mu_1(\sigma,\nu)&=\frac{\sigma^2}{2}+2\sigma-2+2\sigma \beta(\sigma,\nu),\\
\mu_2(\sigma,\nu)&=\frac{\sigma^2}{2}+\frac{(\min\{\nu,3/2\}+1)\sigma}{2}-1+(2\sigma-1)\beta(\sigma,\nu),\\
\mu_3(\sigma,\nu)&=\frac{5\sigma}{4}-1+\beta(\sigma,\nu).
\end{align*}
Then
\begin{align*}
2/\sqrt7&=\inf\{\sigma\in (1/2,1)\ |\ \mu(\sigma,\sigma/(\sigma-1/2))>0,\ \beta(\sigma,\sigma/(\sigma-1/2))>0\},\\
\nu_*(\sigma)&:=\inf\{\nu\in (1,\sigma/(\sigma-1/2))\ |\ \mu(\sigma,\nu)>0,\ \beta(\sigma,\nu)>0\}.
\end{align*}
More precisely, if $\nu=\sigma/(\sigma-1/2)$, then $\mu(\sigma,\sigma/(\sigma-1/2))=\mu_1(\sigma,\sigma/(\sigma-1/2))=7\sigma^2/2-2$ and $\beta(\sigma,\sigma/(\sigma-1/2))=3\sigma/2-1$. Thus the lower bound $2/\sqrt{7}$ is determined as the positive root of $\mu(\sigma,\sigma/(\sigma-1/2))=0$. 
This lower bound is due to technical limitations of our method. In view of \cite{HKN} and by analogy with modified scattering for the Hartree equations, when $d=1$, $\sigma=1/2$ should be the optimal lower bound for modified scattering with the above asymptotic profile \eqref{u_p}; see also Remark \ref{remark_SeWu} below. We however stress that our argument is completely different from that of \cite{HKN}. In particular, neither structural conditions nor an analyticity are required for scattering data $u_+$. Moreover, as mentioned above, it is difficult to apply directly the existing methods for the Hartree equations based on the smoothing property of the convolution to \eqref{NLS}. 
\end{remark}

\begin{remark}
\label{remark_SeWu}
During the preparation of this article, the very recent preprint \cite{SeWu} appeared, which established the existence of modified wave operators for \eqref{NLS} with $d\ge1$, $0<\sigma<1/d$,  $\lambda\in \R$ and sufficiently small, non-vanishing scattering data $u_+$ in an analytic function space. More precisely, let $A^6_\rho$ be a weighted analytic function space defined through the norm 
$$
\|f\|_{A^6_\rho}=\int_{\R^d}e^{6\<\xi\>}\<\xi\>^\rho |\widehat f(\xi)|d\xi. 
$$
The scattering data in \cite{SeWu} is assumed to be of the form $u_+=e^{\Psi}$ with smooth $\Psi$ satisfying $u_+\in L^2\cap A^6_\rho$, $\nabla \Psi\in A^6_\rho$ and the smallness condition $\|u_+\|_{L^2\cap A^6_\rho}+\|e^{\sigma \Re\Psi}\|_{A^6_\rho}\ll1$ with some $\rho>0$. In particular, $u_+$ is nowhere vanishing. Then they showed the modified scattering in $L^2$, thereby constructing the modified wave operator $W_{\sigma,+}(u_+)\in L^2$. 

 The method in \cite{SeWu} successfully constructs the modified wave operators for all $d\ge1$, $0<\sigma<1/d$ and $\lambda\in \R$ under the above assumptions on the scattering data. Notably,  they establish the relevant asymptotic profile for $0<\sigma\le 1/(2d)$ and a mechanism to determine it, both of which were previously unknown in the literature.
 Compared with their result, our admissible range of $\sigma$ is narrow and we consider the one-dimensional defocusing case only. However, the advantage of our result is that we treat general scattering data $u_+\in \mathcal F H^\nu(\R)$ without either smallness or non-vanishing conditions, besides we construct the modified wave operator not only on $L^2(\R)$ but also on $\mathcal F H^1(\R)$ which particularly yields the sharp decay estimates \eqref{remark_1_2}. Furthermore, we deal with the difficulty caused by  the low regularity at $u=0$ of the subcritical nonlinearity $|u|^{2\sigma}u$ which does not appear under non-vanishing conditions. 
 
 The two approaches pursue fundamentally different objectives and employ distinct methodologies in the construction of modified wave operators for subcritical nonlinearities. While the approach of \cite{SeWu} succeeds in treating the full range of exponents and general space dimensions at the cost of imposing above restrictions on the scattering data, our approach, by restricting the admissible dimensions and exponent range, enables us to handle a broader class of large scattering data. Thus, the two works address different aspects of the problem.
 \end{remark}

\begin{remark}We list further remarks on Theorem \ref{theorem_1}. 
\label{remark_theorem_1}
\begin{enumerate}
\item Precisely speaking, the uniqueness in Theorem \ref{theorem_1} means that if $u_1,u_2\in C(\R, L^2(\R))$ are two solutions to \eqref{NLS} satisfying $e^{-it\Delta/2}u_j\in C(\R, \mathcal F H^1(\R))$ for $j=1,2$ and \eqref{theorem_1_1} for any $\delta>0$ with the same asymptotic profile $u_{\mathrm p,+}(t)$, then $u_1\equiv u_2$. 
\item By the time reversal symmetry of \eqref{NLS}, the analogous statement also holds for the negative time direction $t\to -\infty$. Namely, for any $u_-\in \F H^{\nu}(\R)$ there exists a unique global solution $u\in C(\R,L^2(\R))$ to \eqref{NLS} satisfying $e^{-it\Delta/2}u\in C(\R, \mathcal F H^1(\R))$ and \eqref{theorem_1_1} with $u_{\mathrm p,+}$ replaced by $u_{\mathrm p,-}$ as $t\to -\infty$. In particular, the negative time modified wave operator $W_{\sigma,-}:\F H^{\nu}(\R)\ni u_-\mapsto u(0)\in \mathcal F H^1(\R)$  exists. 
\item NLS \eqref{NLS} is invariant under the scaling 
$
(0,\infty)\ni \rho\mapsto u_\rho(t,x)=\rho^{\frac1\sigma}u(\rho^2 t,\rho x),
$  
that is, $u$ satisfies \eqref{NLS} if and only if $u_\rho$ does. Moreover, we have $
\|\F(f_\rho) \|_{L^\infty}=\rho^{\frac1\sigma-1}\|\widehat{f}\|_{L^\infty}
$. Thus, it could be possible to assume without loss of generality that $\|\widehat{u_+}\|_{L^\infty}$ is sufficiently small as long as $\sigma\neq 1$. However, this does not reduce at all the difficulty of the subcritical case $\sigma<1$. To explain the difficulty, we observe that $u-u_{\mathrm p,+}$ satisfies, at least formally, 
$$
i\partial_t (u-u_{\mathrm p,+})+\frac12\Delta (u-u_{\mathrm p,+})=\lambda |u|^{2\sigma} u-\lambda |u_{\mathrm p,+}|^{2\sigma}u_{\mathrm p,+}+\text{(error)}
$$
with some error term \text{(error)} depending only on $u_{\mathrm p}$ which we neglect in the following observation for simplicity. Since $|u_{\mathrm p,+}|\le t^{-1/2}\|\widehat{u_+}\|_{L^\infty}$, one can expect
\begin{align*}
\frac{d}{ds}\|u-u_{\mathrm p,+}\|_{L^2}^2
&\sim 2\lambda \Im\<|u|^{2\sigma} u-|u_{\mathrm p,+}|^{2\sigma}u_{\mathrm p,+},u-u_{\mathrm p,+}\>\\
&\lesssim t^{-\sigma}|\lambda|\|\widehat{u_+}\|_{L^\infty}^{2\sigma}\|u-u_{\mathrm p,+}\|^2+\|u-u_{\mathrm p,+}\|_{L^{4\sigma+4}}^{2\sigma+2},
\end{align*}
where the first term $ t^{-\sigma}\|\widehat{u_+}\|_{L^\infty}^{2\sigma}\|u-u_{\mathrm p,+}\|^2$ is expected to be the worst part regarding the decay rate in $t$ as we seek a solution satisfying $\|u-u_{\mathrm p,+}\|_{H^1}\lesssim t^{-\beta}$ and $\|u\|_{L^\infty}\lesssim t^{-1/2}$ as $t\to \infty$. If either that $\sigma>1$ or that $\sigma=1$ and $|\lambda|\|\widehat{u_+}\|_{L^\infty}^2$ is small enough, then it is in fact possible to construct  a unique global solution $u$ satisfying $\|u-u_{\mathrm p,+}\|_{L^2}\lesssim t^{-\beta}$ with some $\beta>0$ by using this formulation (\cite{Ozawa_1991}). We will give a brief explanation on the argument for the small-data critical case in the beginning of Section \ref{subsection_energy}. However, this  argument does not work for the subcritical case as the first term is no longer a small perturbation regardless of the size of $\|\widehat{u_+}\|_{L^\infty}$ since the decay rate $t^{-\sigma}$ is too weak. Overcoming this difficulty is the most novel aspect of the present paper.
\item It would be interesting to investigate whether a strategy similar to the one developed in the present paper can be applied to establish modified scattering for other subcritical nonlinear dispersive equations, in particular, for the NLS \eqref{NLS} in higher dimensions and for the Hartree  equation in $d=1$ with the nonlinearity $(|x|^{-\sigma}*|u|^2)u$ for some $\sigma<1$. 
\end{enumerate}
\end{remark}

\subsection{Pseudo-conformal transformation}
In what follows, we write for short
$$u_\mathrm p=u_{\mathrm p,+},\quad \varphi=\overline{\widehat{u_+}},\quad \|f\|=\|f\|_{L^2(\R)},\quad \<f,g\>=\int_\R f(x)\overline{g(x)}dx.$$
It is known that the Cauchy problem for \eqref{NLS} is globally well-posed in $L^2(\R)$ (\cite{Tsutsumi_1987}), and that if $e^{-it_0\Delta/2}u(t_1)\in \F H^1(\R)$ for some $t_1\in \R$, then $e^{-it\Delta/2}u\in C(\R,\mathcal FH^1(\R))$; see e.g. \cite[Proposition 2.2]{Masaki_2017} where a simple proof can be found. In order to prove Theorem \ref{theorem_1}, it is therefore sufficient to show the unique existence of the solution $u\in C([t_1,\infty),L^2(\R))$ to \eqref{NLS} with some fixed $t_1>0$ that satisfies $e^{-it\Delta/2}u\in C([t_1,\infty),\mathcal FH^1(\R))$ and \eqref{theorem_1_1} uniformly in $t\ge t_1$. 

One of the primary tools in our argument is the pseudo-conformal transformation. The pseudo-conformal transforms of $u$ and $u_{\mathrm p}$ are defined by
\begin{align}
\label{v}
v(s,x)&=\mathop{\mathrm{PC}}[u](s,x):=\overline{(\mathcal M(t)\mathcal D(t))^{-1}u(t,x)}\big|_{t=\frac1s}=(is)^{-\frac 12}e^{\frac{i|x|^2}{2s}}\overline u\left(\frac1s,\frac x s\right),
\\ 
\label{v_p}
v_{\mathrm p}(s,x)&=\mathop{\mathrm{PC}}[u_{\mathrm p}](s,x)
={\varphi(x)}\exp\left(-\frac{ i \lambda |\varphi(x)|^{2\sigma}s^{\sigma-1}}{\sigma-1}\right).
\end{align}
By virtue of the formula $i\partial_t+\frac12\Delta=e^{it\Delta/2}i\partial_t e^{-is\Delta/2}$, \eqref{Dollard} and a direct calculation, \eqref{NLS} and \eqref{NLS_w_p} for $t\ge t_1$ are transformed into new equations
\begin{align}
\label{NLS_v}
(i\partial_s+\frac12\Delta)v&=\lambda s^{\sigma-2}|v|^{2\sigma}v,\quad 0<s\le s_1,\\
\label{NLS_v_p}
i\partial_s v_{\mathrm p}&=\lambda s^{\sigma-2}|v_{\mathrm p}|^{2\sigma}v_{\mathrm p},\quad 0<s\le s_1,
\end{align}
respectively, where $s_1=t_1^{-1}$.  Moreover, we have the following properties: 

\begin{lemma}
\label{lemma_PC}
Suppose $f\in C([t_1,\infty),L^2(\R))$ and $e^{-it\Delta/2}f\in C([t_1,\infty),\F H^1(\R))$. Then
\begin{align}
\label{lemma_PC_1}
\|f(t)\|&=\|\mathop{\mathrm{PC}}[f](s)\|,\\
\label{lemma_PC_2}
\|xe^{-it\Delta/2}f(t)\|&=\|\nabla \mathop{\mathrm{PC}}[f](s)\|,\\
\label{lemma_PC_3}
\left\|\left|\varphi\left(\frac xt\right)\right|^{\sigma-1}\Re[\overline{u_{\mathrm p}(t)} f(t)]\right\|&=s^{1/2}\left\||\varphi|^{\sigma-1}\Re[\overline{v_{\mathrm p}(s)}\mathop{\mathrm{PC}}[f](s)\right\|,
\end{align}
where $s=t^{-1}$. 
In particular, $\mathop{\mathrm{PC}}[f]\in C((0,s_1],H^1(\R))$. 
\end{lemma}

\begin{proof}
\eqref{lemma_PC_1} follows by the unitarity of $\mathcal M(t)$ and $\mathcal D(t)$ on $L^2(\R)$. \eqref{lemma_PC_2} follows from \eqref{Dollard}  as
$$
\| xe^{-it\Delta/2}f(t)\|=\|x\mathcal M(t)^{-1}\mathcal F^{-1}\mathcal D(t)^{-1}\mathcal M(t)^{-1}f(t)\|=\|\nabla \mathop{\mathrm{PC}}[f](s)\|.
$$
To prove \eqref{lemma_PC_3}, we let $f(t)=\mathcal M(t)\mathcal D(t)\overline g(t)$ and $u_{\mathrm p}(t)=\mathcal M(t)\mathcal D(t)\overline{u_1}(t)$ for short. Since
$$
\overline{\mathcal M(t)}=e^{-\frac{i|x|^2}{2t}},\quad \overline{\mathcal D(t)h}=\overline{(it)^{-1/2}h(x/t)}=i^{1/2}t^{-1/2}\overline{h(x/t)}
$$
for any $h\in L^2(\R)$, we can write
\begin{align*}
&\left\|\left|\varphi\left(\frac xt\right)\right|^{\sigma-1}\Re[\overline{u_{\mathrm p}(t)}f(t)]\right\|\\
&=\left\|\left|\varphi\left(\frac xt\right)\right|^{\sigma-1}\Re\left[e^{-\frac{i|x|^2}{2t}}i^{1/2}t^{-1/2}u_1\left(\frac xt\right)e^{\frac{i|x|^2}{2t}}(it)^{-1/2}\overline g\left(\frac xt\right)\right]\right\|\\
&=t^{-1/2}\left\|\mathcal D(t)|\varphi|^{\sigma-1}\Re\left[u_1(t)\overline{g(t)}\right]\right\|
=s^{1/2}\left\||\varphi|^{\sigma-1}\Re\left[\overline{v_{\mathrm p}(s)}\mathop{\mathrm{PC}}[f](s)\right]\right\|.
\end{align*}
This completes the proof. 
\end{proof}

By \eqref{NLS_v} and this lemma, $u\in C([t_1,\infty),L^2(\R))$ satisfies $e^{-it\Delta/2}u\in C([t_1,\infty),\mathcal FH^1(\R))$ and $$u(t)=e^{i(t-t_1)\Delta/2}u(t_1)-i\int_{t_1}^te^{i(t-s)\Delta/2}\lambda |u(s)|^{2\sigma}u(s) ds,\quad t\ge t_1,$$ if and only if its pseudo-conformal transform $v\in C((0,s_1],H^1(\R))$ satisfies $$v(s)=e^{i(s-s_1)\Delta/2}v(s_1)-i\int_{s_1}^se^{i(s-r)\Delta/2}\lambda r^{\sigma-2} |v(r)|^{2\sigma}v(r) dr,\quad 0<s\le s_1.$$
Note that these Duhamel formulas are well-defined on $L^2(\R)$  under this setting  since, by using the estimate $\|e^{it\Delta/2}(1+|x|)^{-1}\|_{L^2\to L^\infty}\lesssim t^{-1/2}$ and the embedding $H^1(\R)\subset L^\infty(\R)$, we have $u\in C([t_1,\infty),L^\infty(\R))$ and $v\in C((0,s_1],L^\infty(\R))$ and hence $|u|^{2\sigma}u\in C_tL^2_x$ and $|v|^{2\sigma}v\in C_sL^2_x$. 

Now we introduce the following modified energy, which plays a crucial role in this paper: 
\begin{align}
\label{Q}
Q[f](s):=\left(\frac{\|\nabla f\|^2}{4}+ s^{-\sigma+\delta} \|f\|^2+\sigma\lambda  s^{\sigma-2}\||\varphi|^{\sigma-1}\Re[\overline {v_{\mathrm p}(s)}f]\|^2\right)^{1/2}.
\end{align}
Note that, thanks to the defocusing assumption $\lambda>0$, $Q[f]$ is positive and equivalent with the $H^1$-norm $\|f\|_{H^1}$ for each $s>0$. 
It follows from Lemma \ref{lemma_PC} that $u$ satisfies \eqref{theorem_1_1} if and only if 
\begin{align}
Q[v-v_{\mathrm p}](s)\lesssim s^{\beta-\delta},\quad 0<s\le s_1.
\label{theorem_v_1}
\end{align}
In conclusion, we see that Theorem \ref{theorem_1} follows from the following theorem for \eqref{NLS_v}:
\begin{theorem}	
\label{theorem_v}
Under the assumption in Theorem \ref{theorem_1}, there exists $s_1>0$ and a unique solution $v\in C((0,s_1],H^1(\R))$ to \eqref{NLS_v} satisfying \eqref{theorem_v_1}. 
\end{theorem}

The rest of the paper is devoted to the proof of this theorem. Before going into the details, we briefly outline the main idea. The existence and uniqueness of $v$ are shown separately. 

For the existence part, we consider a family $\{v_\ep\}_\ep$  of solutions to \eqref{NLS_v} defined for $0<\ep \le s\ll1$ with the initial datum $v_\ep(\ep,x)=v_{\mathrm p}(\ep,x)$ at $s=\ep$. The unique existence of $v_\ep$ for each $\ep$ is easily verified since we are away from the singularity at $s=0$. Then the main step is to establish \eqref{theorem_v_1} for $w_\ep:=v_{\ep}-v_{\mathrm p}$ uniformly in $\ep,s$. Once we have such a uniform energy estimate, the desired solution $v$ can be constructed by a standard energy method. 

The proof of the energy estimate for $w_\ep$ relies on two main ingredients. We linearize \eqref{NLS_v} around $v_{\mathrm p}$ to incorporate the leading part of the nonlinearity into the linear part as a linear potential term. Using the notation $\vv a=(a,\overline  a)^{\mathrm T}$, the linearization leads an integral equation for $\vv{w_\ep}$ of the form
$$
\vv{w_{\ep}}(s)=\vv{e_1}(s)+\mathcal U(s,\ep)\vv{e_1}(\ep)-\lambda \int_\ep^s  r^{\sigma-2}\mathcal  U(s,r)\left(\vv {(iG)}+\vv{(i e_2)}\right)(r)dr, 
$$
where $\mathcal U(s,r)$ is the propagator generated by the linearized equation, $G$ is the new error term and $e_j$ are some error terms arising from the linearization such that $G,e_j$ are less singular at $s=0$ than the original nonlinear term. Following our previous work \cite{KaMi} for the cubic case $\sigma=1$, we then show the following modified  energy estimate for $\mathcal U(s,r)$: 
$$
Q[\mathcal U(s,r)\vv{\psi_0}](s)\lesssim Q[\vv{\psi_0}](r),\quad\psi_0\in H^1(\R),\quad 0<\ep\le r<s\ll1,
$$
where $Q[\vv{f}]:=Q[f]$. It is worth emphasizing that the linearized potential is non-symmetric,  time-dependent and of order $s^{\rho-2}$ as $s\to 0$. In particular, it is of long-range type in the sense that $\mathcal U(s,r)$ cannot be approximated by the free propagator uniformly in $0<r\le s\ll1$. Therefore,  even the uniform-in-time $L^2$-boundedness of $\mathcal U(s,r)$ is highly non-trivial. Indeed, although the proof of this energy estimate for $\mathcal U(s,r)$ itself is a simple application of integration by parts (or a weighted energy method), the specific structure of the energy $Q[f]$ plays a crucial role in handling several terms coming from the linearized potential. 

The other main ingredient is the use of specific structures of the nonlinear term $G$ and the error terms  $e_1,e_2$ emerging after the linearization. The structures yield crucial cancellations when handling the polynomial growth in $s$ as $s\to 0$ of the derivatives of the nonlinear phase correction term $e^{i\gamma(1/s)|\varphi|^{2\sigma}}$ in the modified energy estimates for $Q[iG], Q[e_1]$ and $Q[ie_2]$. 
This step was not necessary in the critical case, as the growth rate is only logarithmic, and plays an important role to deal with the subcritical range ${2}/{\sqrt7}<\sigma<1$. This is the major technical difference from our previous work \cite{KaMi} on the critical case. 

For the proof of the uniqueness, setting $W=v_1-v_2$ with two solutions $v_1,v_2$ satisfying \eqref{theorem_v_1}, we essentially repeat the same argument as that of the existence part to derive a relevant integral equation for $W$ and then the following estimate
$$
Q[W](s)
\lesssim \int_{0}^s r^{\beta_1-1}Q[W](r)dr,\quad 0<s\ll1,
$$
with some $\beta_1>0$. This shows $W\equiv0$.

\subsection{Organization of the paper}
The proof of the existence part of Theorem \ref{theorem_v} is given in Section \ref{section_existence}. In Section \ref{subsection_energy}, we prepare the energy estimate for the linearized equation. The nonlinear and error estimates are given in Section \ref{subsection_nonlinear_estimate}. The energy estimate for $w_\ep$ is proved in Section \ref{subsection_existence_3} and we finish the proof of the existence part in Section \ref{subsection_existence_3}. The proof of the uniqueness part is given in Section \ref{section_uniqueness}. In Appendix \ref{appendix_A}, we provide proofs of two technical results. 


\section{Proof of Theorem \ref{theorem_v}: Global existence}
\label{section_existence}

This section is devoted to the proof of the existence part of Theorem \ref{theorem_v}. We start with the following Cauchy problem with the initial condition at $s=\ep>0$:
\begin{equation}
\label{NLS_v_ep}
\left\{\begin{aligned}
(i\partial_s+\frac12\Delta)v_\ep&=\lambda s^{\sigma-2}|v_\ep|^{2\sigma}v_\ep,\quad x\in \R,\quad \ep< s\le 1,\\
v_\ep(\ep,x)&=v_{\mathrm p}(\ep,x),\quad x\in \R.
\end{aligned}
\right.
\end{equation}
The following proposition is the main ingredient in the proof of Theorem \ref{theorem_1}. 

\begin{proposition}
\label{proposition_regularized_1}
For any $\varphi\in H^1(\R)$ and $0<\ep<1$, there exists a unique solution $v_\ep\in C([\ep,1],H^1(\R))$ to \eqref{NLS_v_ep}. Moreover, the following statements hold: 

\begin{itemize}
\item $\|v_\ep(s)\|=\|\varphi\|$ for all $s\in [\ep,1]$. 
\item Define
$
w_\ep:=v_\ep-v_{\mathrm p}
$. Under the assumption in Theorem \ref{theorem_1}, there exists $C,s_0>0$ depending only on $\sigma,\lambda,\delta$ and $\|\varphi\|_{H^\nu}$ such that
\begin{align}
Q[w_\ep](s)
&\le C s^{\beta-\delta},\label{proposition_regularized_1_1}
\\
\|w_\ep(s)-w_{\ep}(s')\|_{H^{-1}}
&\le C |s-s'|^{3\sigma/2-1}
\label{proposition_regularized_1_2}
\end{align}
\end{itemize}
for any $0<\ep\le s'\le s\le s_0$, where the modified energy $Q$ is defined by \eqref{Q}. 
\end{proposition}

The key point of this proposition lies in \eqref{proposition_regularized_1_1}; the (unconditional) unique existence of the solution $v_\ep$ and its mass conservation law for each $\ep$ follow from a standard argument since the nonlinearity is defocusing, $H^1$-subcritical and we are away from the singularity at $s=0$; see \cite{Ginibre_Velo_1978} and \cite[Theorem 4.11.1]{Cazenave}. Moreover,  \eqref{proposition_regularized_1_2} is easily deduced from \eqref{proposition_regularized_1_1}; see Section \ref{subsection_existence_3}. 

\subsection{Integral equation and energy estimate}
\label{subsection_energy}
The proof strategy of \eqref{proposition_regularized_1_1} builds upon a similar idea as that of the previous paper \cite{KaMi} for the cubic case $\sigma=1$. 
We start by deriving an integral equation for 
$w_{\ep}=v_\ep-v_{\mathrm p}$. Define
$
\mathcal R(s)=e^{is\Delta/2}-I.
$
Then $w_\ep$ satisfies
\begin{align}
\label{NLS_w_ep}
\left(i\partial_s+\frac12\Delta\right)\left(w_\ep-\mathcal R(s) v_{\mathrm p}\right)=\lambda  s^{\sigma-2}\left(|v_\ep|^{2\sigma}v_\ep-|\varphi|^{2\sigma}v_{\mathrm p}-\mathcal R(s)|\varphi|^{2\sigma}v_{\mathrm p}\right).
\end{align}
Indeed, it follows from \eqref{NLS_v_p} and the formula
$$
i\partial_s+\frac12\Delta=e^{is\Delta/2}i\partial_s e^{-is\Delta/2}=(I+\mathcal R(s))i\partial_s(I+\mathcal R(s))^{-1}
$$
that
\begin{align}
\nonumber
\left(i\partial_s+\frac12\Delta\right)(I+\mathcal R(s))v_{\mathrm p}
=(I+\mathcal R(s))i\partial_sv_{\mathrm p}=\lambda s^{\sigma-2}(I+\mathcal R(s))|\varphi|^{2\sigma}v_{\mathrm p}.
\end{align}
Thus \eqref{NLS_w_ep} follows. This formulation \eqref{NLS_w_ep} of the original NLS \eqref{NLS} (without the pseudo-conformal transform) was used by Hayashi--Naumkin \cite{Hayashi_Naumkin_2006} in the $d$-dimensional critical case $\sigma=1/d$. Here we briefly explain that this formulation is enough to prove the small-data modified scattering in the critical case. Since the difference of nonlinear terms satisfies
\begin{align}
\label{nonlinear}
\lambda  s^{\sigma-2}\left||v_\ep|^{2\sigma}v_\ep-|\varphi|^{2\sigma}v_{\mathrm p}\right|\lesssim \lambda  s^{\sigma-2}\left(|\varphi|^{2\sigma}|w_\ep|+|w_\ep|^{2\sigma+1}\right),
\end{align}
we have, e.g.  for the case $d=\sigma=1$ and $\alpha>1/2$, 
\begin{align*}
&\int_0^s |\lambda| r^{-1}\left\||\varphi|^{2}|w_\ep(r)|+|w_\ep(r)|^{3}\right\|_{H^\alpha}dr\\
&\lesssim s^\beta\left(|\lambda|\|\varphi\|_{H^\alpha}^2 \sup_{0<s\le 1}r^{-\beta}\|w_\ep(s)\|_{H^\alpha}+s^{2\sigma\beta}\sup_{0<s\le 1}r^{-(2\sigma+1)\beta}\|w_\ep(s)\|_{H^\alpha}^{2\sigma+1}\right)
\end{align*}
uniformly in $0\le \ep\le1$ and $\ep<s\le1$. 
 Moreover, the two terms related with $\mathcal R(s)$ in \eqref{NLS_w_ep} can be easily handled thanks to the decay estimate (see Lemma \ref{lemma_R}) $$\|\mathcal R(s)f\|_{\dot H^r}\lesssim s^{\alpha/2}\|f\|_{\dot H^{r+\alpha}},\quad 0\le \alpha\le2,\quad r\in \R,$$  and the fact that the divergence of the phase of $v_{\mathrm p}|_{\sigma=1}=\varphi e^{-i\lambda |\varphi|^2\log s}$ is only $|\log s|$.  Based on these observations, if $d=\sigma=1$, $1/2<\alpha_1<\alpha_2\le1$ and $|\lambda|\|\varphi\|_{H^{\alpha_1}}^2$ is sufficiently small, then one can solve 
\eqref{NLS_w_ep} in the energy space
$$
\{w_\ep\in C((0,1],H^{\alpha_2}(\R))\ |\ \sup_{0<s\le s_0}s^{-\beta}\|w_\ep(s)\|_{H^{\alpha_2}}<\infty\}
$$
for some $0<\beta<\alpha_2/2$ and sufficiently small $s_0$, even for the case $\ep=0$. This will provide a unique solution $u(t)$ to \eqref{NLS} that scatters to $u_{\mathrm p}(t)$ in $\F H^{\alpha_2}$ as $t\to \infty$. Note that this argument works for both defocusing and focusing cases. 

However, this strategy completely fails for the subcritical case $\sigma<1$ regardless of the size of $\|\varphi\|_{L^\infty}$, due to the strong singularity of $ s^{\sigma-2}|w_\ep|$ at $s=0$. To overcome this difficulty, we first extract the worst term of the RHS of \eqref{NLS_w_ep} by rewriting  the difference $|v_\ep|^{2\sigma}v_\ep-|\varphi|^{2\sigma}v_{\mathrm p}$ as
\begin{align}
\nonumber
|v_\ep|^{2\sigma}v_\ep-|\varphi|^{2\sigma}v_{\mathrm p}
\nonumber
=
(\sigma+1)|\varphi|^{2\sigma}w_\ep+\sigma|\varphi|^{2\sigma-2}v_{\mathrm p}^2\overline{w_\ep}+G(w_\ep,v_{\mathrm p}),
\end{align}
where
\begin{align}
\label{G_ep}
G(w_\ep,v_{\mathrm p})
= \int_0^1\left((\sigma+1)w_\ep(|w_{[\theta]}|^{2\sigma}-|v_{\mathrm p}|^{2\sigma})
+\sigma \overline{w_\ep} (|w_{[\theta]}|^{2\sigma-2}w_{[\theta]}^2-|v_{\mathrm p}|^{2\sigma-2}v_{\mathrm p}^2)\right)d\theta
\end{align}
and $w_{[\theta]}=v_{\mathrm p}+\theta w_\ep$. Note that $|v_{\mathrm p}|=|\varphi|$. This formula follows by applying Taylor's formula
$$
f(z_1)=f(z_0)+\int_0^1\left((z_1-z_0)f_z(z_0+\theta(z_1-z_0))+\overline{(z_1-z_0)}f_{\overline z}(z_0+\theta(z_1-z_0)\right)d\theta
$$
to $f(z)=|z|^{2\sigma}z$. For short we set
\begin{align}
e_1(s)&=\mathcal R(s) v_{\mathrm p}(s),\label{e1}\\
e_2(s)&=-\mathcal R(s)|\varphi|^{2\sigma}v_{\mathrm p}(s)+(\sigma+1)|\varphi|^{2\sigma}e_1(s)+\sigma|\varphi|^{2\sigma-2}v_{\mathrm p}(s)^2\overline{e_1(s)}\label{e2}.
\end{align}
Then \eqref{NLS_w_ep} is equivalent to
\begin{align}
\nonumber
&\left(i\partial_s +\frac 12\Delta\right)(w_{\ep}-e_1)\\
\label{NLS_v*}
&= \lambda  s^{\sigma-2}\left\{(\sigma+1) |\varphi|^{2\sigma}(w_{\ep}-e_1)+\sigma|\varphi|^{2\sigma-2}v_{\mathrm p}^2\overline{(w_{\ep}-e_1)}+G (w_{\ep},v_{\mathrm p})+e_2\right\},
\end{align}
where the first two terms of the RHS are most singular at $s=0$ as we will construct the solution $w_\ep$ satisfying $\|w_\ep\|\lesssim s^\beta$ as $s\to 0$ with some $\beta>0$ and the new nonlinear term satisfies $G(w_\ep,v_{\mathrm p})=O(|w_\ep|^2)$ at least. Then, the idea of our previous study \cite{KaMi} is that, instead of dealing with these first two terms as a nonlinear term, we incorporate them into the linear part as a linear potential term as follows. 
Define the time-dependent Hamiltonian $$
\mathcal H(s)=\mathcal H_0+\lambda  s^{\sigma-2}\mathcal  V(s),
$$ where
\begin{align}
\label{H}
\mathcal H_0=-\frac12\begin{pmatrix}\Delta&0\\0&-\Delta\end{pmatrix},\quad \mathcal  V(s,x)=\begin{pmatrix}(\sigma+1) |\varphi|^{2\sigma}&\sigma |\varphi|^{2\sigma-2}v_{\mathrm p}(s)^2\\-\sigma |\varphi|^{2\sigma-2}\overline{v_{\mathrm p}(s)}^2&-(\sigma+1) |\varphi|^{2\sigma}\end{pmatrix}. 
\end{align}
We set $\vv{w_{\ep}}=(w_{\ep},\overline{w_{\ep}})^{\mathrm T}$, $\vv{e_j}=(e_j,\overline{e_j})^{\mathrm T}$ and $\vv G (w_{\ep},v_{\mathrm p})=(G (w_{\ep},v_{\mathrm p}),\overline{G (w_{\ep},v_{\mathrm p})})^{\mathrm T}$, where $\bvec a^{\mathrm T}$ is the transposed vector of $\bvec a$. 
Then the system of \eqref{NLS_v*} and its complex conjugate is written as
\begin{align}
\left(\partial_s +i\mathcal H(s)\right)\left(\vv{w_{\ep}}-\vv{e_1}\right)=-\lambda  s^{\sigma-2}\left(\vv{(i G)} (w_{\ep},v_{\mathrm p})+\vv{(ie_2)}\right). 
\end{align}
Since $w_\ep(\ep)=v_{\ep}(\ep)-v_{\mathrm p}(\ep)=0$ we arrive at the following Duhamel formula:
\begin{align}
\label{IE_1}
\vv{w_{\ep}}(s)-\vv{e_1}(s)=\mathcal U(s,\ep)\vv{e_1}(\ep)-\lambda \int_\ep^s  r^{\sigma-2}\mathcal  U(s,r)\left(\vv {(iG)} (w_{\ep},v_{\mathrm p})+\vv{(i e_2)}\right)(r)dr
\end{align}
for $0<\ep\le s\le1$, where $\mathcal  U(s,r)$ is the linear propagator generated by $\mathcal H(s)$, namely, the solution to the linearized system 
\begin{equation}
\label{linearized_equation}
\left\{\begin{aligned}
&\left(i\partial_s -\mathcal H(s)\right)\bvec \Psi(s,x)=0,\quad 0< s\le1,\ x\in \R,\\
&\bvec\Psi(r,x)=\bvec \Psi_0(x),\quad x\in \R,
\end{aligned}
\right.
\end{equation}
is given by 
$
\bvec \Psi(s,x)=\mathcal  U(s,r)\bvec \Psi_0(x)
$. 
Basic properties for $\mathcal  U(s,r)$ are summarized as follows. Let $\mathbb B(X)$ denote the space of bounded operators on $X$. 
\begin{lemma}
\label{lemma_propagator}
Suppose $\sigma>0$ and $\varphi\in H^1(\R)$. Then there exists a unique propagator 
$$\mathcal  U(s,r):(0,1]\ni s,r\mapsto \mathbb B(H^1(\R)^2)$$ generated by $\mathcal H(s)$ satisfying the following properties:
\begin{itemize}
\item[(1)] $\mathcal  U(s,r)=\mathcal  U (s,s')\mathcal  U (s',s)$ and $\mathcal  U (s,s)=I$ for all $0< r,s,s' \le1$. 
\item[(2)] $(0,1]^2\ni (s,r)\mapsto \mathcal  U(s,r)$ is strongly continuous on $H^1(\R)^2$. 
\item[(3)] For all $\bvec \Psi_0\in  H^1(\R)^2$ and $s,r\in (0,1]$, $\bvec\Psi(s,x)=\mathcal  U(s,r)\bvec\Psi_0(x)$ solves \eqref{linearized_equation} in $H^{-1}(\R)^2$. 
\item[(4)] If $\bvec \Psi_0=\vv{\psi_0}$ with $\psi_0\in H^1(\R)$, then $\mathcal  U(s,r)\vv{\psi_0}=\vv \psi(s)$, where $\psi\in C((0,1],H^1(\R))$ satisfies 
\begin{equation}
\label{psi}
\left\{\begin{aligned}
&i\partial_s\psi +\frac12\Delta \psi=\lambda  s^{\sigma-2}\left(\sigma |\varphi|^{2\sigma}\psi +\sigma |\varphi|^{2\sigma-2}v_{\mathrm p}^2\overline{\psi}\right),\quad 0< s\le1,\\
&\psi(r)=\psi_0.
\end{aligned}
\right.
\end{equation}
\end{itemize}
\end{lemma}

\begin{proof}
The proof is essentially the same as that of \cite[Lemma 2.2]{KaMi}, so we give its brief outline only. 
Observe that $v_{\mathrm p}= \varphi e^{-\frac{i\lambda}{\sigma-1}|\varphi|^{2\sigma} s^{\sigma-1}}$ satisfies
\begin{align*}
|\varphi|=|\varphi|,\quad |\nabla v_{\mathrm p}|\lesssim (1+ s^{\sigma-1}|\varphi|^{2\sigma})|\nabla\varphi|,
\end{align*}
which implies
\begin{align}
\nonumber
\|\mathcal  V(s)\bvec \Psi_0\|_{H^1}
&\lesssim \|\mathcal  V(s)\|_{L^\infty}\|\bvec \Psi_0\|_{H^1}+\|\nabla \mathcal  V(s)\|\|\bvec \Psi_0\|_{L^\infty}\\
&\lesssim (1+ s^{\sigma-1})\|\bvec \Psi_0\|_{H^1}. 
\label{lemma_propagator_proof_2}
\end{align}
With this bound and the unitarity of $e^{-it\mathcal H_0}$ on $H^1(\R)^2$ at hand, one can construct a unique solution $\mathcal  U(s,r)\bvec \Psi_0$ to \eqref{linearized_equation} by solving the Duhamel integral equation
\begin{align}
\label{lemma_propagator_proof_3}
\mathcal  U(s,r)\bvec \Psi_0=e^{-i(s-r)\mathcal H_0}\bvec \Psi_0-i\lambda \int_r^s \tau^{\sigma-2}e^{-i(s-\tau)\mathcal H_0}\mathcal  V(\tau,x)\mathcal  U(\tau,r)\bvec \Psi_0d\tau
\end{align}
in $C((0,1],H^1(\R)^2)$ via the standard successive approximation. The items (1)--(4) then follow from this Duhamel formula and the uniqueness of the solution to \eqref{linearized_equation}. 
\end{proof}

It follows from \eqref{lemma_propagator_proof_2}, \eqref{lemma_propagator_proof_3} and Gronwall's lemma that there exists $C>0$ such that
$$
\|\mathcal  U(s,r)\bvec \Psi_0\|_{H^1(\R)^2}\le C e^{C\{ r^{\sigma-1}- s^{\sigma-1}\}}\|\bvec \Psi_0\|_{{H^1(\R)^2}}
$$
for all $0<r\le s\le1$. This estimate is too rough to obtain uniform estimates for $\vv w_\ep$. Instead, we  prove the following modified energy estimate. In what follows, we denote 
$
Q[\vv f]=Q[f].
$

\begin{proposition}
\label{proposition_energy}
Let $\lambda>0$, $1\le \nu\le 3/2$, $4/(3+2\nu)<\sigma\le1$, $\varphi\in H^1(\R)$ and $\nabla\varphi\in L^{\frac{2}{3-2\nu}}(\R)$. Suppose $\delta>0$ is so small that $(3+2\nu)\sigma-4-(2\nu-1)\delta>0$.  Then, there exists $s_0>0$ such that
\begin{align}
\label{proposition_energy_estimate_1}
Q\big[\mathcal  U(s,r)\vv{\psi_0}\big](s)\lesssim Q[\psi_0](r),\quad \vv{\psi_0}=(\psi_0,\overline{\psi_0})^{\mathrm T}
\end{align}
uniformly in  $\psi_0\in H^1(\R)$ and $0<r\le s\le s_0$. In particular, for sufficiently small $\delta>0$, \eqref{proposition_energy_estimate_1} holds for $4/5<\sigma \le1$ if $\varphi\in H^1(\R)$ and for $2/3<\sigma\le4/5$ if in addition $\nabla \varphi \in L^\infty(\R)$. 
\end{proposition}

\begin{remark}
If $\varphi\in H^\nu(\R)$ for $\nu<3/2$ or $\varphi\in B^{3/2}_{2,1}(\R)$ for $\nu=3/2$, then $\nabla \varphi\in L^{\frac{2}{3-2\nu}}(\R)$ by the Sobolev embeddings $H^{\nu-1}(\R)\subset L^{\frac{2}{3-2\nu}}(\R)$ for $\nu<3/2$ and $B^{1/2}_{2,1}(\R)\subset L^\infty(\R)$. Here $B^{s}_{p,q}(\R)$ denotes the inhomogeneous Besov space. It is also known that $H^{\nu}(\R)\subset B^{3/2}_{2,1}(\R)$ if $\nu>3/2$. 
\end{remark}


\begin{proof}[Proof of Proposition \ref{proposition_energy}]
By Lemma \ref{lemma_propagator} (4), it suffices to prove
\begin{align}
\label{proposition_energy_proof_0}
Q[\psi(s)](s)\lesssim Q[\psi_0](r),
\end{align}
where  $\psi\in C((0,1],H^1(\R))$ is the solution to \eqref{psi}. 
The key ingredient is the following identity: 
\begin{align}
\nonumber
\frac{d}{ds}\Big( Q[\psi](s)\Big)^2&=-(\sigma-\delta)  s^{-\sigma-1+\delta}\|\psi\|^2-(2-\sigma)\sigma\lambda  s^{\sigma-3} \||\varphi|^{\sigma-1}\Re[\overline{v_{\mathrm p}}\psi]\|^2\\
\nonumber
&\quad -4\sigma\lambda  s^{-2+\delta}\<|\varphi|^{2\sigma-2}\Re[\overline{v_{\mathrm p}}\psi],\Im[\overline{v_{\mathrm p}}\psi]\>\\
\label{eq_E_ep}
&\quad-\sigma \lambda  s^{\sigma-2}\Im\<\nabla \psi,|\varphi|^{2\sigma-2}\Re[\overline{\varphi}\nabla\varphi]\psi\>.
\end{align}
To prove this identity, we observe that the first equation of \eqref{psi} is written in the form
\begin{align}
\label{proposition_energy_proof_2}
i\partial_s\psi +\frac12\Delta\psi
=\lambda  s^{\sigma-2}\left\{|\varphi|^{2\sigma}\psi+2\sigma  |\varphi|^{2\sigma-2}\Re[\overline{v_{\mathrm p}}\psi]v_{\mathrm p}\right\}.
\end{align}
Then \eqref{eq_E_ep} is obtained by the combination of the following three identities: 
\begin{align}
&\frac14\frac{d}{ds}\|\nabla \psi\|^2+\lambda  s^{\sigma-2} \frac{d}{ds}\left(\frac12\||\varphi|^{\sigma}\psi\|^2+\sigma\left\||\varphi|^{\sigma-1}\Re[\overline{v_{\mathrm p}}\psi]\right\|^2\right)\nonumber\\
&\quad =-2\sigma\lambda^2  s^{2\sigma-4} \<|\varphi|^{4\sigma-2}\Re[\overline{v_{\mathrm p}}\psi],\Im[\overline{v_{\mathrm p}}\psi]\>,\label{eq_E_ep_1}\\
&\frac{d}{ds}( s^{-\sigma+\delta}\|\psi\|^2)
=-(\sigma-\delta)  s^{-\sigma-1+\delta}\|\psi\|^2-4\sigma\lambda  s^{-2+\delta}\<|\varphi|^{2\sigma-2}\Re[\overline{v_{\mathrm p}}\psi],\Im[\overline{v_{\mathrm p}}\psi]\>\label{eq_E_ep_1_1},\\
&\frac{\lambda s^{\sigma-2}}{2}\frac{d}{ds}\||\varphi|^\sigma \psi\|^2
=\sigma \lambda s^{\sigma-2}\Im\<\nabla \psi,|\varphi|^{2\sigma-2}\Re[\overline{\varphi}\nabla\varphi]\psi\>\nonumber\\
&\quad\qquad\qquad\qquad\qquad-2\sigma\lambda^2 s^{2\sigma-4}\<|\varphi|^{4\sigma-2}\Re[\overline{v_{\mathrm p}}\psi],\Im[\overline{v_{\mathrm p}}\psi]\>\label{eq_E_ep_1_2}.
\end{align}
Indeed, \eqref{eq_E_ep} is obtained by calculating both sides of \eqref{eq_E_ep_1}+\eqref{eq_E_ep_1_1}-\eqref{eq_E_ep_1_2} as follows: 
\begin{align*}
&\text{LHS of }\eqref{eq_E_ep_1}+\eqref{eq_E_ep_1_1}-\eqref{eq_E_ep_1_2}\\
&=\frac14\frac{d}{ds}\|\nabla \psi\|^2
+\sigma\lambda  s^{\sigma-2} \frac{d}{ds}\left\||\varphi|^{\sigma-1}\Re[\overline{v_{\mathrm p}}\psi]\right\|^2
+\frac{d}{ds}( s^{-\sigma+\delta}\|\psi\|^2)\\
&=\frac{d}{ds}\Big( Q[\psi](s)\Big)^2-(\sigma-2)\sigma\lambda s^{\sigma-3}\left\||\varphi|^{\sigma-1}\Re[\overline{v_{\mathrm p}}\psi]\right\|^2
\end{align*}
and
\begin{align*}
&\text{RHS of }\eqref{eq_E_ep_1}+\eqref{eq_E_ep_1_1}-\eqref{eq_E_ep_1_2}\\
&=
-(\sigma-\delta)  s^{-\sigma-1+\delta}\|\psi\|^2-4\sigma\lambda  s^{-2+\delta}\<|\varphi|^{2\sigma-2}\Re[\overline{v_{\mathrm p}}\psi],\Im[\overline{v_{\mathrm p}}\psi]\>\\
&\quad -\sigma \lambda s^{\sigma-2}\Im\<\nabla \psi,|\varphi|^{2\sigma-2}\Re[\overline{\varphi}\nabla\varphi]\psi\>.
\end{align*}
 Identities \eqref{eq_E_ep_1_1} and \eqref{eq_E_ep_1_2} are obtained by direct calculations. Precisely, we have 
\begin{align*}
\frac{d}{ds}( s^{-\sigma+\delta}\|\psi\|^2)
=-(\sigma-\delta)  s^{-\sigma-1+\delta}\|\psi\|^2
+2s^{-\sigma+\delta}\Re\<\partial_s\psi,\psi\>,
\end{align*}
where $\<\cdot,\cdot\>$ is identified with the duality coupling $\<\cdot,\cdot\>_{H^{-1},H^1}$ and \eqref{proposition_energy_proof_2} implies
\begin{align*}
\Re\<\partial_s\psi,\psi\>
&=2\sigma \lambda s^{\sigma-2}\Im\<|\varphi|^{2\sigma-2}\Re[\overline{v_{\mathrm p}}\psi]v_{\mathrm p},\psi\>\\
&=-2\sigma \lambda s^{\sigma-2}\<|\varphi|^{2\sigma-2}\Re[\overline{v_{\mathrm p}}\psi],\Im[\overline{v_{\mathrm p}}\psi]\>.
\end{align*}
Similarly, 
\begin{align*}
&\frac{\lambda s^{\sigma-2}}{2}\frac{d}{ds}\||\varphi|^\sigma \psi\|^2\\
&=\frac{\lambda s^{\sigma-2}}{2}\left(-\Im\<\Delta \psi,|\varphi|^{2\sigma}\psi\>
+4\sigma \lambda s^{\sigma-2}\Im\<|\varphi|^{2\sigma-2}\Re[\overline{v_{\mathrm p}}\psi]v_{\mathrm p},|\varphi|^{2\sigma}\psi\>\right)\\
&=\sigma\lambda s^{\sigma-2} \Im\<\nabla \psi,|\varphi|^{2\sigma-2}\Re[\overline{\varphi}\nabla \varphi]\psi\>
-2\sigma \lambda^2 s^{2\sigma-4}\<|\varphi|^{4\sigma-2}\Re[\overline{v_{\mathrm p}}\psi],\Im[\overline{v_{\mathrm p}}\psi]\>.
\end{align*}
On the other hand, \eqref{eq_E_ep_1} follows at least formally by calculating both sides of $$\Re \<\text{LHS of }\eqref{proposition_energy_proof_2},\partial_s \psi\>=\Re \<\text{RHS of }\eqref{proposition_energy_proof_2},\partial_s \psi\>.$$
Indeed, we obtain
\begin{align*}
&\Re \<i\partial_s\psi +\frac12\Delta\psi,\partial_s \psi\>
=-\frac14\frac{d}{ds}\left\|\nabla \psi\right\|^2
\end{align*}
and, by using \eqref{NLS_v_p}, 
\begin{align*}
&\Re\< |\varphi|^{2\sigma}\psi+2\sigma |\varphi|^{2\sigma-2}\Re[\overline{v_{\mathrm p}}\psi],\partial_s\psi\>\\
&=\frac{1}{2}\frac{d}{ds}\||\varphi|^{\sigma}\psi\|^2+2\sigma \Re\< |\varphi|^{2\sigma-2}\Re[\overline{v_{\mathrm p}}\psi],\partial_s(\overline{v_{\mathrm p}}\psi)\>\\
&\quad -2\sigma\Re\< |\varphi|^{2\sigma-2}\Re[\overline{v_{\mathrm p}}\psi],i\lambda s^{\sigma-2}|\varphi|^{2\sigma}\overline{v_{\mathrm p}}\psi \>\\
&=\frac{1}{2}\frac{d}{ds}\||\varphi|^{\sigma}\psi\|^2
+\sigma \frac{d}{ds}\||\varphi|^{\sigma-1}\Re[\overline{v_{\mathrm p}}\psi]\|^2
+2\sigma\lambda s^{\sigma-2}\< |\varphi|^{4\sigma-2}\Re[\overline{v_{\mathrm p}}\psi],\Im[\overline{v_{\mathrm p}}\psi] \>. 
\end{align*}
Identity \eqref{eq_E_ep_1} thus follows. 
Note however that the term $\<i\partial_s\psi +\frac12\Delta\psi,\partial_s \psi\>$ does not make sense if $\psi\in C((0,1].H^1(\R))\cap C^1((0,1],H^{-1}(\R))$ only. To be more precise, if we set $\psi^m=(I-m^{-1}\Delta)^{-1}\psi$, then \eqref{eq_E_ep_1} is obtained by calculating both sides of
\begin{align}
\Re \<\text{LHS of }\eqref{proposition_energy_proof_2},\partial_s \psi^m\>
&=\Re \<\text{RHS of }\eqref{proposition_energy_proof_2},\partial_s \psi^m\>,\label{eq_E_ep_2}
\end{align}
and  taking $m\to \infty$. This is possible since $(I-m^{-1}\Delta)^{-1}$ commutes with $i\partial_s+\frac12\Delta$. Since the computation is rather involved, we give the technical details in Appendix \ref{appendix_A} below.

Now we prove \eqref{proposition_energy_proof_0}. Note that the first two terms of the RHS of \eqref{eq_E_ep} are negative and thus good terms. Other two terms  are not sign definite, but can be controlled by the first two terms as follows. For the third term, taking $s_0>0$ further small if necessary, we obtain
\begin{align}
\nonumber
&\left|-4\sigma\lambda  s^{-2+\delta}\<|\varphi|^{2\sigma-2}\Re[\overline{v_{\mathrm p}}\psi],\Im[\overline{v_{\mathrm p}}\psi]\>\right|\\
\nonumber
&\le 4\sigma\lambda  s^{\frac{\delta}{2}}\|\varphi\|_{L^\infty}^{\sigma}
\cdot  s^{-\frac{\sigma+1-\delta}{2}}\|\psi\|
\cdot  s^{\frac{\sigma-3}{2}}\|\varphi|^{\sigma-1}\Re[\overline{v_{\mathrm p}}\psi]\|\\
\label{proposition_energy_estimate_1_proof_4}
&\le \frac{(\sigma-\delta)  s^{-\sigma-1+\delta}\|\psi\|^2}{2}+\frac{(2-\sigma)\sigma\lambda  s^{\sigma-3} \||\varphi|^{\sigma-1}\Re[\overline{v_{\mathrm p}}\psi]\|^2}{2}
\end{align}
for all $s\in (0,s_0]$ with sufficiently small $s_0=s_0(\sigma,\lambda,\delta,\varphi)$. 
To deal with the last term of \eqref{eq_E_ep}, we use Gagliardo--Nirenberg's  inequality
\begin{align}
\|f\|_{L^{\frac{1}{\nu-1}}(\R)}\lesssim \|\nabla f\|_{L^2(\R)}^{\frac32-\nu}\|f\|_{L^2(\R)}^{\nu-\frac12},\quad 1\le \nu\le 3/2,
\label{GN}
\end{align}
and Young's inequality \begin{align}
ab\le \frac{a^p}{p}+\frac{b^q}{q}
\label{Young}
\end{align} with $p=\frac{2}{5/2-\nu}$, $q=\frac{2}{\nu-1/2}$, $a=\sqrt{m}^{-(\nu-1/2)}\|\nabla \psi\|^{5/2-\nu}$ and $b=(\sqrt{m}\|\psi\|)^{\nu-1/2}$ to obtain\begin{align*}|\Im\<\nabla \psi,|\varphi|^{2\sigma-2}\Re[\overline{\varphi}\nabla\varphi]\psi\>|&\lesssim \|\nabla \psi\|_{L^2}\|\varphi\|_{L^\infty}^{2\sigma-1}\|\nabla \varphi\|_{L^{\frac{2}{3-2\nu}}}\|\psi\|_{L^{\frac{1}{\nu-1}}}\\&\lesssim \|\nabla \psi\|^{\frac52-\nu}\|\psi\|^{\nu-\frac12}\\&\le C_0 (m^{-\frac{\nu-1/2}{5/2-\nu}}\|\nabla \psi\|^2+m\|\psi\|^2)\end{align*}with some $C_0>0$ uniformly in $m>0$. By choosing $$m=\frac{(\sigma-\delta) s^{1-2\sigma+\delta}}{2C_0 \sigma\lambda}$$so that $C_0\sigma\lambda m s^{\sigma-2}=(\sigma-\delta) s^{-\sigma-1+\delta}/2$ and $\sigma\lambda C_0 s^{\sigma-2}m^{-\frac{\nu-1/2}{5/2-\nu}}\sim  s^{\frac{(3+2\nu)\sigma-4-(2\nu-1)\delta}{5-2\nu}-1}$, we have
\begin{align}
\nonumber
&|\sigma \lambda  s^{\sigma-2}\Im\<\nabla \psi,|\varphi|^{2\sigma-2}\Re[\overline{\varphi}\nabla\varphi]\psi\>|\\
\label{proposition_energy_estimate_1_proof_5}
&\le C s^{\frac{(3+2\nu)\sigma-4-(2\nu-1)\delta}{5-2\nu}-1}\|\nabla \psi\|^2+\frac{(\sigma-\delta)  s^{-\sigma-1+\delta}\|\psi\|^2}{2}
\end{align}
with some $C>0$. It follows from \eqref{eq_E_ep}, \eqref{proposition_energy_estimate_1_proof_4} and \eqref{proposition_energy_estimate_1_proof_5} that
\begin{align*}
\frac{d}{ds}\Big(Q[\psi](s)\Big)^2
\lesssim  s^{\frac{(3+2\nu)\sigma-4-(2\nu-1)\delta}{5-2\nu}-1}\|\nabla \psi\|^2
\lesssim  s^{\frac{(3+2\nu)\sigma-4-(2\nu-1)\delta}{5-2\nu}-1}Q[\psi](s)^2.
\end{align*}
Assuming $(3+2\nu)\sigma-4-(2\nu-1)\delta>0$, we obtain \eqref{proposition_energy_proof_0} as follows: 
$$Q[\psi](s)^2\le Q[\psi_0](r)^2\exp\left(C \int_r^s \tau^{\frac{(3+2\nu)\sigma-4-(2\nu-1)\delta}{5-2\nu}-1}d\tau\right)\lesssim Q[\psi_0](r)^2$$uniformly in $0< r<s\le s_0$. 
\end{proof}


\subsection{Nonlinear and error estimates}
\label{subsection_nonlinear_estimate}
In this section we prove the modified energy estimates for the nonlinear term $G(w_{\ep},v_{\mathrm p})$ and error terms $e_1,e_2$. We start with the nonlinear term. 
\begin{proposition}
\label{proposition_nonlinear}
Under the same conditions on $\sigma,\nu$ and $\varphi$ in Proposition \ref{proposition_energy}, we have 
\begin{align}
Q[iG(w_\ep,v_{\mathrm p})](s)
\nonumber
&\lesssim s^{\frac{\sigma^2}{2}+\sigma-1-\frac{(\sigma+1)\delta}{2}}Q[w_\ep](s)^{2\sigma+1}+ s^{\frac{\sigma}{4}-\frac{\delta}{4}}Q[w_\ep](s)^2\\
&\quad+ s^{\frac{\sigma^2}{2}+\frac{(\nu-1)\sigma}{2}-\frac{(\sigma-\nu+1)\delta}{2}}Q[w_\ep](s)^{2\sigma}
\label{proposition_nonlinear_2}
\end{align}
uniformly in $0<\ep\le s\le s_0$ provided $s_0=s_0(\sigma,\nu,\varphi)$ is sufficiently small.  
\end{proposition}

To prove this proposition, the following modification of the energy $Q$ plays an important role: 
\begin{align}
\label{widetilde_Q}
\widetilde Q[f]=\left(\frac14 \|\nabla(e^{-i\gamma \Phi}f)\|^2+s^{-\sigma+\delta} \|f\|^2+\sigma\lambda s^{\sigma-2}\||\varphi|^{\sigma-1}\Re[\overline{v_{\mathrm p}} f]\|^2\right)^{1/2}
\end{align}
where $\Phi=|\varphi|^{2\sigma}$. 
\begin{lemma}
\label{lemma_nonlinear}
Under the same conditions on $\sigma,\nu$ and $\varphi$ in Proposition \ref{proposition_energy}, we have 
$$\widetilde Q[f](s)\sim Q[f](s)$$ uniformly in $s\in (0,s_0]$. 
\end{lemma}

\begin{proof}
We compute $\nabla\Phi =2\sigma |\varphi|^{2\sigma-2}\Re[\overline{\varphi}\nabla\varphi]$ and 
$$
\|\nabla (e^{-i\gamma \Phi }f)\|^2=\|\nabla f-i\gamma f \nabla \Phi\|^2=\|\nabla f\|^2+2\Re\<\nabla f,-i\gamma f\nabla \Phi\>+\gamma^2\|f \nabla \Phi \|^2.
$$
It follows from the same computations as in the proof of \eqref{proposition_energy_estimate_1_proof_5} that
\begin{align*}
|\<\nabla f,f \nabla \Phi \>|
\lesssim m^{-\frac{\nu-1/2}{5/2-\nu}}\|\nabla f\|^2+m\|f\|^2
\end{align*}
for all $m>0$. Choosing $m=s^{1-2\sigma+\delta_1}$ with some $\delta_1>\delta$ sufficiently close to $\delta$ so that $\gamma m\sim s^{-\sigma+\delta_1}$ and $\gamma m^{-\frac{\nu-1/2}{5/2-\nu}}\sim s^{\alpha}$ with $\alpha=\frac{(3+2\nu)\sigma-4+(2\nu-1)\delta_1}{5-2\nu}>0$ and $s_0$ is small enough, we have
\begin{align*}
2\gamma |\<\nabla f, f \nabla \Phi\>|
\le C s^{\alpha} \|\nabla f\|^2+s^{-\sigma+\delta_1}\|f\|^2
\le \frac14(\|\nabla f\|^2+s^{-\sigma+\delta}\|f\|^2)
\end{align*}
for any $s\in (0,s_0]$. Note that $s_0$ can be taken uniformly in $s,f$. Similarly, we use \eqref{GN} and \eqref{Young} with $p=\frac{2}{3-2\nu}$, $q=\frac{2}{2\nu-1}$, $a=\sqrt{m}^{-(2\nu-1)}\|\nabla f \|^{3-2\nu}$ and $b=(\sqrt{m}\|f\|)^{2\nu-1}$ to obtain
\begin{align*}
\|f\|_{L^{\frac{1}{\nu-1}}}^2
\le \|\nabla f\|^{3-2\nu}\|f\|^{2\nu-1}\lesssim m^{-\frac{2\nu-1}{3-2\nu}}\|\nabla f\|^2+m\|f\|^2.
\end{align*}
Choosing $m=s^{2-3\sigma+\delta_1}$ and $\delta_1>\delta$ sufficiently close to $\delta$ so that $\gamma^2m\sim s^{-\sigma+\delta_1}$ and $\gamma^2m^{-\frac{2\nu-1}{3-2\nu}}\sim s^\alpha$, we have
\begin{align*}
\gamma^2\|f \nabla \Phi \|^2
&\lesssim \gamma^2\|\varphi\|_{L^\infty}^{2\sigma-1}\|\nabla\varphi\|_{L^{\frac{2}{3-2\nu}}}^2\|f\|_{L^{\frac{1}{\nu-1}}}^2\\
&\le Cs^{\alpha} \|\nabla f\|^2+s^{-\sigma+\delta_1}\|f|^2
\le \frac14(\|\nabla f\|^2+s^{-\sigma+\delta}\|f\|^2).
\end{align*}
Hence, $\widetilde Q[f](s)\lesssim  Q[f](s)$. Conversely, if we set $g=e^{-i\gamma\Phi}f$, then
 $$\|\nabla f\|^2=\|\nabla (e^{i\gamma\Phi }g)\|^2=\|\nabla g\|^2+2\Re\<\nabla g,i\gamma g\nabla \Phi\>+\gamma^2\|g\nabla \Phi\|^2.$$
By the same argument with $f$ replaced by $g$ then yields $ Q[f](s)\lesssim  \widetilde Q[f](s)$ for $s\in (0,s_0]$. 
\end{proof}

\begin{proof}[Proof of Proposition \ref{proposition_nonlinear}]
Note that $Q$ can be replaced by $\widetilde Q$ by the above lemma. The following elementary inequalities are used frequently: 
\begin{align}
\label{proposition_nonlinear_proof_1}
\big||z_1|^{2\sigma-k} z_1^k-|z_2|^{2\sigma-k} z_2^k\big|
&\lesssim |z_1-z_2|^{2\sigma}+|z_2|^{2\sigma-1}|z_1-z_2|,\\
\big||z_1|^{2\sigma-k-1} z_1^k-|z_2|^{2\sigma-k-1} z_2^k\big|
\label{proposition_nonlinear_proof_2}
&\lesssim |z_1-z_2|^{2\sigma-1}
\end{align}
for $z_1,z_2\in \C$ and $k=0,1,2,...$. We also use the following estimates: 
\begin{align*}
\|w_\ep\|&
\lesssim  s^{\sigma/2-\delta/2}\widetilde Q[w_\ep](s),\\
\|\nabla (e^{-i\gamma \Phi}w_\ep)\|&\lesssim \widetilde Q[w_\ep](s),\\
\|w_\ep\|_{L^\infty}&\lesssim \|w_\ep\|^{1/2}\|\nabla (e^{-i\gamma \Phi}w_\ep)\|^{1/2}\lesssim  s^{\sigma/4-\delta/4 }\widetilde Q[w_\ep](s),\\
\||\varphi|^{\sigma-1}\Re[\overline{v_{\mathrm p}}w_\ep]\|&\lesssim  s^{1-\sigma/2}\widetilde Q[w_\ep](s).
\end{align*}
These follow from the definition of the energy norm $\widetilde Q[w]$. Let $G=G(w_\ep,v_{\mathrm p})$ for short. To prove \eqref{proposition_nonlinear_2}, we need to estimate three quantities $$ s^{-\sigma/2+\delta/2}\|G\|,\quad  s^{\sigma/2-1}\||\varphi|^{\sigma-1}\Im[\overline{v_{\mathrm p}}G]\|,\quad\|\nabla (e^{-i\gamma \Phi}G)\|.$$ 
Recall that $G$ is given by $\int_0^1G_{1,\theta}d\theta,
$
where 
$$
G_{1,\theta}= (\sigma+1)w_\ep\left(|w_{[\theta]}|^{2\sigma}-|\varphi|^{2\sigma}\right)
+\sigma \overline{w_\ep}\left(|w_{[\theta]}|^{2\sigma-2}w_{[\theta]}^2-|\varphi|^{2\sigma-2}v_{\mathrm p}^2\right)
$$
and $w_{[\theta]}=v_{\mathrm p}+\theta w_\ep$. We start with dealing with $\|G\|$. It follows from \eqref{proposition_nonlinear_proof_1} and $|v_{\mathrm p}|=|\varphi|$ that 
\begin{align*}
\left||w_{[\theta]}|^{2\sigma-k}w_{[\theta]}^k-|\varphi|^{2\sigma-k}v_{\mathrm p}^k\right|
&\lesssim |w_{[\theta]}-v_{\mathrm p}|^{2\sigma}+|\varphi|^{2\sigma-1}|w_{[\theta]}-v_{\mathrm p}|\\
&\lesssim |w_\ep|^{2\sigma}+\|\varphi\|_{L^\infty}^{2\sigma-1}|w_\ep|
\end{align*}
and hence
$
|G_{1,\theta}|\lesssim |w_\ep|^{2\sigma+1}+|w_\ep|^2
$ uniformly in $\theta\in [0,1]$. Then
\begin{align}
\nonumber
 s^{-\sigma/2+\delta/2}\|G\|
&\lesssim   s^{-\sigma/2+\delta/2}\|w_\ep\|(\|w_\ep\|_{L^\infty}^{2\sigma}+\|w_\ep\|_{L^\infty})\\
&\lesssim  s^{\sigma^2/2-\sigma\delta/2}\widetilde Q[w_\ep](s)^{2\sigma+1}+ s^{\sigma/4-\delta/4}\widetilde Q[w_\ep](s)^{2}.
\label{proposition_nonlinear_proof_3}
\end{align}

For the term $|\varphi|^{\sigma-1}\Im[\overline{v_{\mathrm p}}G]$, the same argument implies a rough bound $$|\varphi|^{\sigma-1}\Im[\overline{v_{\mathrm p}}G]\lesssim |\varphi|^{2\sigma-1}|w_\ep|^2++|\varphi||w_\ep|^{2\sigma+1}. $$
However, this is not enough to prove \eqref{proposition_nonlinear_2}. Instead, we utilize the structure of $\Im[\overline{v_{\mathrm p}}G]$ to show \begin{align}
\label{proposition_nonlinear_proof_4_0}
|\Im[\overline{v_{\mathrm p}}G_{1,\theta}]|\lesssim |\varphi|^{2\sigma-1}\left|\Re[\overline{v_{\mathrm p}}w_\ep]\right||w_\ep|+|\varphi||w_\ep|^{2\sigma+1},
\end{align}
which implies the desired estimate as 
\begin{align}
\nonumber
& s^{\sigma/2-1} \||\varphi|^{\sigma-1}\Im[\overline{v_{\mathrm p}}G]\|\\
\nonumber
&\lesssim  s^{\sigma/2-1}\|\varphi\|_{L^\infty}^{2\sigma-1}\left(\||\varphi|^{\sigma-1}\Re[\overline{v_{\mathrm p}}w_\ep]\|\|w_\ep\|_{L^\infty}
+\|w_\ep\|\|w_\ep\|_{L^\infty}^{2\sigma}\right)\\
&\lesssim   s^{\sigma/4-\delta/4}\widetilde Q[w_\ep](s)^2+ s^{\sigma^2/2+\sigma-1-(\sigma+1)\delta/2}\widetilde Q[w_\ep](s)^{2\sigma+1}.
\label{proposition_nonlinear_proof_4}
\end{align}
In order to show \eqref{proposition_nonlinear_proof_4_0}, we first calculate 
\begin{align*}
\partial_z|z|^{2\sigma}&=\sigma |z|^{2\sigma-2}\overline z,\quad  
\partial_{\overline{z}}|z|^{2\sigma}=\sigma |z|^{2\sigma-2} z,\\
\partial_z(|z|^{2\sigma-2}z^2)&=(\sigma+1) |z|^{2\sigma-2}z,\quad
\partial_{\overline z}(|z|^{2\sigma-2}z^2)=(\sigma-1) |z|^{2\sigma-4}z^3,
\end{align*}
and use the Taylor expansions of $|w_{[\theta]}|^{2\sigma}$ and $|w_{[\theta]}|^{2\sigma-2}w_{[\theta]}^2$ around $v_{\mathrm p}$ to observe
\begin{align*}
\overline{v_{\mathrm p}}G_{1,\theta}
&=(\sigma+1)  \overline{v_{\mathrm p}}w_\ep\left\{\sigma|\varphi|^{2\sigma-2}\overline{v_{\mathrm p}}\theta w_\ep
+\sigma|\varphi|^{2\sigma-2}v_{\mathrm p}\theta \overline{w_\ep}\right\}\\
&\quad +\sigma \overline{v_{\mathrm p}}\overline{w_\ep}\left\{(\sigma+1) |\varphi|^{2\sigma-2}v_{\mathrm p}\theta w_\ep+(\sigma-1)|\varphi|^{2\sigma-4}v_{\mathrm p}^3\theta \overline{w_\ep}\right\}+\overline{v_{\mathrm p}}G_{2,\theta}\\
&=\theta (\sigma+1)\sigma |\varphi|^{2\sigma-2}\overline{v_{\mathrm p}}^2w_\ep^2
+2\theta (\sigma+1)\sigma|\varphi|^{2\sigma}|w_\ep|^2\\
&\quad +\theta \sigma(\sigma-1)|\varphi|^{2\sigma-2}v_{\mathrm p}^2\overline{w_\ep}^2
+\overline{v_{\mathrm p}}G_{2,\theta}
\end{align*}
where, setting $w_{[\theta,\rho]}=v_{\mathrm p}+\rho (w_{[\theta]}-v_{\mathrm p})=v_{\mathrm p}+\rho\theta w_\ep$ for short, $G_{2,\theta}$ is given by
\begin{align*}
G_{2,\theta}
&=\theta (\sigma+1)\sigma w_\ep^2 \int_0^1 \left(|w_{[\theta,\rho]}|^{2\sigma-2}\overline{w_{[\theta,\rho]}}-|\varphi|^{2\sigma-2}\overline{v_{\mathrm p}}\right)d\rho\\
&\quad+2\theta (\sigma+1)\sigma |w_\ep|^2 \int_0^1 \left(|w_{[\theta,\rho]}|^{2\sigma-2}w_{[\theta,\rho]}-|\varphi|^{2\sigma-2}v_{\mathrm p}\right)d\rho
\\
&\quad+\theta \sigma(\sigma-1) \overline{w_\ep}^2 \int_0^1 \left(|w_{[\theta,\rho]}|^{2\sigma-4}w_{[\theta,\rho]}^3-|\varphi|^{2\sigma-4}v_{\mathrm p}^3\right)d\rho. 
\end{align*}
Since $\Im \overline z=-\Im z$ and $\Im[z^2]=2\Re z\Im z$, we have 
\begin{align*}
\Im[\overline{v_{\mathrm p}}G_{1,\theta}]
=4\theta\sigma |\varphi|^{2\sigma-2}\Re[\overline{v_{\mathrm p}}w_\ep]\Im[\overline{v_{\mathrm p}}w_\ep]+\Im[\overline{v_{\mathrm p}}G_{2,\theta}].
\end{align*}
Moreover, \eqref{proposition_nonlinear_proof_2} implies
$
|G_{2,\theta}|\lesssim |w_\ep|^{2\sigma+1}
$ 
and thus \eqref{proposition_nonlinear_proof_4_0} follows. 

It remains to deal with $\|\nabla G\|$. To this end, we first calculate, for any $h\in H^1(\R)$ 
\begin{align}
\nabla (|h|^{2\sigma-2}h^{2})
\label{proposition_nonlinear_proof_6}
&=(\sigma+1) |h|^{2\sigma-2}h\nabla h+(\sigma-1)|h|^{2\sigma-4} h^3\overline{\nabla h}.
\end{align}
For short we set $\widetilde w_\ep=e^{-i\gamma \Phi}w_\ep$. Then 
$$
e^{-i\gamma \Phi}G_{1,\theta}=(\sigma+1)\widetilde w_\ep\left(|\widetilde w_{[\theta]}|^{2\sigma}-|\varphi|^{2\sigma}\right)
+\sigma \overline{\widetilde w_\ep}\left(|\widetilde w_{[\theta]}|^{2\sigma-2}\widetilde w_{[\theta]}^2-|\varphi|^{2\sigma-2}\varphi^2\right),
$$
where 
$\nabla \{\overline {\widetilde w_\ep}(|\widetilde w_{[\theta]}|^{2\sigma-2}\widetilde w_{[\theta]}^2-|\varphi|^{2\sigma-2}\varphi^2)\}$ is given by a linear combination of seven terms
\begin{align*}
&(\overline{\nabla  \widetilde w_\ep})\left(|\widetilde w_{[\theta]}|^{2\sigma-2}w_{[\theta]}^{2}-|\varphi|^{2\sigma-2}\varphi^{2}\right),\\
&\overline{\widetilde w_\ep}\left(|\widetilde w_{[\theta]}|^{2\sigma-2}\widetilde w_{[\theta]}-|\varphi|^{2\sigma-2}\varphi\right)\nabla \varphi,\\
&\overline{\widetilde w_\ep}\left(|\widetilde w_{[\theta]}|^{2\sigma-4}\widetilde w_{[\theta]}^3-|\varphi|^{2\sigma-4}\varphi^3\right)\overline{\nabla \varphi},\\
&\overline{\widetilde w_\ep}\left(|\widetilde w_{[\theta]}|^{2\sigma-2}\widetilde w_{[\theta]}-|\varphi|^{2\sigma-2}\varphi\right)\nabla (\widetilde w_{[\theta]}-\varphi),\\
&\overline{\widetilde w_\ep}\left(|\widetilde w_{[\theta]}|^{2\sigma-4}\widetilde w_{[\theta]}^3-|\varphi|^{2\sigma-4}\varphi^3\right)\overline{\nabla (\widetilde w_{[\theta]}-\varphi)},\\
&\overline{\widetilde w_\ep}|\varphi|^{2\sigma-2}\varphi\nabla (\widetilde w_{[\theta]}-\varphi),\\
&\overline{\widetilde w_\ep}|\varphi|^{2\sigma-4}\varphi^3\overline{\nabla (\widetilde w_{[\theta]}-\varphi)}.
\end{align*}
Applying \eqref{proposition_nonlinear_proof_1}, \eqref{proposition_nonlinear_proof_2} to these terms shows
\begin{align*}
&|\nabla \{\overline {\widetilde w_\ep}(|\widetilde w_{[\theta]}|^{2\sigma-2}\widetilde w_{[\theta]}^2-|\varphi|^{2\sigma-2}\varphi^2)\}|\\
&\lesssim |\nabla \widetilde w_\ep|\left(|w_\ep|^{2\sigma}+\|\varphi\|_{L^\infty}^{2\sigma-1}|w_\ep|\right)+|w_\ep|^{2\sigma}|\nabla \varphi|
+|w_\ep|^{2\sigma}|\nabla w_\ep|+\|\varphi\|_{L^\infty}^{2\sigma-1}|w_\ep||\nabla w_\ep|\\
&\lesssim |\nabla \widetilde w_\ep|(|w_\ep|^{2\sigma}+|w_\ep|)+ |\nabla \varphi||w_\ep|^{2\sigma}.
\end{align*}
By a similar argument, we also obtain the same estimate for $|\nabla \{ {\widetilde w_\ep}(|\widetilde w_{[\theta]}|^{2\sigma}-|\varphi|^{2\sigma})\}|$. 
Therefore, 
\begin{align}
\nonumber
\|\nabla   ( e^{-i\gamma \Phi}G)\|
&\lesssim \|\nabla \widetilde w_\ep\|\left(\|w_\ep\|_{L^\infty}^{2\sigma}+\|w_\ep\|_{L^\infty}\right)
+ \left\||\nabla\varphi||w_\ep|^{2\sigma}\right\|
\end{align}
By H\"older's inequality $L^{\frac{2}{3-2\nu}}(\R)\cdot L^{\frac{1}{\nu-1}}(\R)\subset L^2(\R)$, Sobolev's embedding $H^{s_\nu}(\R)\subset L^{\frac{2\sigma}{\nu-1}}(\R)$ with $s_\nu=1/2-(\nu-1)/(2\sigma)=(\sigma-\nu+1)/(2\sigma)\in (0,1)$ and the interpolation, we obtain
\begin{align}
\left\||\nabla\varphi||w_\ep|^{2\sigma}\right\|
&\le \|\nabla \varphi\|_{L^{\frac{2}{3-2\nu}}}\|w_\ep\|_{L^{\frac{2\sigma}{\nu-1}}}^{2\sigma}\nonumber\\
&\lesssim \left(\|w_\ep\|^{1-s_\nu}\|w_\ep\|_{H^1}^{s_\nu}\right)^{2\sigma}\nonumber\\
&\lesssim s^{(\sigma+\nu-1)(\sigma-\delta)/2}\widetilde Q[w_\ep](s)^{2\sigma}.
\label{proposition_nonlinear_proof_7}
\end{align}
Thus
\begin{align*}
\|\nabla (  e^{-i\gamma \Phi}G)\|
\lesssim  s^{\sigma^2/2-\sigma\delta/2}\widetilde Q[w_\ep](s)^{2\sigma+1}+ s^{ \sigma/4-\delta/4}
\widetilde Q[w_\ep](s)^2+s^{(\sigma+\nu-1)(\sigma-\delta)/2}\widetilde Q[w_\ep](s)^{2\sigma}.
\end{align*}
Since $\sigma/2^2-\sigma\delta/2>\sigma^2/2+\sigma-1-(\sigma+1)\delta/2$ for $0<\sigma<1$ and $\delta>0$, this estimate, combined with \eqref{proposition_nonlinear_proof_3} and \eqref{proposition_nonlinear_proof_4}, imply \eqref{proposition_nonlinear_2}. 
This completes the proof. 
\end{proof}

\begin{remark}It seems to be difficult to improve the bound \eqref{proposition_nonlinear_proof_4_0}.  Indeed, at least for the critical case $\sigma=1$, $\Im[\overline{v_{\mathrm p}}G]$ can be written explicitly by$$\Im[\overline{v_{\mathrm p}}G]=2\Re[\overline{v_{\mathrm p}}w_\ep]\Im[\overline{v_{\mathrm p}}w_\ep]+|w_\ep|^2\Im[\overline{v_{\mathrm p}}w_\ep]. $$Therefore, $\Im[\overline{v_{\mathrm p}}G]$ necessarily includes the contribution that cannot be expressed in the form $\Re[\overline{v_{\mathrm p}}w_\ep]f(|w_\ep|^2)$. As a consequence, it seems to be also difficult to  improve \eqref{proposition_nonlinear_2}. 
\end{remark}


We next state the energy estimates for the error terms $e_1,e_2$ given by \eqref{e1} and \eqref{e2}. 

\begin{proposition}
\label{proposition_e}
Let $3/4<\sigma<1$, $0< s\le1$ and $\varphi\in H^{\nu}(\R)$. Then
\begin{align}
\label{proposition_e_1}
Q[e_1](s)+Q[ie_2](s)
\lesssim s^{\nu(\sigma-1/2)+\sigma/2-1}
\end{align}
for $1\le \nu<2$. Moreover, we have
\begin{align}
\label{proposition_e_2}
Q[e_1](s)&\lesssim s^{3\sigma/2-1}
\end{align}
provided $\nu=2$, and \begin{align}
\label{proposition_e_3}
Q[ie_2](s)&\lesssim s^{\nu(\sigma-1/2)-\sigma/2}
\end{align}
provided $2\le \nu\le \sigma/(\sigma-1/2)$. 
\end{proposition}

\begin{remark}In particular, if $\varphi\in H^{\sigma/(\sigma-1/2)}(\R)$ then $Q[ie_2](s)\lesssim s^{\sigma/2}$. It will be seen in the proof that decay rate $s^{\sigma/2}$ in this bound and $s^{3\sigma/2-1}$ in \eqref{proposition_e_2} are optimal. 
\end{remark}

For the proof of this proposition, we prepare a few preliminary lemmas. 

\begin{lemma}
\label{lemma_R}
For any $r\in \R$, $0\le \nu\le2$ and $s\ge0$, 
\begin{align*}
\|\mathcal R(s)f\|_{\dot H^r}&\lesssim s^{\nu/2} \|f\|_{\dot H^{r+\nu}},\\ 
\|(\mathcal R(s)-is\Delta/2)f\|_{\dot H^r}&\lesssim s^{1+\nu/2} \|f\|_{\dot H^{2+r+\nu}}.
\end{align*}
\end{lemma}

\begin{proof}
Since $\mathcal R(s)=\F^{-1}(e^{-is|\xi|^2/2}-1)\F$, the lemma follows from the unitarity of $\F$ in $\dot H^r$ and the estimates
$|e^{-ia}-1|\lesssim a^{\nu/2}$ and $|e^{-ia}-1-ia|\lesssim a^{1+\nu/2}$
for $a\ge0$ and $0\le \nu\le2$. \end{proof}

\begin{lemma}
\label{lemma_v_p}
Let $0\le \nu\le1$, $\gamma\ge1$, $\Phi \in H^{1}(\R)$ be real-valued and $f\in H^{\nu}(\R)$. Then
\begin{align*}
\|e^{i\gamma\Phi }f\|_{H^{\nu}}\lesssim \gamma^\nu(1+\|\Phi \|_{H^{1}})^\nu\|f\|_{H^{\nu}}. 
\end{align*}
\end{lemma}

\begin{proof}
Identifying $e^{i \gamma\Phi }$ with the multiplication operator $f\mapsto e^{i \gamma\Phi }f$, the lemma follows by interpolating between the following two bounds: $\|e^{i \gamma\Phi }f\|=\|f\|$ and
\begin{align*}
\|e^{i \gamma\Phi }f\|_{H^1}
\lesssim \|f\|+ \gamma\|\nabla\Phi \|\|f\|_{L^\infty}+\|\nabla f\|
\lesssim \gamma\left(1+\|\Phi \|_{H^{1}}\right)\|f\|_{H^1}.
\end{align*}
\end{proof}

\begin{lemma}
\label{lemma_Holder}
Let $m$ be an integer, $F(z)=|z|^{\alpha-m}z^m$ and $f\in H^\beta(\R)$. Then 
$$
\|F(\varphi)f\|_{H^\beta}\lesssim \begin{cases}\|\varphi\|_{H^1}^\alpha \|f\|_{H^\beta}&\text{if}\quad 0\le \beta\le1<\alpha,\ \varphi\in H^1(\R),\\\|\varphi\|_{H^{3/2}}^\alpha \|f\|_{H^\beta}&\text{if}\quad 0\le \beta<\alpha<1,\ \varphi\in H^{3/2}(\R).\end{cases}
$$
\end{lemma}

\begin{proof}
We first observe $\|F(\varphi)f\|\lesssim \|\varphi\|_{H^1}^{\alpha}\|f\|$ by the embedding $H^1(\R)\subset L^\infty(\R)$. When $\alpha>1$ and $\beta>1/2$ in which case $F\in C^1(\C)$ and $H^\beta(\R)$ is algebra, we have
\begin{align*}
\|F(\varphi)f\|_{\dot H^\beta}
\lesssim \|F'(\varphi)\|_{L^\infty}\|\varphi\|_{H^\beta}\|f\|_{H^\beta}
&\lesssim \|\varphi\|_{H^\beta}^\alpha\|f\|_{H^\beta},
\end{align*}
where  $F'=(\partial_zF,\partial_{\overline z}F)$. In order to prove $F(\varphi)f\in \dot H^\beta(\R)$ for $\alpha>1$ and $\beta\le1/2$ or $0\le \beta<\alpha<1$, we first employ the fractional Leibniz rule by Kenig--Ponce--Vega  \cite{KPV} to observe
$$
\|F(\varphi)f\|_{\dot H^\beta}\lesssim \|f|\nabla|^\beta F(\varphi)\|_{L^2}+\|F(\varphi)|\nabla|^\beta f\|_{L^2}+\||\nabla|^{\beta-\beta_1}F(\varphi)\|_{L^{p_1}}\||\nabla|^{\beta_1}f\|_{L^{p_2}},
$$
where $1<p_1,p_2<\infty$, $1/2=1/p_1+1/p_2$ and $0<\beta_1<\beta$. Let $p>\max\{2,1/\beta\}$ be specified later. For the first term $\|f|\nabla|^\beta F(\varphi)\|_{L^2}$, H\"older's inequality $L^p(\R)\cdot L^{\frac{2p}{p-2}}(\R)\subset L^2(\R)$ and Sobolev's embedding $H^\beta(\R) \subset L^{\frac{2p}{p-2}}(\R)$ imply
\begin{align*}
\|f|\nabla|^\beta F(\varphi)\|_{L^2}\le \|f\|_{L^{\frac{2p}{p-2}}}\||\nabla|^\beta F(\varphi)\|_{L^p}\lesssim \|f\|_{H^\beta} \||\nabla|^\beta F(\varphi)\|_{L^p}.
\end{align*}
The second term $\|F(\varphi)|\nabla|^\beta f\|_{L^2}$ is simply dominated by
$ \|\varphi\|_{H^1}^{\alpha}\|f\|_{H^\beta}$. 
For the third term, if we take $p_1>p$ close to $p$ so that $\beta_1:=1/p-1/p_1\in (0,\beta)$, then
$$
\||\nabla|^{\beta-\beta_1}F(\varphi)\|_{L^{p_1}}\||\nabla|^{\beta_1}f\|_{L^{p_2}}\lesssim \||\nabla|^\beta F(\varphi)\|_{L^p}\||\nabla|^{\beta_1+1/p_1}f\|_{L^2}\lesssim \||\nabla|^\beta F(\varphi)\|_{L^p}\|f\|_{H^\beta}
$$
by the embeddings $\dot W^{\beta_1,p}(\R)\subset L^{p_1}(\R)$ and $\dot H^{1/p_1}(\R)\subset L^{p_2}(\R)$. 

It remains to deal with $\||\nabla|^\beta F(\varphi)\|_{L^p}$. When $\alpha>1$ and $\beta\le1/2$, we use the fractional chain rule by Christ--Weinstein \cite[Proposition 3.1]{ChWe} to obtain
$$
\||\nabla|^\beta F(\varphi)\|_{L^p}\lesssim \|F'(\varphi)\|_{L^{q_1}}\||\nabla|^\beta \varphi\|_{L^{q_2}}\lesssim \|\varphi\|_{L^{(\alpha-1)q_1}}^{\alpha-1}\||\nabla|^\beta \varphi\|_{L^{q_2}},
$$
where $1/p=1/q_1+1/q_2$ and $1<q_1,q_2<\infty$. If $(\alpha-1)q_1>2$, then
$$
\|\varphi\|_{L^{(\alpha-1)q_1}}^{\alpha-1}\||\nabla|^\beta \varphi\|_{L^{q_2}}\lesssim \|\varphi\|_{H^{1/2}}^{\alpha-1}\|\varphi\|_{H^{\beta+1/2}}\lesssim \|\varphi\|_{H^1}^\alpha.
$$
On the other hand, when $0\le \beta<\alpha<1$,  the fractional chain rule for H\"older continuous  functions by Visan \cite[Proposition A.1]{Visan} implies
$$
\||\nabla|^\beta F(\varphi)\|_{L^p}\lesssim \|\varphi\|_{L^{(\alpha-\beta/s)q_1}}^{\alpha-\beta/s}\||\nabla|^s \varphi\|_{L^{{\beta q_2}/{s}}}^{\beta/s}
$$
provided $\beta/\alpha<s<1$, $1<p<\infty$, $1/p=1/q_1+1/q_2$ and $(1-\frac{\beta}{\alpha s})q_1>1$. Fixing such an $s$, we can choose  $p>\max\{2,1/\beta\}$ sufficiently large so that $(\alpha-{\beta}/{ s})q_1>2$ and $\beta q_2/s>2$. Then
$$
\|\varphi\|_{L^{(\alpha-\beta/s)q_1}}^{\alpha-\beta/s}\||\nabla|^s \varphi\|_{L^{{\beta q_2}/{s}}}^{\beta/s}\lesssim \|\varphi\|_{H^{1/2}}^{\alpha-\beta/s}\|\varphi\|_{H^{s+1/2}}^{\beta/s}\lesssim \|\varphi\|_{H^{3/2}}^{\alpha}.
$$
This completes the proof.
\end{proof}

\begin{proof}[Proof of Proposition \ref{proposition_e}]In the following argument, $C_\nu$ denotes constants depending on $\|\varphi\|_{H^\nu}$ but not on $\ep,s$, which may vary line to line. 
We denote for short $$\Phi=|\varphi|^{2\sigma},\quad \gamma=\gamma(1/s)=\frac{\lambda  s^{\sigma-1}}{1-\sigma},\quad \mathcal R_1=\mathcal R-\frac{is}{2}\Delta=e^{is\Delta/2}-I-\frac{is}{2}\Delta.$$
The following  formulas are used frequently in the proof:
\begin{align}
e^{i\gamma\Phi }f&=\nabla^{-1}e^{i\gamma\Phi }\left\{i\gamma f \nabla \Phi+\nabla f\right\}\label{e_1},\\
&=\nabla^{-4}e^{i\gamma\Phi}\left\{\gamma^4 f(\nabla\Phi)^4-i\gamma^3\nabla\left( f(\nabla\Phi)^3\right)\right\}\nonumber\\
&\quad -\gamma^2\nabla^{-3}e^{i\gamma\Phi}\nabla\left(f(\nabla\Phi)^2\right)\nonumber\\
&\quad +\nabla^{-2}e^{i\gamma\Phi}\left\{2i\gamma (\nabla f)\nabla \Phi+i\gamma f\Delta\Phi+\Delta f\right\},
\label{e_2}\\
\nabla(e^{i\gamma\Phi }f)&=e^{i\gamma\Phi }\left\{i\gamma f \nabla \Phi+\nabla f\right\}\label{e_3},\\
&=\nabla^{-2}e^{i\gamma\Phi}\left\{-i\gamma^3 f(\nabla\Phi)^3-\gamma^2\nabla\left( f(\nabla\Phi)^2\right)\right\}\nonumber\\
&\quad +\nabla^{-1}e^{i\gamma\Phi}\left\{2i\gamma (\nabla f)\nabla \Phi+i\gamma f\Delta\Phi+\Delta f\right\}.
\label{e_4}
\end{align}
\eqref{e_1} is easily verified. Then
\begin{align*}
e^{i\gamma\Phi}f=\nabla^{-2}e^{i\gamma\Phi}\{-\gamma^2f(\nabla\Phi)^2+2i\gamma (\nabla f)\nabla \Phi+i\gamma f\Delta \Phi+\Delta f\},
\end{align*}
where the term $-\gamma^2e^{i\gamma\Phi}f(\nabla\Phi)^2$ is further rewritten as
\begin{align*}
-\gamma^2e^{i\gamma\Phi}f(\nabla\Phi)^2
&=\nabla^{-1}e^{i\gamma\Phi}\left\{-i\gamma^3 f(\nabla\Phi)^3-\gamma^2\nabla\left(f(\nabla\Phi)^2\right)\right\}\\
&=\nabla^{-2}e^{i\gamma\Phi}\left\{\gamma^4f(\nabla\Phi)^4-i\gamma^3 \nabla\left(f(\nabla\Phi)^3\right)\right\}-\gamma^2 \nabla^{-1}e^{i\gamma\Phi}\nabla\left(f(\nabla\Phi)^2\right).
\end{align*}Hence \eqref{e_2} follows. \eqref{e_3} and \eqref{e_4} can be verified similarly. Note that $\mathcal R(s) \nabla^{-j}$ for $0\le j\le 2$ and $\mathcal R_1(s) \nabla^{-k}$ for $2\le k\le 4$ are bounded on $L^2(\R)$ thanks to Lemma \ref{lemma_R}, so the singularities of the inverse power $\nabla^{-j}$ at the zero frequency appeared in these formulas do not cause a problem. 

Now we prove \eqref{proposition_e_1} for $1\le \nu<2$. Lemma \ref{lemma_Holder} with $\alpha=2\sigma>1$ and $\beta=\nu-1<1$ implies
$$
\|\varphi\nabla\Phi \|_{H^{\nu-1}}=2\sigma\||\varphi|^{2\sigma-2}\varphi\Re[\overline{\varphi}\nabla\varphi]\|_{H^{\nu-1}}\lesssim \|\varphi\|_{H^\nu}^{2\sigma+1}.
$$
Recalling that $e_1=\mathcal R v_{\mathrm p}=\mathcal R e^{i\gamma\Phi }\varphi$, \eqref{e_1} with $f=\varphi$ and Lemmas \ref{lemma_R} and \ref{lemma_v_p} then imply
\begin{align}
\|e_1\|
\lesssim s^{\nu/2}\left(s^{\sigma-1}\|e^{i\gamma \Phi}\varphi\nabla \Phi\|_{\dot H^{\nu-1}}+\|e^{i\gamma \Phi}\nabla \varphi\|_{\dot H^{\nu-1}}\right)
\le C_\nu s^{\nu(\sigma-1/2)}.
\label{e_5}
\end{align}
Similarly, it follows from \eqref{e_3} and Lemmas \ref{lemma_R} and \ref{lemma_v_p} that
\begin{align*}
\|\nabla e_1\|\lesssim s^{(\nu-1)/2}\left(s^{\sigma-1}\|e^{i\gamma \Phi}\varphi\nabla \Phi\|_{\dot H^{\nu-1}}+\|e^{i\gamma \Phi}\nabla \varphi\|_{\dot H^{\nu-1}}\right)
\le C_\nu s^{\nu(\sigma-1/2)-1/2}. 
\end{align*}
Since $0<\sigma<1$, these estimates imply \eqref{proposition_e_1} for $e_1$ as
\begin{align*}
Q[e_1](s)
&\lesssim \|\nabla e_1\|+ s^{-\sigma/2+\delta/2}\|e_1\|+ s^{\sigma/2-1}\|\varphi\|_{L^\infty}^\sigma\|e_1\|\\
&\le C_\nu s^{\nu(\sigma-1/2)}(s^{-1/2}+s^{-\sigma/2+\delta/2}
+s^{\sigma/2-1}
)\\
&\le C_\nu s^{\nu(\sigma-1/2)+\sigma/2-1}.
\end{align*}
To deal with $e_2$, we first rewrite \eqref{e2} as
\begin{align}
\nonumber
e_2&=- \mathcal R |\varphi|^{2\sigma}v_{\mathrm p}+(\sigma+1)|\varphi|^{2\sigma}e_1+\sigma |\varphi|^{2\sigma-2}v_{\mathrm p}^2\overline{e_1}\\
&=\left(\Phi \mathcal R-\mathcal R\Phi \right)v_{\mathrm p}+2\sigma |\varphi|^{2\sigma-2}v_{\mathrm p}\Re\left[\overline{v_{\mathrm p}}\mathcal Rv_{\mathrm p}\right]
\label{e_6}
\end{align}
and compute
\begin{align}
\nonumber
\nabla e_2
&=\left\{(\nabla \Phi )\mathcal R-\mathcal R(\nabla \Phi )\right\}v_{\mathrm p}+\left(\Phi \mathcal R-\mathcal R\Phi \right)\nabla v_{\mathrm p}
+2\sigma \nabla(|\varphi|^{2\sigma-2}v_{\mathrm p})\Re\left[\overline{v_{\mathrm p}}\mathcal Rv_{\mathrm p}\right]\\
&\quad +2\sigma |\varphi|^{2\sigma-2}v_{\mathrm p}\Re\left[\overline{\nabla v_{\mathrm p}}\mathcal Rv_{\mathrm p}\right]
+2\sigma |\varphi|^{2\sigma-2}v_{\mathrm p}\Re\left[\overline{v_{\mathrm p}}\mathcal R\nabla v_{\mathrm p}\right]. 
\label{e_7}
\end{align}
Since $\Phi \varphi\in H^{\nu}(\R)$ and $\varphi\nabla\Phi ,\varphi\Phi \nabla\Phi ,\Phi \nabla\varphi \in H^{\nu-1}(\R)$ by Lemma \ref{lemma_Holder}, the same argument as above based on Lemmas \ref{lemma_R} and \ref{lemma_v_p} implies
\begin{align*}
\|\mathcal R\Phi v_{\mathrm p}\|
&\le C_\nu   s^{\nu(\sigma-1/2)},\\
\|\mathcal R(\nabla \Phi )v_{\mathrm p}\|
&\lesssim s^{(\nu-1)/2}\|e^{i\gamma \Phi }\varphi\nabla \Phi \|_{\dot H^{\nu-1}}
\le C_\nu   s^{(\nu-1)(\sigma-1/2)},\\
\|\mathcal R\nabla v_{\mathrm p}\|
&\lesssim s^{(\nu-1)/2}\left(s^{\sigma-1}\|e^{i\gamma \Phi }\varphi\nabla \Phi \|_{\dot H^{\nu-1}}+\|e^{i\gamma \Phi }\nabla \varphi\|_{\dot H^{\nu-1}}\right)
\le C_\nu   s^{\nu(\sigma-1/2)-1/2},\\
\|\mathcal R\Phi \nabla v_{\mathrm p}\|
&\lesssim s^{(\nu-1)/2}\left(s^{\sigma-1}\|e^{i\gamma \Phi }\varphi|\nabla \Phi |^2\|_{\dot H^{\nu-1}}+\|e^{i\gamma \Phi }\Phi \nabla \varphi\|_{\dot H^{\nu-1}}\right)
\le C_\nu   s^{\nu(\sigma-1/2)-1/2},\\
\|\mathcal Rv_{\mathrm p}\|_{L^\infty}
&\lesssim \|\mathcal Rv_{\mathrm p}\|^{1/2}\|\mathcal R\nabla v_{\mathrm p}\|^{1/2}\le C_\nu   s^{\nu(\sigma-1/2)-1/4}.
\end{align*}
Hence
\begin{align*}
\|e_2\|
&\lesssim \|\Phi \|_{L^\infty}\|\mathcal Rv_{\mathrm p}\|+\|\mathcal R\Phi v_{\mathrm p}\|\le C_\nu   s^{\nu(\sigma-1/2)},\\
\|\nabla e_2\|
&\lesssim \|\nabla\Phi \|\|\mathcal Rv_{\mathrm p}\|_{L^\infty}+\|\mathcal R(\nabla \Phi )v_{\mathrm p}\|\\
\nonumber
&\quad
+\|\Phi \|_{L^\infty}\|\mathcal R\nabla v_{\mathrm p}\|+\|\mathcal R\Phi \nabla v_{\mathrm p}\|+
\|\nabla(|\varphi|^{2\sigma-2}v_{\mathrm p})\varphi\| \|\mathcal Rv_{\mathrm p}\|_{L^\infty}\\
&\le C_\nu  (s^{\nu(\sigma-1/2)-1/4}+s^{(\nu-1)(\sigma-1/2)}+s^{\nu(\sigma-1/2)-1/2}+s^{\nu(\sigma-1/2)+\sigma-5/4})\\
&\le C_\nu  s^{\nu(\sigma-1/2)-1/2}
\end{align*}
provided $3/4<\sigma<1$. Then \eqref{proposition_e_1} for $e_2$ follows as in the case of $e_1$. 

Next we assume $\varphi\in H^2(\R)$ and prove \eqref{proposition_e_2}. To obtain the desired bound for $\|e_1\|$, \eqref{e_5} with $\nu=2$ is sufficient. 
To deal with the term $|\varphi|^{\sigma-1}\Re[\overline{v_{\mathrm p}}e_1]$, we decompose $\Re[\overline{v_{\mathrm p}}e_1]$ into $\Re[\overline{v_{\mathrm p}}\mathcal R_1v_{\mathrm p}]$ and $\Re[\overline{v_{\mathrm p}}\frac{is}{2}\Delta v_{\mathrm p}]$, where $\mathcal R_1=\mathcal R-\frac{is}{2}\Delta$. For the first term, \eqref{e_2} with $f=\varphi$ implies
\begin{align}
v_{\mathrm p}
&=\nabla^{-4}e^{i\gamma\Phi}\left\{\gamma^4 \varphi(\nabla\Phi)^4-i\gamma^3\nabla\left( \varphi(\nabla\Phi)^3\right)\right\}\nonumber
 -\gamma^2\nabla^{-3}e^{i\gamma\Phi}\nabla\left(\varphi(\nabla\Phi)^2\right)\nonumber\\
&\quad +\nabla^{-2}e^{i\gamma\Phi}\left\{2i\gamma (\nabla \varphi)\nabla \Phi+i\gamma \varphi\Delta\Phi+\Delta \varphi\right\}.
\label{e_8}
\end{align}
Here all of $\varphi(\nabla\Phi)^4$, $\nabla( \varphi(\nabla\Phi)^3)$, $\nabla(\varphi(\nabla\Phi)^2)$, $(\nabla \varphi)\nabla \Phi$ and $\varphi\Delta\Phi$ belong to $L^2(\R)$. Indeed, since 
\begin{align}
\nabla \Phi&=2\sigma|\varphi|^{2\sigma-2}\Re[\overline \varphi \nabla \varphi],\nonumber\\
\Delta \Phi&=2\sigma(2\sigma-2)|\varphi|^{2\sigma-4}\left(\Re[\overline \varphi \nabla \varphi]\right)^2+2\sigma |\varphi|^{2\sigma-2}\left(|\nabla\varphi|^2+\Re[\overline \varphi \Delta \varphi]\right),
\label{e_8_0}
\end{align}
we have 
\begin{align*}
\|\nabla \Phi\|_{L^\infty}&\lesssim  \|\varphi\|_{L^\infty}^{2\sigma-1}\|\nabla\varphi\|_{L^\infty}\le C_2,\\
\|\varphi\Delta\Phi\|&\lesssim \|\varphi\|_{L^\infty}^{2\sigma-1}\||\nabla\varphi|^2\|+\|\varphi\|_{L^\infty}^{2\sigma}\|\Delta\varphi\|\le C_2,\\
\|\varphi(\nabla\Phi)^4\|
&\lesssim \|\varphi\|\|\nabla \Phi\|_{L^\infty}^4\le C_2,\\
\|\nabla\left( \varphi(\nabla\Phi)^j\right)\|
&\lesssim \|\nabla\varphi\|\|\nabla\Phi\|_{L^\infty}^j+\|\nabla \Phi\|_{L^\infty}^{j-1}\|\varphi\Delta\Phi\|\le C_2,\quad j=2,3.
\end{align*}
Moreover, 
$
\|\mathcal R_1\nabla^{-j}\|_{L^2\to L^2}\lesssim s^{j/2}$ for $ 2\le j\le 4$ by Lemma \ref{lemma_R}. 
Hence, these estimates and \eqref{e_8} imply
\begin{align}
\|\mathcal R_1v_{\mathrm p}\|
&\lesssim s^2\left(s^{4(\sigma-1)}\|\varphi(\nabla \Phi)^4\|+s^{3(\sigma-1)}\|\nabla(\varphi (\nabla \Phi)^3)\|\right)+s^{3/2+2(\sigma-1)}\|\nabla(\varphi (\nabla \Phi)^2)\|\nonumber\\
&\quad+s\left(s^{\sigma-1}(\|(\nabla\varphi)\nabla\Phi\|+\|\varphi\Delta\Phi\|)+\|\Delta\varphi\|\right)\nonumber\\
&\le C_1  s^{4\sigma-2}+C_2s^{3\sigma-1}+C_2 s^{2\sigma-1/2}+C_2s^\sigma
\le C_2 s^\sigma
\label{e_9}
\end{align}
provided $2/3\le\sigma\le1$. 
Besides, the second term $\Re[\overline{v_{\mathrm p}}\frac{is}{2}\Delta v_{\mathrm p}]$ can be computed explicitly as
\begin{align}
-\frac{s}{2}\Im[\overline{v_{\mathrm p}}\Delta v_{\mathrm p}]
&=-\frac{s}{2}\Im\left[-\gamma^2|\varphi|^2|\nabla\Phi |^2+2i\gamma\overline\varphi\nabla\varphi\nabla\Phi +i\gamma|\varphi|^2\Delta\Phi +\overline\varphi\Delta\varphi\right]\nonumber
\\
&=-s\gamma\Re[\overline\varphi\nabla\varphi\nabla\Phi]-\frac{s\gamma}{2}|\varphi|^2\Delta\Phi-\frac s2\Im[\overline \varphi\Delta\varphi],
\label{e_10}
\end{align}
which implies
\begin{align}
\label{e_11}
\left\||\varphi|^{\sigma-1}\Re\left[\overline{v_{\mathrm p}}\frac{is}{2}\Delta v_{\mathrm p}\right]\right\|\le C_2   s^{\sigma}.
\end{align}
These estimates \eqref{e_9} and \eqref{e_11} imply
\begin{align}
 s^{\sigma/2-1}\||\varphi|^{\sigma-1}\Re[\overline{v_{\mathrm p}}e_1]\|\le C_2  s^{3\sigma/2-1}.
\label{e_12}
\end{align}
To estimate $\|\nabla e_1\|$, we use the formula \eqref{e_4} with $f=\varphi$ and Lemma \ref{lemma_R} to observe
\begin{align}
\|\nabla e_1\|
\le  C_1s^{1+3(\sigma-1)}+C_2s^{1+2(\sigma-1)}+C_2s^{1/2+\sigma-1}+C_2s^{1/2}
\label{e_13}
\end{align}
Now \eqref{proposition_e_2} follows from \eqref{e_5} with $s=2$, \eqref{e_12} and \eqref{e_13} since
$$\min\left\{2\sigma-1-\frac \sigma2+\frac\delta2,\ \frac{3\sigma}{2}-1,\ 2\sigma-1,\ \sigma-\frac12, \frac12\right\}=\frac{3\sigma}{2}-1$$
for $2/3<\sigma<1$ and $\delta>0$. Note that the decay rate $s^{\sigma}$ in \eqref{e_11} is optimal due to the formula \eqref{e_10}, so is the rate $s^{3\sigma/2-1}$ in \eqref{proposition_e_2}. 

Finally, we prove \eqref{proposition_e_3}. 
To show the desired estimates for $e_2$ and $|\varphi|^{\sigma-1}\Re[\overline{v_{\mathrm p}}i e_2]$, we write
\begin{align*}
e_2
&=\left(\Phi \frac{is}{2}\Delta-\frac{is}{2}\Delta\Phi \right)v_{\mathrm p}
+\left(\Phi \mathcal R_1-\mathcal R_1\Phi \right)v_{\mathrm p}\\
&\quad -s\sigma |\varphi|^{2\sigma-2}v_{\mathrm p}\Im \left[\overline{v_{\mathrm p}} \Delta v_{\mathrm p}\right]
+2\sigma |\varphi|^{2\sigma-2}v_{\mathrm p}\Re\left[\overline{v_{\mathrm p}}\mathcal R_1v_{\mathrm p}\right]\\
&=:\mathrm{I}_1+\mathrm{I}_2+\mathrm{I}_3+\mathrm{I}_4.
\end{align*}
The first term $\mathrm{I}_1=\left(\Phi \frac{is}{2}\Delta-\frac{is}{2}\Delta\Phi \right)v_{\mathrm p}$ is further  calculated as
\begin{align*}
\mathrm{I}_1
=-is(\nabla\Phi )\nabla v_{\mathrm p}-\frac{is}{2}(\Delta\Phi )v_{\mathrm p}
=se^{i\gamma \Phi }\left(\gamma \varphi|\nabla\Phi |^2-i(\nabla\Phi) \nabla \varphi\right)-\frac{is}{2}(\Delta\Phi )v_{\mathrm p},
\end{align*}
which implies 
$\Im[\overline{v_{\mathrm p}}\mathrm{I}_1]=-s\Im[i\overline{\varphi}\nabla\Phi \nabla \varphi]-\frac s2\Im[i|\varphi|^2\Delta\Phi ]$. Hence
\begin{align}
\|\mathrm{I}_1\|
&\lesssim
s\left(s^{\sigma-1}\|\varphi|\nabla\Phi |^2\|+\|(\nabla\Phi )\nabla\varphi\|\right)
+s\|\varphi\Delta\Phi \|
\le C_2  s^{\sigma},\label{e_14}\\
\||\varphi|^{\sigma-1}\Im[\overline{v_{\mathrm p}}\mathrm{I}_1]\|
&\lesssim s\left(\||\varphi|^\sigma|\nabla\Phi |\nabla\varphi\|+\||\varphi|^{\sigma+1}\Delta\Phi \|\right)
\le C_2  s.
\label{e_15}
\end{align}
By \eqref{e_10}, the third term $\mathrm{I}_3=-s\sigma |\varphi|^{2\sigma-2}v_{\mathrm p}\Im \left[\overline{v_{\mathrm p}} \Delta v_{\mathrm p}\right]$ satisfies 
\begin{align}
\|\mathrm{I}_3\|\le C_2  s^\sigma,\quad \Im [\overline{v_{\mathrm p}}\mathrm{I}_3]=0.
\label{e_16}
\end{align}
For the fourth term $\mathrm{I}_4=2\sigma |\varphi|^{2\sigma-2}v_{\mathrm p}\Re\left[\overline{v_{\mathrm p}}\mathcal R_1v_{\mathrm p}\right]$, \eqref{e_9} implies
\begin{align}
\label{e_17}
\|\mathrm{I}_4\|\le C_2 s^\sigma,\quad \Im [\overline{v_{\mathrm p}}\mathrm{I}_4]=0.
\end{align}
To deal with the second term $\mathrm{I}_2=\left(\Phi \mathcal R_1-\mathcal R_1\Phi \right)v_{\mathrm p}$, we observe by \eqref{e_8} and \eqref{e_9} that
\begin{align*}
\|\mathcal R_1v_{\mathrm p}\|\le C_2(s^{4\sigma-2}+s^{3\sigma-1}+s^{2\sigma-1/2}+s)+\gamma \|\mathcal R_1\nabla^{-2}e^{i\gamma\Phi}\left(2 (\nabla\varphi)\nabla\Phi+\varphi\Delta\Phi\right)\|,
\end{align*}
where the term $C_2(s^{4\sigma-2}+s^{3\sigma-1}+s^{2\sigma-1/2}+s)$ is dominated by $C_2s$ provided $\sigma\ge 3/4$. On the other hand, by \eqref{e_8_0}, the term $2 (\nabla\varphi)\nabla\Phi+\varphi\Delta\Phi$  can  be written in the form
\begin{align*}
&|\varphi|^{2\sigma-2}\left\{a_1\overline \varphi (\nabla\varphi)^2+a_2\varphi|\nabla\varphi|^2\right\}
+a_3|\varphi|^{2\sigma-4}\varphi^3(\overline{\nabla\varphi})^2
+a_4|\varphi|^{2\sigma-2}\varphi\Re[\overline{\varphi}{\Delta\varphi}]
\end{align*}
with some constants $a_1,...a_4$. For short, we set 
$$
\mathrm{I}_{21}:=|\varphi|^{2\sigma-2}\left\{a_1\overline \varphi (\nabla\varphi)^2+a_2\varphi|\nabla\varphi|^2\right\}
+a_3|\varphi|^{2\sigma-4}\varphi^3(\overline{\nabla\varphi})^2,\quad 
\mathrm{I}_{22}:=a_4|\varphi|^{2\sigma-2}\varphi\Re[\overline{\varphi}{\Delta\varphi}],
$$
which satisfy
$
|\mathrm{I}_{21}|\lesssim |\varphi|^{2\sigma-1}|\nabla\varphi|^2$ and $|\mathrm{I}_{22}|\lesssim |\varphi|^{2\sigma}|\Delta\varphi|
$,  and write
\begin{align*}
e^{i\gamma\Phi}\left(2 (\nabla\varphi)\nabla\Phi+\varphi\Delta\Phi\right)
=\mathrm{I}_{21}+(e^{i\gamma\Phi}-1)\mathrm{I}_{21}+e^{i\gamma\Phi}\mathrm{I}_{22}.
\end{align*}
Since $0<2-2\sigma<2\sigma-1<1$ if $3/4<\sigma<1$, by Lemmas \ref{lemma_R} and \ref{lemma_Holder} with $(\alpha,\beta)=(2\sigma-1,2-2\sigma)$, the first term $\mathrm{I}_{21}$ satisfies
\begin{align}
\gamma \|\mathcal R_1\nabla^{-2}\mathrm{I}_{21}\|
&\lesssim s^{\sigma-1+(2-\sigma)}\|\mathrm{I}_{21}\|_{\dot H^{2-2\sigma}}\nonumber\\
&\le s\|\varphi\|_{H^{3/2}}^{2\sigma-1}\left(\||\nabla \varphi|^2\|_{H^{2-2\sigma}}+\|(\nabla\varphi)^2\|_{H^{2-2\sigma}}\right)
\le C_2s.
\label{e_18}
\end{align}
For the term $(e^{i\gamma\Phi}-1)\mathrm{I}_{21}$, thanks to the bound $|e^{i\gamma\Phi}-1|\lesssim \gamma\Phi$, we obtain
\begin{align*}
|\nabla\{(e^{i\gamma\Phi}-1)\mathrm{I}_{21}\}|
\lesssim \gamma\Phi|\nabla \mathrm{I}_{21}|+\gamma |\nabla \Phi||\mathrm{I}_{21}|
\lesssim s^{\sigma-1}\left(|\varphi|^{4\sigma-1}|\nabla\varphi||\Delta\varphi|+|\varphi|^{4\sigma-2}|\nabla\varphi|^3\right)
\end{align*}
and hence
\begin{align}
\gamma \|\mathcal R_1\nabla^{-2}(e^{i\gamma\Phi}-1)\mathrm{I}_{21}\|
\lesssim s^{\sigma-1+3/2}\|\nabla\{(e^{i\gamma\Phi}-1)\mathrm{I}_{21}\}\|\le C_2 s^{2\sigma-1/2}\le C_2s
\label{e_19}
\end{align}
whenever $\sigma\ge 3/4$. 
The term $e^{i\gamma\Phi}\mathrm{I}_{22}$ is estimated by using Lemma  \ref{lemma_R} and \ref{lemma_v_p} as 
\begin{align}
\gamma \|\mathcal R_1\nabla^{-2}e^{i\gamma\Phi}\mathrm{I}_{22}\|\lesssim s^{\sigma-1+\nu/2}\|e^{i\gamma\Phi}\mathrm{I}_{22}\|_{\dot H^{\nu-2}}\le 
C_\nu s^{\nu(\sigma-1/2)+1-\sigma}
\label{e_20}
\end{align}
for $2\le \nu\le 3$. In particular, if we choose $\nu=\frac{\sigma}{\sigma-1/2}$, which belongs to $[2,3]$ if $3/4\le \sigma\le1$, then 
$$
\gamma \|\mathcal R_1\nabla^{-2}e^{i\gamma\Phi}\mathrm{I}_{22}\|\le C_{\frac{\sigma}{\sigma-1/2}}s.$$
By \eqref{e_18}--\eqref{e_20}, we obtain
\begin{align}
\|\mathcal R_1v_{\mathrm p}\|\le C_{\nu}s^{\nu(\sigma-1/2)+1-\sigma},\quad 2\le \nu\le \frac{\sigma}{\sigma-1/2}.
\label{e_21}
\end{align}
Since $\Phi\varphi\in H^2(\R)$, the same argument also provides the same bound for $\|\mathcal R_1\Phi v_{\mathrm p}\|$. Hence
\begin{align}
\|\mathrm{I}_2\|\le C_{\nu}s^{\nu(\sigma-1/2)+1-\sigma},\quad 2\le \nu\le \frac{\sigma}{\sigma-1/2}.
\label{e_22}
\end{align}
It follows from \eqref{e_14}--\eqref{e_17} and \eqref{e_21} that, for $2\le \nu\le {\sigma}/{(\sigma-1/2)}$, 
\begin{align}
s^{-\sigma/2+\delta/2}\|e_2\|+s^{\sigma/2-1}\||\varphi|^{\sigma-1}\Im[\overline{v_{\mathrm p}}e_2]\|\le C_\nu s^{\nu(\sigma-1/2)-\sigma/2}. 
\label{e_23}
\end{align}

It remains to deal with $\nabla e_2$ given by the formula \eqref{e_7}. For the term $\left\{(\nabla \Phi )\mathcal R-\mathcal R(\nabla \Phi )\right\}v_{\mathrm p}$ in \eqref{e_7}, we use \eqref{e_1} with $f=\varphi,\nabla\Phi$ to observe
\begin{align}
&\|\left\{(\nabla \Phi )\mathcal R-\mathcal R(\nabla \Phi )\right\}v_{\mathrm p}\|\nonumber\\
&\lesssim \|\nabla \Phi\|_{L^\infty}\|\mathcal R\nabla^{-1}\|_{\dot H^1\to L^2}\left(s^{\sigma-1}e^{i\gamma \Phi}\varphi\nabla \Phi\|_{\dot H^1}+\|e^{i\gamma \Phi}\nabla\varphi\|_{\dot H^1}\right)\nonumber\\
&\quad +s^{\sigma-1}\|\mathcal R\nabla^{-1}\|_{\dot H^1\to L^2}\|e^{i\gamma \Phi}\varphi(\nabla \Phi)^2\|_{\dot H^1}+\|\mathcal R\nabla^{-1}\|_{L^2\to L^2}\|e^{i\gamma \Phi}\nabla(\varphi\nabla\Phi)\|
\nonumber\\
&\le C_2 (s^{1+2(\sigma-1)}+s^{1+\sigma-1}+s^{1+2(\sigma-1)}+s^{1/2})
\le C_2s^{\sigma/2}
\label{e_24}
\end{align}
provided $\sigma\ge 2/3$. 
We next compute 
\begin{align*}
\left(\Phi \mathcal R-\mathcal R\Phi \right)\nabla v_{\mathrm p}
=\left(\Phi \mathcal R-\mathcal R\Phi \right)e^{i\gamma\Phi }(i\gamma \varphi\nabla\Phi +\nabla\varphi),
\end{align*}
where the term $\left(\Phi \mathcal R-\mathcal R\Phi \right)e^{i\gamma\Phi }\nabla\varphi$ satisfies the same bound as \eqref{e_24}. Moreover, the term $i\gamma \left(\Phi \mathcal R-\mathcal R\Phi \right)e^{i\gamma\Phi }\varphi\nabla\Phi $ is further calculated as
\begin{align*}
&i\gamma \Phi \mathcal R(\nabla+1)^{-1}(\nabla+1)e^{i\gamma\Phi }\varphi\nabla\Phi 
-i\gamma \mathcal R(\nabla+1)^{-1}(\nabla+1)\Phi e^{i\gamma\Phi }\varphi\nabla\Phi \\
&=-\gamma^2 \left(\Phi \mathcal R(\nabla+1)^{-1}-\mathcal R(\nabla+1)^{-1}\Phi \right)e^{i\gamma\Phi } \varphi|\nabla\Phi |^2\\
&\quad +i\gamma \Phi \mathcal R(\nabla+1)^{-1}e^{i\gamma\Phi }(\nabla+1)(\varphi\nabla\Phi )
-i\gamma \mathcal R(\nabla+1)^{-1}e^{i\gamma\Phi }(\nabla+1)(\varphi\Phi \nabla\Phi )\\
&=\mathrm{II}_1+\mathrm{II}_2+\mathrm{II}_3,
\end{align*}
where we set for short
\begin{align*}
\mathrm{II}_1&=-\gamma^2 \left(\Phi \frac{is}{2}\Delta(\nabla+1)^{-1}-\frac{is}{2}\Delta(\nabla+1)^{-1}\Phi \right)e^{i\gamma\Phi } \varphi|\nabla\Phi |^2,\\
\mathrm{II}_2&=-\gamma^2 \left(\Phi \mathcal R_1(\nabla+1)^{-1}-\mathcal R_1(\nabla+1)^{-1}\Phi \right)e^{i\gamma\Phi } \varphi|\nabla\Phi |^2,\\
\mathrm{II}_3&=i\gamma \Phi \mathcal R(\nabla+1)^{-1}e^{i\gamma\Phi }(\nabla+1)(\varphi\nabla\Phi )
-i\gamma \mathcal R(\nabla+1)^{-1}e^{i\gamma\Phi }(\nabla+1)(\varphi\Phi \nabla\Phi ).
\end{align*}
Since 
$$
\Delta(\nabla+1)^{-1}=\nabla(\nabla+1-1)(\nabla+1)^{-1}=\nabla-\nabla(\nabla+1)^{-1}=\nabla-1+(\nabla+1)^{-1},
$$
we have, for any $f\in L^2(\R)$ and $g\in H^1(\R)$ satisfying $\nabla g\in L^\infty(\R)$, 
\begin{align}
\left\{g\Delta(\nabla+1)^{-1}-\Delta(\nabla+1)^{-1}g\right\}f
=-(\nabla g)f+[g,(\nabla+1)^{-1}]f.
\label{e_25}
\end{align}
Since $\|[g,(\nabla+1)^{-1}]\|_{L^2\to L^2}\le 2\|g\|_{L^\infty}$, applying this formula with $g=\Phi$ to $\mathrm{II}_1$ shows
\begin{align}
\|\mathrm{II}_1\|\lesssim s^{2\sigma-1}\left(\|\nabla\Phi \|_{L^\infty}+\|\Phi \|_{L^\infty}\right)\|e^{i\gamma\Phi } \varphi|\nabla\Phi |^2\|\le C_2 s^{2\sigma-1}.
\label{e_26}
\end{align}
For the term $\mathrm{II}_2$, since $|\varphi||\nabla\Phi |^2\lesssim |\varphi|^{4\sigma-1}|\nabla\varphi|^2$ and thus $\varphi|\nabla\Phi |^2\in H^1(\R)$, we can rewrite it as
\begin{align}
\mathrm{II}_2
&=-\gamma^2 \Phi \mathcal R_1(\nabla+1)^{-2}e^{i\gamma\Phi }\left\{i\gamma \varphi|\nabla\Phi |^3+(\nabla+1)(\varphi|\nabla\Phi |^2)\right\}\nonumber\\
&\quad 
+\gamma^2 \mathcal R_1(\nabla+1)^{-2}e^{i\gamma\Phi }\left\{i\gamma \Phi \varphi|\nabla\Phi |^3+(\nabla+1)(\Phi \varphi|\nabla\Phi |^2)\right\}. 
\label{e_27}\end{align}
Then Lemmas \ref{lemma_R} and \ref{lemma_v_p} imply
\begin{align*}
\|\mathrm{II}_2\|
&\lesssim 
s^{3(\sigma-1)+3/2}\left(\|e^{i\gamma\Phi }\varphi|\nabla\Phi |^3\|_{\dot H^1}+\|e^{i\gamma\Phi }\Phi \varphi|\nabla\Phi |^3\|_{\dot H^1}\right)\nonumber\\
&\quad +s^{2(\sigma-1)+1}\left(\|e^{i\gamma\Phi }(\nabla+1)(\varphi|\nabla\Phi |^2)\|+\|e^{i\gamma\Phi }(\nabla+1)(\Phi \varphi|\nabla\Phi |^2)\|\right)\nonumber\\
&\le C_2(s^{4\sigma-5/2}+s^{2\sigma-1}).
\end{align*}
To deal with the term $\mathrm{II}_3$, we observe by \eqref{e_8_0} that $(\nabla+1)(\varphi\nabla\Phi )$ is a linear combination of 
\begin{align*}
&|\varphi|^{2\sigma-2}\overline\varphi(\nabla\varphi)^2,\quad 
|\varphi|^{2\sigma-2}\varphi|\nabla\varphi|^2,\quad 
|\varphi|^{2\sigma-4}\varphi^3(\overline{\nabla\varphi})^2,\\
&|\varphi|^{2\sigma-2}\varphi\Re[\overline\varphi\Delta\varphi],\quad
|\varphi|^{2\sigma-2}\varphi\Re[\overline\varphi\nabla\varphi].
\end{align*}
Using Lemma \ref{lemma_Holder} with $(\alpha,\beta,m)=(2\sigma-1,\nu-2,1)$ or $(2\sigma,\nu-2,0)$, we thus have $(\nabla+1)(\varphi\nabla\Phi )\in H^{\nu-2}(\R)$ provided $\nu<2\sigma+1$. We also have $(\nabla+1)(\varphi\Phi\nabla\Phi )\in H^{\nu-2}(\R)$ by the same argument. Thus Lemmas \ref{lemma_R} and \ref{lemma_v_p} imply
\begin{align}
\|\mathrm{II}_3\|\le C_{\nu} s^{\sigma-1+1/2+(\nu-2)(\sigma-1/2)}=C_\nu s^{\nu(\sigma-1/2)+1/2-\sigma},\quad 2\le \nu<2\sigma+1.
\label{e_28}
\end{align}
In particular, choosing $\nu=\frac{3\sigma-1}{2\sigma-1}$ which belongs to $[2,2\sigma+1)$ if $3/4<\sigma\le1$, we have
$$
\|\mathrm{II}_3\|\le C_{\frac{3\sigma-1}{2\sigma-1}} s^{\sigma/2}. 
$$
The remaining three terms in \eqref{e_7} can essentially be dealt with by the same argument. Indeed, in the same manner as the proof of \eqref{e_9}, the term $\nabla(|\varphi|^{2\sigma-2}v_{\mathrm p})\Re\left[\overline{v_{\mathrm p}}\mathcal Rv_{\mathrm p}\right]$ is estimated as 
\begin{align}
&\|\nabla(|\varphi|^{2\sigma-2}v_{\mathrm p})\Re\left[\overline{v_{\mathrm p}}\mathcal Rv_{\mathrm p}\right]\|\nonumber\\
&\le \|\nabla(|\varphi|^{2\sigma-2}\varphi)\overline{\varphi}\|_{L^\infty}\|\mathcal Rv_{\mathrm p}\|
+\gamma \||\varphi|^{2\sigma-2}\varphi\nabla\Phi \|_{L^\infty}\|\Re\left[\overline{v_{\mathrm p}}\mathcal Rv_{\mathrm p}\right]\|\nonumber\\
&\le C_2s^{2\sigma-1}.
\label{e_29}
\end{align}
To deal with the sum of last two terms of the RHS of \eqref{e_7}, we compute
\begin{align*}
\overline{\nabla v_{\mathrm p}}\mathcal Rv_{\mathrm p}+\overline{v_{\mathrm p}}\mathcal R\nabla v_{\mathrm p}
&=\overline{\left(i\gamma v_{\mathrm p}\nabla\Phi +e^{i\gamma\Phi }\nabla\varphi\right)}\mathcal Rv_{\mathrm p}+\overline{v_{\mathrm p}}\mathcal R\left({i\gamma v_{\mathrm p}\nabla\Phi }+{e^{i\gamma\Phi }\nabla\varphi}\right)\\
&=-i\gamma \overline{v_{\mathrm p}}\left\{(\nabla\Phi )\mathcal R-\mathcal R(\nabla\Phi )\right\}v_{\mathrm p}
+e^{-i\gamma\Phi }\overline{(\nabla\varphi)}\mathcal Rv_{\mathrm p}
+\overline{v_{\mathrm p}}\mathcal Re^{i\gamma\Phi }\nabla\varphi,
\end{align*}
where the last two terms satisfy the desired bound as
\begin{align}
&\||\varphi|^{2\sigma-2}v_{\mathrm p}(e^{-i\gamma\Phi }\overline{(\nabla\varphi)}\mathcal Rv_{\mathrm p}
+\overline{v_{\mathrm p}}\mathcal Re^{i\gamma\Phi }\nabla\varphi)\|\nonumber\\
&\le C_2 \left(\|\mathcal Re^{i\gamma\Phi }\varphi\|+\|\mathcal Re^{i\gamma\Phi }\nabla\varphi\|\right)\nonumber\\
&\le C_2\left(\|\mathcal R\nabla^{-1}e^{i\gamma\Phi}(i\gamma \varphi\nabla\Phi+\nabla\varphi)\|+\|\mathcal R\nabla^{-1}e^{i\gamma\Phi}(i\gamma \nabla\varphi\nabla\Phi+\Delta\varphi)\|\right)\nonumber\\
&\le C_2 (s^{2\sigma-1}+s^{1/2}).
\label{e_30}
\end{align}
Moreover, $-i\gamma \overline{v_{\mathrm p}}\left\{(\nabla\Phi )\mathcal R-\mathcal R(\nabla\Phi )\right\}v_{\mathrm p}$ can be written in the form
\begin{align*}
&-i\gamma \overline{v_{\mathrm p}}\left\{(\nabla\Phi )\mathcal R-\mathcal R(\nabla\Phi )\right\}v_{\mathrm p}\\
&=\gamma^2 \overline{v_{\mathrm p}}\left\{(\nabla\Phi )\mathcal R(\nabla+1)^{-1}-\mathcal R(\nabla+1)^{-1}(\nabla\Phi )\right\}e^{i\gamma\Phi }\varphi\nabla\Phi \\
&\quad -i\gamma \overline{v_{\mathrm p}}\left\{(\nabla\Phi )\mathcal R(\nabla+1)^{-1}e^{i\gamma\Phi }(\nabla+1)\varphi- \mathcal R(\nabla+1)^{-1}e^{i\gamma\Phi }(\nabla+1)(\varphi\nabla\Phi )\right\}
\\
&=\gamma^2 \overline{v_{\mathrm p}}\left\{(\nabla\Phi )\frac{is}{2}\Delta(\nabla+1)^{-1}-\frac{is}{2}\Delta(\nabla+1)^{-1}(\nabla\Phi )\right\}e^{i\gamma\Phi }\varphi\nabla\Phi \\
&\quad +\gamma^2 \overline{v_{\mathrm p}}\left\{(\nabla\Phi )\mathcal R_1(\nabla+1)^{-1}-\mathcal R_1(\nabla+1)^{-1}(\nabla\Phi )\right\}e^{i\gamma\Phi }\varphi\nabla\Phi \\
&\quad -i\gamma \overline{v_{\mathrm p}}\left\{(\nabla\Phi )\mathcal R(\nabla+1)^{-1}e^{i\gamma\Phi }(\nabla+1)\varphi- \mathcal R(\nabla+1)^{-1}e^{i\gamma\Phi }(\nabla+1)(\varphi\nabla\Phi )\right\}\\
&=:\mathrm{III}_1+\mathrm{III}_2+\mathrm{III}_3.
\end{align*}
Since the formula \eqref{e_25} with $g=\nabla\Phi $ implies
\begin{align*}
\Re[\mathrm{III}_1]
&=\Re\left[-\frac{is\gamma^2}{2} \overline{v_{\mathrm p}}\left(\Delta\Phi +[\nabla\Phi ,(\nabla+1)^{-1}]\right)e^{i\gamma\Phi }\varphi\nabla\Phi \right]\\
&=\Re\left[-\frac{is\gamma^2}{2}\left(|\varphi|^2\Delta\Phi \nabla\Phi +\overline{v_{\mathrm p}}[\nabla\Phi ,(\nabla+1)^{-1}]e^{i\gamma\Phi }\varphi\nabla\Phi \right)\right]\\
&=\frac{s\gamma^2}{2}\Im\left[\overline{v_{\mathrm p}}[\nabla\Phi ,(\nabla+1)^{-1}]e^{i\gamma\Phi }\varphi\nabla\Phi \right].
\end{align*}
Hence
\begin{align*}
\||\varphi|^{2\sigma-2}v_{\mathrm p}\Re[\mathrm{III}_1]\|\lesssim s^{2\sigma-1}\|\varphi\|_{L^\infty}^{2\sigma}\|\nabla\Phi \|_{L^\infty}\|e^{i\gamma\Phi }\varphi\nabla\Phi \|\le C_2s^{2\sigma-1}. 
\end{align*}
Rewriting $\mathrm{III}_2$ as
\begin{align*}
&i\gamma^3 \overline{v_{\mathrm p}}\left\{(\nabla\Phi )\mathcal R_1(\nabla+1)^{-2}-\mathcal R_1(\nabla+1)^{-2}(\nabla\Phi )\right\}e^{i\gamma\Phi }\varphi|\nabla\Phi |^2\\
&\quad +\gamma^2 \overline{v_{\mathrm p}}\left\{(\nabla\Phi )\mathcal R_1(\nabla+1)^{-2}e^{i\gamma\Phi }(\nabla+1)(\varphi\nabla\Phi )-\mathcal R_1(\nabla+1)^{-2}e^{i\gamma\Phi }(\nabla+1)(\varphi|\nabla\Phi |^2)\right\},
\end{align*}
we observe by the same argument as that for $\mathrm{II}_2$ that
\begin{align*}
\||\varphi|^{2\sigma-2}v_{\mathrm p}\Re[\mathrm{III}_2]\|\le C_2(s^{4\sigma-5/2}+s^{2\sigma-1}).
\end{align*}
Moreover, by the same argument as that for $\mathrm{II}_3$, we have
\begin{align*}
\||\varphi|^{2\sigma-2}v_{\mathrm p}\Re[\mathrm{III}_3]\|\le C_\nu s^{\nu(\sigma-1/2)+1/2-\sigma},\quad 2\le \sigma<2\sigma+1.
\end{align*}
These estimates for $\mathrm{III}_1,\mathrm{III}_2,\mathrm{III}_3$, together with \eqref{e_23}, \eqref{e_26}--\eqref{e_30}, imply 
\begin{align}
\|\nabla e_2\|\le C_{\nu}s^{\nu(\sigma-1/2)+1/2-\sigma},\quad 2\le \nu\le \frac{3\sigma-1}{2\sigma-1}.
\label{e_31}
\end{align}
Since $-\sigma/2<1/2-\sigma$ and $(3\sigma-1)/(2\sigma-1)<\sigma/(\sigma-1/2)$ for $\sigma<1$, the desired bound \eqref{proposition_e_3} for $Q[ie_2]$ follows by \eqref{e_23} and \eqref{e_31}. This completes the proof. 
\end{proof}

\subsection{Proof of Proposition \ref{proposition_regularized_1}}
\label{subsection_existence_3}
Here we prove \eqref{proposition_regularized_1_1} and \eqref{proposition_regularized_1_2}, thereby completing the proof of Proposition \ref{proposition_regularized_1}. Let $\varphi\in H^\nu(\R)$. Recall that $Q[f](s)$ was defined by \eqref{Q}. Define for short
$$
\rho(s)= s^{-\beta+\delta}Q[w_\ep](s),
$$
where $\beta=\beta(\sigma,\nu)$ is defined by
\begin{align}
\label{p_2_1_0}
\beta(\sigma,\nu)=\begin{cases}
\beta^-(\sigma,\nu):=\nu(\sigma-1/2)+{3\sigma}/{2}-2&\text{if}\quad 1\le \nu<2,\\
\beta^+(\sigma,\nu):=\nu(\sigma-1/2)+\sigma/2-1&\text{if}\quad 2\le \nu<\frac{\sigma}{\sigma-1/2}.
\end{cases}
\end{align}
In order to show \eqref{proposition_regularized_1_1} and \eqref{proposition_regularized_1_2}, it is sufficient to prove there exist $C_0,\beta_0>0$ such that 
\begin{align}
\label{p_2_1_1}
1+\rho(s)\le C_0+C_0\int_\ep^s r^{\beta_0-1}\left(1+\rho(r)\right)^{2\sigma+1}dr
\end{align}
for any $0<\ep< s\le s_0$. Indeed, if we denote by $x(s)$ the RHS of \eqref{p_2_1_1}, then $x\in C^1([\ep,s_0])$ and $x'(s)\le C_0 s^{\beta_0-1}x(s)^{2\sigma+1}$ on $[\ep,s_0]$. Then
$$\int_\ep^s \frac{x'}{x^{2\sigma+1}}dr=\left[-\frac{1}{2\sigma x(s)^{2\sigma}}\right]_\ep^s=\frac{1}{2\sigma x(\ep)^{2\sigma}}-\frac{1}{2\sigma x(s)^{2\sigma}}\le C_0\int_\ep^s r^{\beta_0-1}ds\le\frac{C_0 s^{\beta_0}}{\beta_0},$$where $x(\ep)=C_0$. Therefore
$$
x(s)^{2\sigma}
\le\frac{\beta_0 C_0^{2\sigma}}{\beta_0-2\sigma C_0^{2\sigma+1} s^{\beta_0}}
$$
as long as the RHS is finite. By taking $s_0$ small if necessary in such a way that $4\sigma C_0^{2\sigma+1} s_0^{\beta_0}<\beta_0$, we have $\rho(s)\le 2^{1/(2\sigma)}C_0$ for all $0<\ep< s\le s_0$. 
Hence \eqref{proposition_regularized_1_1} follows. In order to obtain \eqref{proposition_regularized_1_2}, we use the NLS \eqref{NLS_w_ep} and the equation
$$
i\partial_s e_1=i\partial_s(e^{is\Delta/2}-I)v_{\mathrm p}=\lambda  s^{\sigma-2}\mathcal R|\varphi|^{2\sigma} v_{\mathrm p}+\frac12e^{is\Delta/2}\Delta v_{\mathrm p},
$$
Lemma \ref{lemma_R}, Proposition \ref{proposition_e} and \eqref{proposition_regularized_1_1} to observe 
\begin{align*}
\|\partial_s w_\ep\|_{H^{-1}}
&\lesssim 
\|\Delta w_\ep\|_{H^{-1}}+\|\Delta e_1\|_{H^{-1}}+\|\partial_s e_1\|_{H^{-1}}\\
&\quad+ s^{\sigma-2}\left(\||v_\ep|^{2\sigma}v_\ep-|v_{\mathrm p}|^{2\sigma}v_{\mathrm p}\|+\|\mathcal R|\varphi|^{2\sigma}v_{\mathrm p}\|_{H^{-1}}\right)\\
&\lesssim \|w_\ep\|_{H^1}+\|v_{\mathrm p}\|_{H^1}+ s^{\sigma-1}\||\varphi|^{2\sigma}v_{\mathrm p}\|_{H^1}\\
&\quad + s^{\sigma-2}\left\{\left(\|v_\ep\|_{L^\infty}^{2\sigma}+\|\varphi\|_{L^\infty}^{2\sigma}\right)\|w_\ep\|+s\||\varphi|^{2\sigma}v_{\mathrm p}\|_{H^1}\right\}\\
&\lesssim 1+s^{\sigma-1}+s^{3\sigma/2-2+\beta-3\delta/2}+s^{2\sigma-2}\\
&\lesssim s^{3\sigma/2-2}
\end{align*}
provided $\delta<2\beta/3$. 
Then \eqref{proposition_regularized_1_2} follows by integrating over the interval $[s',s]$. 

We now turn to the proof of \eqref{p_2_1_1}. It follows from \eqref{IE_1} and Proposition \ref{proposition_energy} that
\begin{align*}
\rho(s) \lesssim s^{-\beta+\delta}\left(Q[e_1](\ep)+Q[e_1](s)\right)+\int_\ep^s  r^{\sigma-2-\beta+\delta}\left(Q[iG(w_\ep,v_{\mathrm p})](r)+Q[ie_2](r)\right)dr. 
\end{align*}
To estimate the nonlinear term $Q[G(w_\ep,v_{\mathrm p})](s)$, we use Proposition \ref{proposition_nonlinear}  to observe
\begin{align}
& s^{\sigma-1-\beta+\delta}Q[i G(w_\ep,v_{\mathrm p})](s)\nonumber\\
&\lesssim 
 s^{\frac{\sigma^2}{2}+2\sigma-2+\frac{(\sigma+1)\delta}{2}-\beta+\delta}Q[w_\ep](s)^{2\sigma+1}
+ s^{\frac{5\sigma}{4}-1-\frac{\delta}{4}-\beta+\delta}Q[w_\ep](s)^{2}\nonumber\\
&\quad + s^{\frac{\sigma^2}{2}+\frac{(\min\{\nu,3/2\}+1)\sigma}{2}-1-\frac{(\sigma-\nu+1)\delta}{2}-\beta+\delta}Q[w_\ep](s)^{2\sigma}\nonumber\\
&\lesssim s^{\mu-C_1\delta}(1+\rho(s))^{2\sigma+1},
\label{p_2_1_2}
\end{align}
where $C_1=C_1(\sigma,\nu)>0$ is independent of $\delta$ and $\mu=\min\{\mu_1,\mu_2,\mu_3\}$ with
\begin{align*}
\mu_1(\sigma,\nu)&=\frac{\sigma^2}{2}+2\sigma-2+2\sigma \beta(\sigma,\nu),\\
\mu_2(\sigma,\nu)&=\frac{\sigma^2}{2}+\frac{(\min\{\nu,3/2\}+1)\sigma}{2}-1+(2\sigma-1)\beta(\sigma,\nu),\\
\mu_3(\sigma,\nu)&=\frac{5\sigma}{4}-1+\beta(\sigma,\nu).
\end{align*}
For the error terms $e_1$ and $e_2$, Proposition \ref{proposition_e} implies
\begin{align*}
s^{-\beta+\delta}\left(Q[e_1](\ep)+Q[e_1](s)\right)
&\lesssim
\begin{cases}
s^{\nu(\sigma-1/2)+\sigma/2-1-\beta+\delta}&\text{if}\quad 1\le \nu<2,\\
s^{{3\sigma}/{2}-1-\beta+\delta}&\text{if}\quad \nu\ge 2,
\end{cases}\\
s^{\sigma-1-\beta}Q[e_2](s)
&\lesssim 
\begin{cases}
s^{\nu(\sigma-1/2)+{3\sigma}/{2}-2-\beta}&\text{if}\quad 1\le \nu<2,\\
s^{\nu(\sigma-1/2)+\sigma/2-1-\beta}&\text{if}\quad 2\le \nu<\sigma/(\sigma-1/2).
\end{cases}
\end{align*}
Plugging \eqref{p_2_1_0} into these estimates, we obtain 
\begin{align}
s^{-\beta+\delta}\left(Q[e_1](\ep)+Q[e_1](s)\right)+\int_\ep^s r^{\sigma-2-\beta+\delta}Q[e_2](r)dr\lesssim 1
\label{p_r_p_3}
\end{align}
uniformly with respect to $0<\ep\le s\le s_0$. 
It follows from \eqref{p_2_1_2} and \eqref{p_r_p_3} that \eqref{p_2_1_1} holds with  $\beta_0=\mu/2$ as long as $\beta,\mu>0$ and $0<\delta<\mu/(2C_1)$.  

Now we shall show $\beta,\mu>0$ provided $2/\sqrt{7}<\sigma<1$ and $\nu_*(\sigma)<\nu<\sigma/(\sigma-1/2)$, where $\nu_*(\sigma)$ was given in Theorem \ref{theorem_1}. To check that $2/\sqrt{7}$ is in fact the possible infimum, we compute
\begin{align*}
\beta(\sigma,\sigma/(\sigma-1/2))&=3\sigma/2-1,\quad
\mu_1(\sigma,\sigma/(\sigma-1/2))=7\sigma^2/2-2,\\
\mu_2(\sigma,\sigma/(\sigma-1/2))&=7\sigma^2/2-9\sigma/4,\quad
\mu_3(\sigma,\sigma/(\sigma-1/2))=11\sigma/4-2.
\end{align*}
Thus $\beta(\sigma,\sigma/(\sigma-1/2))$ and $\mu(\sigma,\sigma/(\sigma-1/2))$ are positive if and only if 
$$
\sigma>\max\{2/3,2/\sqrt{7},9/14,8/11\}=2/\sqrt{7}. 
$$

It remains to determine $\nu_*(\sigma)$. To this end, we denote by $\nu_0(\sigma)$ (resp. $\nu_j(\sigma)$ for $j=1,2,3$) the solution to $\beta(\sigma,\nu)=0$  (resp. $\mu_j(\sigma,\nu,\beta)=0)$ for $j=1,2,3$). We also use the notation $\nu_j^\pm(\sigma)=\nu_j(\sigma)$ when $\beta=\beta^\pm$, respectively.
For short, we also set
$$
\nu_*^{\#}(\sigma)=\inf\{\nu\in [1,\sigma/(\sigma-1/2))\ |\ \nu\#2,\ \beta(\sigma,\nu)>0,\ \mu(\sigma,\nu)>0\},\quad \#\in \{\ge,\le\},
$$
where we use the convention $\inf \emptyset=\infty$. 
Then, we have $$\nu_*(\sigma)=\max\nu_j(\sigma)=\min\{\nu_*^{\ge}(\sigma),\nu_*^{\le}(\sigma)\}.$$
Now we determine $\nu_*^{\ge}(\sigma)$. To this end, we assume $\beta=\beta^+=\nu(\sigma-1/2)+\sigma/2-1$, or equivalently, restrict the range of $\nu$ onto $[2,\sigma/(\sigma-1/2))$, in which case $\nu_j=\nu^+_j$. Then we have
\begin{align*}
\mu_1(\sigma,\nu)&=3\sigma^2/2-2+2\sigma(\sigma-1/2)\nu,\\
\mu_2(\sigma,\nu)&=3\sigma^2/2-5\sigma/4+(2\sigma-1)(\sigma-1/2)\nu,\\
\mu_3(\sigma,\nu)&=7\sigma/4-2+(\sigma-1/2)\nu
\end{align*}
and hence
\begin{align*}
\nu_0^+(\sigma)=\frac{2-\sigma}{2\sigma-1},\quad 
\nu_1^+(\sigma)=\frac{4-3\sigma^2}{2\sigma(2\sigma-1)},\quad
\nu_2^+(\sigma)=\frac{5\sigma/2-3\sigma^2}{(2\sigma-1)^2},\quad
\nu_3^+(\sigma)=\frac{4-7\sigma/2}{2\sigma-1}.
\end{align*}
It is easy to check that, in the range $2/\sqrt7<\sigma<1$, 
\begin{align*}
\nu_1^+(\sigma)-\nu_0^+(\sigma)&=\frac{4-4\sigma-\sigma^2}{2\sigma(2\sigma-1)}\ge0\quad \text{if and only if}\quad 2/\sqrt7< \sigma\le 2\sqrt2-2,\\
\nu_0^+(\sigma)-\nu_2^+(\sigma)&=\frac{2\sigma^2+5\sigma-4}{2(2\sigma-1)^2}>0,\quad
\nu_1^+(\sigma)-\nu_3^+(\sigma)=\frac{4\sigma^2+4-8\sigma}{2\sigma(2\sigma-1)}>0.
\end{align*}
Hence
$$
\max_{0\le j\le3}\nu_j^+(\sigma)=
\begin{cases}
\nu_1^+(\sigma)&\text{if}\quad 2/\sqrt{7}<\sigma\le 2\sqrt{2}-2,\\
\nu_0^+(\sigma)&\text{if}\quad 2\sqrt{2}-2\le \sigma<1.
\end{cases}
$$
Moreover, we observe that $\nu_1^+(\sigma)\ge2$ if and only if $\sigma\le \frac{2+4\sqrt3}{11}$  and that $\nu_0^+(\sigma)\ge2$ if and only if $\sigma\le4/5$. Since $4/5<\frac{2+4\sqrt3}{11}<2\sqrt2-2$, we have
\begin{align}
\nu_*^{\ge}(\sigma)=\max\{\nu_j^+(\sigma),2\}
=\begin{cases}
\nu_1^+(\sigma)=\frac{4-3\sigma^2}{2\sigma(2\sigma-1)}&\text{if}\quad 2/\sqrt7<\sigma<\frac{2+4\sqrt3}{11},\\
2&\text{if}\quad \frac{2+4\sqrt3}{11}\le \sigma<1.
\end{cases}
\label{p_r_p_4}
\end{align}
Note that $\frac{4-3\sigma^2}{2\sigma(2\sigma-1)}|_{\sigma=\frac{2+4\sqrt3}{11}}=2$. To determine $\nu_*^{\le}(\sigma)$, assuming $\beta=\beta^-=\nu(\sigma-1/2)+3\sigma/2-2$ and hence $\nu_j=\nu_j^-$ by definition, we compute
\begin{align*}
\mu_1(\sigma,\nu)&=7\sigma^2/2-2\sigma-2+2\sigma(\sigma-1/2)\nu,\\
\mu_2(\sigma,\nu)&=7\sigma^2/2+\left(\min\{\nu/2,3/4\}-5\right)\sigma+1+(2\sigma-1)(\sigma-1/2)\nu,\\
\mu_3(\sigma,\nu)&=11\sigma/4-3+(\sigma-1/2)\nu
\end{align*}
and
\begin{align*}
\nu_0^-(\sigma)&=\frac{4-3\sigma}{2\sigma-1},\quad
\nu_1^-(\sigma)=\frac{4+4\sigma-7\sigma^2}{2\sigma(2\sigma-1)},\\
\nu_2^-(\sigma)&=\frac{-2+\left(10-\min\{\nu,3/2\}\right)\sigma-7\sigma^2}{(2\sigma-1)^2},\quad
\nu_3^-(\sigma)=\frac{6-11\sigma/2}{2\sigma-1}.
\end{align*}
Then, for $2/\sqrt7<\sigma<1$, 
\begin{align*}
\nu_1^-(\sigma)-\nu_0^-(\sigma)&=\frac{4-4\sigma-\sigma^2}{2\sigma(2\sigma-1)}\ge0\quad \text{if and only if}\quad 2/\sqrt7< \sigma\le 2\sqrt2-2,\\
\nu_0^-(\sigma)-\nu_2^-(\sigma)
&=\frac{\sigma^2+\left(\min\{\nu,3/2\}+1\right)\sigma-2}{(2\sigma-1)^2}
\ge \frac{\sigma^2+2\sigma-2}{(2\sigma-1)^2}>0,\\
\nu_1^-(\sigma)-\nu_3^-(\sigma)&=\frac{4-8\sigma+4\sigma^2}{2\sigma(2\sigma-1)}>0.
\end{align*}
Hence 
$$
\max_{0\le j\le3} \nu_j^-(\sigma)
=\begin{cases}
\nu_1^-(\sigma)&\text{if}\quad 2/\sqrt{7}<\sigma\le 2\sqrt{2}-2,\\
\nu_0^-(\sigma)&\text{if}\quad 2\sqrt{2}-2\le \sigma<1.
\end{cases}$$ 
Moreover, noting that $2\sqrt2-2<6/7$, we have $\nu_0^-(\sigma)\le 2$ if and only if $\sigma\ge 6/7$ and 
$\max\{\nu_1^-(\sigma),\nu_0^-(\sigma)\}>2$ for $\sigma\le 6/7$. Therefore, we have
\begin{align}
\nu^\le_*(\sigma)
=
\begin{cases}
\infty &\text{if}\quad 2/\sqrt7<\sigma<6/7,\\
\nu_0^-(\sigma)=
\frac{4-3\sigma}{2\sigma-1}&\text{if}\quad 6/7\le \sigma<1.\\
\end{cases}
\label{p_r_p_5}
\end{align}
Since $\frac{2+4\sqrt3}{11}<6/7$, by \eqref{p_r_p_4} and \eqref{p_r_p_5}, we conclude
$$
\nu_*(\sigma)=\min\{\nu_*^\ge(\sigma),\nu_*^\le(\sigma)\}=
\begin{cases}
\frac{4-3\sigma^2}{2\sigma(2\sigma-1)}&\text{if}\quad \frac{2}{\sqrt7}<\sigma\le\frac{2+4\sqrt3}{11},\\
2 &\text{if}\quad \frac{2+4\sqrt3}{11}\le\sigma\le\frac 67,\\
\frac{4-3\sigma}{2\sigma-1}&\text{if}\quad \frac 67\le \sigma<1.
\end{cases}
$$
This completes the proof. 
\qed

\subsection{Proof of Theorem \ref{theorem_v}: Global existence}
\label{section_proof_existence}
We are ready to prove the existence part of Theorem \ref{theorem_v}. 
Let $s_1=s_0/2$ and $0<\ep\le s_1$. We first construct $v$ satisfying \eqref{theorem_v_1}. Define $$\widetilde v_\ep(s,x):=v_\ep(\ep+s,x),\quad \widetilde v_{\mathrm p}(s,x):=v_{\mathrm p}(\ep+s,x),\quad \widetilde w_\ep(s,x):=w_\ep(\ep+s,x)=\widetilde v_\ep-\widetilde v_{\mathrm p},$$ where $w_\ep=v_\ep-v_{\mathrm p}$ and $v_\ep$ is the solution to \eqref{NLS_v_ep} obtained by Proposition \ref{proposition_regularized_1}. Then $\widetilde v_\ep$ satisfies\begin{align}\label{proof_existence_0}(i\partial_s+\frac12\Delta)\widetilde v_\ep=(\ep+s)^{\sigma-2}|\widetilde v_\ep|^{2\sigma}\widetilde v_\ep,\quad x\in \R,\quad 0<s\le s_1. \end{align}
Moreover, we have $\|\widetilde v_\ep\|=\|\varphi\|$ and, by \eqref{proposition_regularized_1_1} and \eqref{proposition_regularized_1_2}, 
\begin{align*}
&\|\nabla \widetilde w_\ep(s)\|
+(\ep+s)^{-\sigma/2+\delta/2}\|\widetilde w_\ep(s)\|
+(\ep+s)^{-1+\sigma/2}\||\varphi|^{\sigma-1}\Re[\overline{\widetilde v_{\mathrm p}(s)}\widetilde w_\ep(s)]\|\lesssim (\ep+s)^{\beta-\delta},\\
&\|\widetilde w_\ep(s)-\widetilde w_\ep(s')\|_{H^{-1}}\lesssim |s-s'|^{3\sigma/2-1},
\end{align*}
uniformly in $\ep\in (0,s_1]$ and $s\in [0,s_1]$. In particular, $\{\widetilde w_\ep\}_{\ep\in (0,s_1]}$ is bounded in $C([0,s_1],H^1(\R))$ and, for any $\alpha<1$, the map $[0,s_1]\ni s\mapsto \widetilde w_\ep\in H^\alpha(\R)$ is uniformly equicontinuous in $\ep\in (0,s_1]$. Then, one can use an abstract theorem \cite[Proposition 1.1.2]{Cazenave} to construct a function
\begin{align}
\label{proof_existence_1}
w\in C_{\mathrm w}([0,s_1],H^1(\R))\cap W^{1,\infty}([0,s_1],H^{-1}(\R))\cap C([0,s_1],H^\alpha(\R))
\end{align}
and a sequence $\ep_j\in (0,s_1]$ with $\ep_j\to 0$ as $j\to \infty$ such that $\widetilde  w_{\ep_j}(s)\rightharpoonup w(s)$ in $H^1(\R)$ as $j\to \infty$ for all $s\in [0,s_1]$, where $C_{\mathrm w}([0,s_1],H^1(\R))$ denotes the space of $H^1$-valued weakly continuous functions. 
By virtue of the weak lower semi-continuity of Hilbert norms,  this weak convergence and the above modified energy estimate for $\widetilde w_{\ep_j}$ yield
\begin{align*}
\|w(s)\|&\le \liminf_{j\to\infty}\| \widetilde w_{\ep_j}(s)\|\lesssim s^{\sigma/2+\beta-3\delta/2},\\
\|\nabla w(s)\|&\le \liminf_{j\to\infty}\|\nabla  \widetilde w_{\ep_j}(s)\|\lesssim s^{\beta-\delta},\\
\||\varphi|^{\sigma-1}\Re[\overline{ v_{\mathrm p}(s)}  w(s)]\|&\le \liminf_{j\to\infty}\||\varphi|^{\sigma-1}\Re[\overline{\widetilde v_{\mathrm p}(s)} \widetilde w_{\ep_j}(s)]\|\lesssim s^{1+\beta-\sigma/2-\delta},
\end{align*}
uniformly in $s\in [0,s_1]$. 
Define $v:=w+v_{\mathrm p}$. Since $$v_{\mathrm p}\in C((0,\infty),H^1(\R))\cap W_{\mathrm{loc}}^{1,\infty}((0,\infty),H^{-1}(\R)),$$ we thus have the following properties
\begin{itemize}
\item $\widetilde v_{\ep_j}(s)\rightharpoonup v(s)$ in $H^1(\R)$ as $j\to \infty$ for all $s\in (0,s_1]$;
\item $v \in C_{\mathrm w}((0,1],H^1(\R))\cap W_{\mathrm{loc}}^{1,\infty}((0,1],H^{-1}(\R))\cap C((0,1],H^\alpha(\R))$ for $\alpha<1$;
\item $v$ satisfies \eqref{theorem_v_1} for $s\in (0,s_1]$. 
\end{itemize}
Moreover, for any $0<s_1'<s_1$, there exists $C>0$ such that, for a.e. $s\in [s_1',s_1]$, 
$$
|\<\partial_s v(s),v(s)\>_{H^{-1},H^1}|\le \|\partial_s v\|_{H^{-1}}\|v(s)\|_{H^1}\le C. 
$$
Hence, $\|v(s)\|^2$ belongs to $W^{1,\infty}_{\mathrm{loc}}((0,s_1])$ and satisfies, for  a.e. $s\in (0,s_1]$,  
\begin{align}
\label{proof_existence_2}
\frac{d}{ds}\|v(s)\|^2=2\Re\<\partial_s v,v\>_{H^{-1},H^1}.
\end{align}

Next we show that $v$ satisfies NLS \eqref{NLS_v}. Suppose $s\in (0,s_1]$. We first observe by the compact embedding $H^1(I)\subset L^2(I)$ that $\widetilde v_{\ep_j}(s)\to v(s)$ in $L^2(I)$ as $j\to \infty$ for any compact interval $I\Subset \R$. Moreover, by virtue of \eqref{theorem_v_1}, $v$ satisfies
$$
\|v(s)\|_{L^\infty(\R)}\le \|w(s)\|_{H^1(\R)}+\|v_{\mathrm p}(s)\|_{L^\infty(\R)}\lesssim s^{\beta-\delta}+\|\varphi\|_{L^\infty}\lesssim1
$$
uniformly in $0< s\le s_1$. The same uniform bound also holds for $\widetilde v_{\ep_j}$ by \eqref{proposition_regularized_1_1}. Hence 
\begin{align}
&\||\widetilde v_{\ep_j}(s)|^{2\sigma}\widetilde v_{\ep_j}(s)-|v(s)|^{2\sigma}v(s)\|_{L^2(I)}\nonumber\\ 
&\lesssim \left(\|\widetilde v_{\ep_j}(s)\|_{L^\infty(\R)}^{2\sigma}+\|v(s)\|_{L^\infty(\R)}^{2\sigma}\right)\|\widetilde v_{\ep_j}(s)-v(s)\|_{L^2(I)}\nonumber\\
&\lesssim \|\widetilde v_{\ep_j}(s)-v(s)\|_{L^2(I)},
\nonumber
\end{align}
where the implicit constants are independent of $j\in \N$, $s\in (0,s_1]$ and $I\subset \R$. Thus, $$|\widetilde v_{\ep_j}(s)|^{2\sigma}\widetilde v_{\ep_j}(s)\to  |v(s)|^{2\sigma}v(s)\ \text{in}\ L^2(I)$$ for all $I\Subset \R$. Since $I\Subset \R$ can be taken arbitrarily, we obtain by the density argument that $$|\widetilde v_{\ep_j}(s)|^{2\sigma}\widetilde v_{\ep_j}(s)\rightharpoonup |v(s)|^{2\sigma}v(s)\ \text{in}\ L^2(\R).$$ Since $\{|\widetilde v_{\ep_j}(s)|^{2\sigma}\widetilde v_{\ep_j}(s)\}_j$ is bounded in $H^1(\R)$, by passing to a subsequence if necessary, we  have
$$|\widetilde v_{\ep_j}(s)|^{2\sigma}\widetilde v_{\ep_j}(s)\rightharpoonup |v(s)|^{2\sigma}v(s)\ \text{in}\ H^1(\R).$$
Now we can take the limit $j\to \infty$ in \eqref{proof_existence_0} with $\ep=\ep_j$, showing that $v$ satisfies \eqref{NLS_v} in $H^{-1}(\R)$ for a.e. $s\in (0,s_1]$. In particular, \eqref{proof_existence_2} implies
$
\frac{d}{ds}\|v(s)\|^2=0
$ for a.e. $s\in (0,s_1]$. 
Since $\|v(s)\|$ is locally Lipschitz continuous on $(0,s_1]$ by the property $\|v(\cdot)\|^2\in W^{1,\infty}_{\mathrm{loc}}((0,s_1])$, $\|v(s)\|$ is constant on $[s_1',s_1]$  for any $0<s_1'<s_1$ and hence on $(0,s_1]$. In particular, the mass conservation 
$\|v(s)\|=\|\varphi\|$ holds for all $s\in (0,s_1]$ since
$$
\big|\|v(s)\|-\|\varphi\|\big|=\big|\|v(s)\|-\|v_{\mathrm p}(s)\|\big|\le \|w(s)\|\to 0\quad\text{as}\quad s\to 0.
$$
This, combined with the mass conservation $\|\varphi\|=\|\widetilde v_{\ep_j}(s)\|$, implies $\widetilde v_{\ep_j}(s)\to v(s)$ in $L^2(\R)$ for all $s\in (0,s_1]$. Since $\{\widetilde v_{\ep_j}(s)\}_j$ is bounded in $H^1(\R)$, we also obtain by interpolation that $\widetilde v_{\ep_j}(s)\to v(s)$  and $|\widetilde v_{\ep_j}(s)|^{2\sigma}\widetilde v_{\ep_j}(s)\to |v(s)|^{2\sigma}v(s)$ in $H^\alpha(\R)$ for any $\alpha<1$ and $s\in (0,s_1]$. This ensures that $v$ satisfies \eqref{NLS_v} in $H^{-2+\alpha}(\R)$ for all $0<s\le s_1$ and Duhamel's formula 
\begin{align}
\label{Duhamel}
v(s)=e^{i(s-\tau)\Delta/2}v(\tau)-i\lambda \int_{\tau}^s r^{\sigma-2}e^{i(s-r)\Delta/2}|v(r)|^{2\sigma}v(r)dr
\end{align}
in $H^\alpha(\R)$ for all $0<s,\tau\le s_1$. 

It remains to prove $v\in C((0,s_1],H^1(\R)$, which implies \eqref{Duhamel} in $H^1(\R)$ for all $0<s,\tau\le s_1$ and thus $v$ is a strong $H^1$ solution to \eqref{NLS_v} satisfying \eqref{theorem_v_1}. 
Let us set 
$$
E(s)=\frac12\|\nabla v(s)\|^2+\frac{\lambda s^{\sigma-2}}{\sigma+1}\int_\R|v(s)|^{2\sigma+2}dx.
$$
Then the following pseudo-conformal conservation law holds: 
\begin{align}
\label{law}
E(s)-E(\tau)=\frac{\lambda (\sigma-2)}{\sigma+1}\int_{\tau}^s r^{\sigma-3}\int_\R|v(r)|^{2\sigma+2}dxdr;
\end{align}
see Appendix \ref{appendix_A} for the proof. 
In particular, for any $s,\tau\in (0,s_1]$, 
\begin{align*}
\frac12\|\nabla v(s)\|^2-\frac12\|\nabla v(\tau)\|^2&=\frac{\lambda }{\sigma+1}\left(-s^{\sigma-2}\int_\R|v(s)|^{2\sigma+2}dx+ \tau^{\sigma-2}\int_\R|v(\tau)|^{2\sigma+2}dx\right)\\
&\quad +\frac{\lambda (\sigma-2)}{\sigma+1}\int_{\tau}^s r^{\sigma-3}\int_\R|v(r)|^{2\sigma+2}dxdr,
\end{align*}
where the RHS converges to $0$ as $s\to \tau$ since $v\in C((0,s_1],H^\alpha(\R))\subset C((0,s_1],L^{2\sigma+2}(\R))$ for $1/2<\alpha<1$. This shows that $(0,s_1]\ni s\mapsto \|v(s)\|_{H^1}$ is continuous, thereby implying $v\in C((0,s_1],H^1(\R))$ since $v\in C_{\mathrm w}((0,s_1],H^1(\R))$.  \qed


\section{Proof of Theorem \ref{theorem_v}: Uniqueness}
\label{section_uniqueness}
Here we prove the uniqueness part of Theorem \ref{theorem_v}, thereby completing the proof of Theorem \ref{theorem_v}. Let $\varphi\in H^\nu(\R)$ and $v_1,v_2\in C((0,s_1],H^1(\R))$ be two solutions to \eqref{NLS_v} satisfying \eqref{theorem_v_1} for any $\delta>0$. Let $W:=v_1-v_2$.  
By the same calculation as that in Section \ref{subsection_energy}, $W$ satisfies
\begin{align*}
(i\partial_s+\frac12\Delta)W
&=\lambda s^{\sigma-2}\left\{(\sigma+1)|v_1|^{2\sigma}W+\sigma |v_1|^{2\sigma-2}v_1^2\overline W+G(W,v_1)\right\}\\
&=\lambda s^{\sigma-2}\left\{(\sigma+1)|\varphi|^{2\sigma}W+\sigma |\varphi|^{2\sigma-2}v_{\mathrm p}^2\overline W+G(W,v_1)+\widetilde{G}(W)\right\},
\end{align*}
where $G(W,v_1)$,  defined by \eqref{G_ep}, is of the form
$$
G(W,v_1)=\int_0^1\left((\sigma+1)W(|W_{[\theta]}|^{2\sigma}-|v_1|^{2\sigma})
+\sigma \overline{W} (|W_{[\theta]}|^{2\sigma-2}W_{[\theta]}^2-|v_1|^{2\sigma-2}v_1^2)\right)d\theta
$$
with $W_{[\theta]}=v_1+\theta W$, and
\begin{align*}
\widetilde{G}(W)
:=(\sigma+1)W\left(|v_1|^{2\sigma}-|\varphi|^{2\sigma}\right)+\sigma \overline W\left(|v_1|^{2\sigma-2}v_1^2-|\varphi|^{2\sigma-2}v_{\mathrm p}^2\right).
\end{align*}
Hence  $\vv W=(W,\overline W)^{\mathrm T}$ satisfies, for any $0<\ep\le s\le1$, 
\begin{align*}
\vv W(s)=\mathcal U(s,\ep)\vv W(\ep)-\lambda \int_{\ep}^s r^{\sigma-2}\mathcal U_0(s,r)\left(\vv{(iG)}(W,v_1)+\vv{(i\widetilde{G})}(W)\right)(r)dr.
\end{align*}
Then Proposition \ref{proposition_energy} implies
\begin{align}
Q[W](s)
\lesssim 
Q[W](\ep)+
\int_{\ep}^s r^{\sigma-2}\left(Q[i G(W,v_1)](r)+Q[i \widetilde{G}(W)](r)\right)dr,
\label{proof_uniqueness_1}
\end{align}
where, by \eqref{theorem_v_1}, $Q[W](\ep)$ satisfies
\begin{align}
Q[W](\ep)\le Q[v_1-v_{\mathrm p}](\ep)+Q[v_2-v_{\mathrm p}](\ep) \lesssim \ep^{\beta-\delta}. 
\label{proposition_uniqueness_2}
\end{align}
Moreover, we have the following estimate for the nonlinear term: 
\begin{proposition}
\label{proposition_uniqueness}
Let $\delta>0$ be sufficiently small. Then there exists $\beta_1>0$ such that
\begin{align}
Q[i G(W,v_1)](s)+Q[i \widetilde{G}(W)](s)
&\lesssim s^{\beta_1+1-\sigma} Q[W](s)
\label{proposition_uniqueness_1}
\end{align}
uniformly in $0<s\le s_1$. 
\end{proposition}

Before proving this proposition, we shall finish the proof of the uniqueness. By \eqref{proposition_uniqueness_2}, \eqref{proposition_uniqueness_1} and the dominated convergence theorem, we can take the limit $\ep\searrow 0$ in \eqref{proof_uniqueness_1}, yielding
$$
Q[W](s)
\lesssim \int_{0}^s r^{\beta_1-1}Q[W](r)dr,\quad 0<s\le s_1.
$$
This shows $Q[W](s)=0$ for all $0<s\le s_1$ by Gronwall's inequality, and hence $v_1\equiv v_2$. \qed

\begin{proof}[Proof of Proposition \ref{proposition_uniqueness}]
By Lemma \ref{lemma_nonlinear}, $Q[f](s)\sim\widetilde Q[f](s)$ uniformly in $0<s\le s_1$ with $\widetilde Q[f]$ defined by \eqref{widetilde_Q}. Then, by the same argument as in the proof of \eqref{p_2_1_1}, it is sufficient to prove
\begin{align}
&\widetilde Q[i G(W,v_1)](s) +\widetilde Q[i \widetilde{G}(W)](s)\nonumber\\
&\lesssim (s^{\frac{\sigma^2}{2}+\sigma-1+2\sigma\beta-}+s^{\frac{\sigma}{4}+\beta-}+s^{\frac{\sigma^2}{2}+\frac{(\nu_1-1)\sigma}{2}+(2\sigma-1)\beta-})\widetilde Q[W](s),
\label{p_u_p_1}
\end{align}
where $\nu_1=\min\{\nu, 3/2\}$ and we suppress the dependence with respect to $\delta$, simply writing $A\lesssim s^{b-}B$ if $A\lesssim s^{b-C\delta}B$ with some $b,C>0$. 
The proof of \eqref{p_u_p_1} is almost identical to that of Proposition \ref{proposition_nonlinear}. Indeed, by the definition \eqref{widetilde_Q} of $\widetilde Q$ and \eqref{theorem_v_1}, we have
\begin{align}
\|W\|&\lesssim s^{\frac\sigma2-}\widetilde Q[W]\lesssim s^{\frac\sigma2+\beta-},\quad \|\nabla(e^{-i\gamma \Phi}W)\|\lesssim \widetilde Q[W]\lesssim s^{\beta-},\nonumber\\
\|W\|_{L^\infty}&\lesssim s^{\frac\sigma4-}\widetilde Q[W]\lesssim s^{\frac\sigma4+\beta-},\quad \||\varphi|^{\sigma-1}\Re[\overline{v_{\mathrm p}}W]\|\lesssim s^{1-\frac\sigma2}\widetilde Q[W]\lesssim s^{1-\frac\sigma2+\beta-}. \label{p_u_p_1_0}
\end{align}
These estimates, \eqref{proposition_nonlinear_proof_1} and \eqref{theorem_v_1} imply
\begin{align}
s^{-\frac{\sigma}{2}+\frac\delta2}\|G(W,v_1)\|
\lesssim s^{-\frac{\sigma}{2}+\frac\delta2}\|W\|\left(\|W\|^{2\sigma}_{L^\infty}+\|W\|_{L^\infty}\right)
\lesssim (s^{\frac{\sigma^2}{2}+2\sigma\beta-}+s^{\frac\sigma4+\beta-})\widetilde Q[W](s).
\label{p_u_p_2}
\end{align}
To deal with $\Im[\overline{v_{\mathrm p}}G(W,v_1)]$, we repeat the same argument used to show \eqref{proposition_nonlinear_proof_4_0}. By expanding $|W_{[\theta]}|^{2\sigma}$ and $|W_{[\theta]}|^{2\sigma-2}W_{[\theta]}^2$ around $v_1$, the integrand of $\Im[\overline{v_{\mathrm p}}G(W,v_1)]$ can be written as
\begin{align*}
&(\sigma+1)\Im\left[\overline{v_{\mathrm p}}W(|W_{[\theta]}|^{2\sigma}-|v_1|^{2\sigma})\right]
+\sigma \Im\left[\overline{v_{\mathrm p}}\overline{W} (|W_{[\theta]}|^{2\sigma-2}W_{[\theta]}^2-|v_1|^{2\sigma-2}v_1^2)\right]\\
&=(\sigma+1) \Im\left[\overline{v_{\mathrm p}}W\left\{\sigma |v_1|^{2\sigma-2}\overline{v_1}\theta W+\sigma|v_1|^{2\sigma-2}v_1\theta \overline{W}\right\}\right]\\
&\quad +\sigma \Im\left[\overline{v_{\mathrm p}}\overline{W} \left\{(\sigma+1)|v_1|^{2\sigma-2}v_1\theta W+(\sigma-1)|v_1|^{2\sigma-4}v_1^3\theta \overline{W}\right\}\right]
+\Im\left[\overline{v_{\mathrm p}}G_{3,\theta}\right],
\end{align*}
where, setting $W_{[\theta,\rho]}:=v_1+\rho(W_{[\theta]}-v_1)=v_1+\rho\theta W$, 
\begin{align*}
G_{3,\theta}
&=\theta (\sigma+1)\sigma W^2 \int_0^1 \left(|W_{[\theta,\rho]}|^{2\sigma-2}\overline{W_{[\theta,\rho]}}-|v_1|^{2\sigma-2}\overline{v_1}\right)d\rho\\
&\quad+2\theta (\sigma+1)\sigma |W|^2 \int_0^1 \left(|W_{[\theta,\rho]}|^{2\sigma-2}W_{[\theta,\rho]}-|v_1|^{2\sigma-2}v_1\right)d\rho
\\
&\quad+\theta \sigma(\sigma-1) \overline{W}^2 \int_0^1 \left(|W_{[\theta,\rho]}|^{2\sigma-4}W_{[\theta,\rho]}^3-|v_1|^{2\sigma-4}v_1^3\right)d\rho. 
\end{align*}
The first two terms are further rewritten as \begin{align*}
&\Im\left[\overline{v_{\mathrm p}}W\left\{\sigma |v_1|^{2\sigma-2}\overline{v_1}\theta W+\sigma|v_1|^{2\sigma-2}v_1\theta \overline{W}\right\}\right]\\
&=\Im\left[\overline{v_{\mathrm p}}W\left\{\sigma |v_{\mathrm p}|^{2\sigma-2}\overline{v_{\mathrm p}}\theta W+\sigma|v_{\mathrm p}|^{2\sigma-2}v_{\mathrm p}\theta \overline{W}\right\}\right]\\
&\quad +\Im\left[\overline{v_{\mathrm p}}W\left\{\sigma \left(|v_1|^{2\sigma-2}\overline{v_1}-|v_{\mathrm p}|^{2\sigma-2}\overline{v_{\mathrm p}}\right)\theta W
+\sigma\left(|v_1|^{2\sigma-2}v_1-|v_{\mathrm p}|^{2\sigma-2}v_{\mathrm p}\right)\theta \overline{W}\right\}\right]\\
&=2\sigma \theta |\varphi|^{2\sigma-2}\Re[\overline{v_{\mathrm p}}W]\Im[\overline{v_{\mathrm p}}W]+O(|\varphi||v_1-v_{\mathrm p}|^{2\sigma-1}|W|^2),\\
&\Im\left[\overline{v_{\mathrm p}}\overline{W} \left\{(\sigma+1)|v_1|^{2\sigma-2}v_1\theta W+(\sigma-1)|v_1|^{2\sigma-4}v_1^3\theta \overline{W}\right\}\right]\\
&=\Im\left[\overline{v_{\mathrm p}}\overline{W} \left\{(\sigma+1)|v_{\mathrm p}|^{2\sigma-2}v_{\mathrm p}\theta W+(\sigma-1)|v_{\mathrm p}|^{2\sigma-4}v_{\mathrm p}^3\theta \overline{W}\right\}\right]\\
&\quad +\Im\left[\overline{v_{\mathrm p}}\overline{W} \left\{(\sigma+1)\left(|v_1|^{2\sigma-2}v_1-|v_{\mathrm p}|^{2\sigma-2}v_{\mathrm p}\right)\theta W
+(\sigma-1)\left(|v_1|^{2\sigma-4}v_1^3-|v_{\mathrm p}|^{2\sigma-4}v_{\mathrm p}^3\right)\theta \overline{W}\right\}\right]\\
&=-2(\sigma-1)\theta |\varphi|^{2\sigma-2}\Re[\overline{v_{\mathrm p}}W]\Im[\overline{v_{\mathrm p}}W]+O(|\varphi||v_1-v_{\mathrm p}|^{2\sigma-1}|W|^2).
\end{align*}
Besides, by \eqref{proposition_nonlinear_proof_2}, $|\Im\left[\overline{v_{\mathrm p}}G_{3,\theta}\right]|$ is dominated by $|\varphi||W|^{2\sigma+1}$.
Therefore, we have\begin{align*}
\Im[\overline{v_{\mathrm p}}G(W,v_1)]
\lesssim |\varphi|^{2\sigma-1}\Re[\overline{v_{\mathrm p}}W]|W|+|\varphi| |v_1-v_{\mathrm p}|^{2\sigma-1}|W|^2+|\varphi||W|^{2\sigma+1}.
\end{align*}
Hence \eqref{theorem_v_1} and \eqref{p_u_p_1_0} imply
\begin{align}
&s^{\frac\sigma2-1}\||\varphi|^{\sigma-1}\Im[\overline{v_{\mathrm p}}G(W,v_1)]\|\nonumber\\
&\lesssim s^{\frac\sigma2-1}\|\varphi\|_{L^\infty}^{2\sigma-1}\||\varphi|^{\sigma-1}\Re[\overline{v_{\mathrm p}}W]\|\|W\|_{L^\infty}\nonumber\\
&\quad +s^{\frac\sigma2-1}\|\varphi\|_{L^\infty}^\sigma \left(\|v_1-v_{\mathrm p}\|_{L^\infty}^{2\sigma-1}\|W\|_{L^\infty}+\|W\|_{L^\infty}^{2\sigma}\right)\|W\|\nonumber\\
&\lesssim s^{\frac\sigma4-}\widetilde Q[W](s)^2+s^{\frac\sigma2-1+\frac{\sigma^2}{2}+\frac\sigma2-}(\widetilde Q[W](s)^2+\widetilde Q[W](s)^{2\sigma+1})\nonumber\\
&\lesssim s^{\frac\sigma4+\beta-}\widetilde Q[W](s)+s^{\frac{\sigma^2}{2}+\sigma-1+2\sigma\beta-}\widetilde Q[W](s).
\label{p_u_p_3}
\end{align}
To deal with $\|\nabla (e^{-i\gamma \Phi} G(W,v_1))\|$, we set $\widetilde v_1=e^{-i\gamma \Phi}v_1$ and $\widetilde W=e^{-i\gamma \Phi}W$ so that
$$
e^{-i\gamma \Phi} G(W,v_1)=\int_0^1 \left \{(\sigma+1)\widetilde W(|\widetilde W_{[\theta]}|^{2\sigma}-|\widetilde v_1|^{2\sigma})+\sigma \overline{\widetilde W}(|\widetilde W_{[\theta]}|^{2\sigma-2}\widetilde W_{[\theta]}^2-|\widetilde v_1|^{2\sigma-2}\widetilde v_1^2)\right\}d\theta,
$$
where $\widetilde W_{[\theta]}=\widetilde v_1+\theta\widetilde W$. 
By \eqref{proposition_nonlinear_proof_6}, $\nabla (\overline{\widetilde W}(|\widetilde W_{[\theta]}|^{2\sigma-2}\widetilde W_{[\theta]}^2-|\widetilde v_1|^{2\sigma-2}\widetilde v_1^2))$ is a linear combination of
\begin{align*}
&(\overline{\nabla  \widetilde W})(|\widetilde W_{[\theta]}|^{2\sigma-2}\widetilde W_{[\theta]}^{2}-|\widetilde v_1|^{2\sigma-2}\widetilde v_1^{2}),\\
&\overline{\widetilde W}(|\widetilde W_{[\theta]}|^{2\sigma-2}W_{[\theta]}-|\widetilde v_1|^{2\sigma-2}\widetilde v_1)\nabla \widetilde v_1,\quad
\overline{\widetilde W}(|\widetilde W_{[\theta]}|^{2\sigma-4}\widetilde W_{[\theta]}^3-|\widetilde v_1|^{2\sigma-4}\widetilde v_1^3)\overline{\nabla \widetilde v_1},\\
&\overline{\widetilde W}(|\widetilde W_{[\theta]}|^{2\sigma-2}\widetilde W_{[\theta]}-|\widetilde v_1|^{2\sigma-2}\widetilde v_1)\nabla (\widetilde W_{[\theta]}-\widetilde v_1),\quad
\overline{\widetilde W}(|\widetilde W_{[\theta]}|^{2\sigma-4}\widetilde W_{[\theta]}^3-|\widetilde v_1|^{2\sigma-4}\widetilde v_1^3)\overline{\nabla (\widetilde W_{[\theta]}-\widetilde v_1)},\\
&\overline{\widetilde W}|\widetilde v_1|^{2\sigma-2}\widetilde v_1\nabla (\widetilde W_{[\theta]}-\widetilde v_1),\quad
\overline{\widetilde W}|\widetilde v_1|^{2\sigma-4}\widetilde v_1^3\overline{\nabla (\widetilde W_{[\theta]}-\widetilde v_1)}.
\end{align*}
With the estimates \eqref{proposition_nonlinear_proof_1}, \eqref{proposition_nonlinear_proof_2}, \eqref{proposition_nonlinear_proof_7}, \eqref{p_u_p_1_0} and
$
|\nabla \widetilde v_1|\le |\nabla(\widetilde v_1-\varphi)|+|\nabla \varphi|
$
at hand, we have
\begin{align*}
&\|\nabla \{\overline{\widetilde W} (|\widetilde W_{[\theta]}|^{2\sigma-2}\widetilde W_{[\theta]}^2-|\widetilde v_1|^{2\sigma-2}\widetilde v_1^2)\}\|\\
&\lesssim \||\nabla \widetilde W|(|\widetilde W|^{2\sigma}+|v_1|^{2\sigma-1}|\widetilde W|)\|+\||\widetilde W|^{2\sigma}\nabla(\widetilde v_1-\varphi)\|+\||\widetilde W|^{2\sigma}\nabla\varphi\|\\
&\quad +\||\widetilde W|^{2\sigma}|\nabla \widetilde W|\|+\||\widetilde W||\widetilde v_1|^{2\sigma-1}\nabla \widetilde W\|\\
&\lesssim \|\nabla \widetilde W\|(\| W\|_{L^\infty}^{2\sigma}+\| W\|_{L^\infty})+s^{\beta}\|W\|_{L^\infty}^{2\sigma}
+\|\nabla\varphi\|_{L^{\frac{2}{3-2\nu_1}}}\|W\|_{L^{\frac{2\sigma}{\nu_1-1}}}^{2\sigma}\\
&\lesssim s^{\frac{\sigma}{4}+\beta-}\widetilde Q[W](s)+s^{\frac{\sigma^2}{2}+\frac{(\nu_1-1)\sigma}{2}+(2\sigma-1)\beta-}\widetilde Q[W](s),
\end{align*}
where we used the same argument as that in the proof of \eqref{proposition_nonlinear_proof_7} to estimate $\|W\|_{L^{\frac{2\sigma}{\nu_1-1}}}^{2\sigma}$ by $s^{\frac{\sigma^2}{2}+\frac{(\nu_1-1)\sigma}{2}-}\widetilde Q[W](s)^{2\sigma}$. 
The same bound also holds for $\nabla\{\widetilde W(|\widetilde W_{[\theta]}|^{2\sigma}-|\widetilde v_1|^{2\sigma})\}$. Hence
\begin{align}
\|\nabla e^{-i\gamma \Phi}G(W,v_1)\|\lesssim s^{\frac{\sigma}{4}+\beta-}\widetilde Q[W](s)+s^{\frac{\sigma^2}{2}+\frac{(\nu_1-1)\sigma}{2}+(2\sigma-1)\beta-}\widetilde Q[W](s).
\label{p_u_p_4}
\end{align}
The above three bounds \eqref{p_u_p_2}--\eqref{p_u_p_4} yield the desired estimate \eqref{p_u_p_1} for $G(W,v_1)$. 

The proof for $\widetilde{G}(W)$ is almost analogous. Indeed, \eqref{theorem_v_1}, \eqref{proposition_nonlinear_proof_1} and \eqref{p_u_p_1_0} imply
\begin{align*}
s^{-\frac\sigma2+\frac\delta2}\|\widetilde{G}(W)\|
\lesssim s^{-\frac\sigma2+\frac\delta2}\|W\|(\|\varphi\|_{L^\infty}^{2\sigma-1}\|v_1-v_{\mathrm p}\|_{L^\infty}+\|v_1-v_{\mathrm p}\|_{L^\infty}^{2\sigma})
\lesssim s^{\frac\sigma4-}\widetilde Q[W](s).
\end{align*}
Moreover, by setting $w=v_1-v_{\mathrm p}$ and rewriting $\overline{v_{\mathrm p}}\widetilde G(W)$ as 
\begin{align*}
\overline{v_{\mathrm p}}\widetilde G(W)
&=(\sigma+1)\overline{v_{\mathrm p}}W\left\{\sigma |v_{\mathrm p}|^{2\sigma-2}\overline{v_{\mathrm p}}w+\sigma|v_{\mathrm p}|^{2\sigma-2}v_{\mathrm p} \overline{w}\right\}\\
&\quad +\sigma \overline{v_{\mathrm p}}\overline{W} \left\{(\sigma+1)|v_{\mathrm p}|^{2\sigma-2}v_{\mathrm p}w+(\sigma-1)|v_{\mathrm p}|^{2\sigma-4}v_{\mathrm p}^3 \overline{w}\right\}
+\overline{v_{\mathrm p}}G_{4}
\end{align*}
with $w_\theta=v_{\mathrm p}+\theta w$ and 
\begin{align*}
G_4
&=(\sigma+1)\sigma Ww\int_0^1 \left(|w_\theta|^{2\sigma-1}\overline w-|v_{\mathrm p}|^{2\sigma-2}\overline{v_{\mathrm p}}\right)d\theta\\
&\quad +2 (\sigma+1)\sigma \Re[ W\overline w] \int_0^1 \left(|w_\theta|^{2\sigma-2}w_{\theta}-|v_{\mathrm p}|^{2\sigma-2}v_{\mathrm p}\right)d\theta
\\
&\quad+ \sigma(\sigma-1) \overline{Ww} \int_0^1 \left(|w_\theta|^{2\sigma-4}w_\theta^3-|v_{\mathrm p}|^{2\sigma-4}v_{\mathrm p}^3\right)d\theta. 
\end{align*}
Since $\Re[z_1z_2]=\Re[z_1]\Im[z_2]+\Im[z_1]\Re[z_2]$, we have
\begin{align*}
|\Im[\overline{v_{\mathrm p}}\widetilde G(W)]|\lesssim |\varphi|^{2\sigma-1}\left(\Re[\overline{v_{\mathrm p}}W]|w|+|W|\Re[\overline{v_{\mathrm p}}w]\right)+|\varphi||W||w|^{2\sigma}.
\end{align*}
The same argument as above then shows
$$
s^{\frac\sigma2-1}\||\varphi|^{\sigma-1}\Im[\overline{v_{\mathrm p}}\widetilde G(W)]\|
\lesssim s^{\frac\sigma4+\beta-}\widetilde Q[W](s)+s^{\frac{\sigma^2}{2}+\sigma-1+2\sigma\beta-}\widetilde Q[W](s).
$$
Finally, the desired estimate for $\|\nabla (e^{-i\gamma \Phi}\widetilde G(W))\|$ is obtained by completely the same argument as that for \eqref{p_u_p_4}. 
Indeed, $\nabla\{e^{-i\gamma \Phi}\overline W(|v_1|^{2\sigma-2}v_1^2-|\varphi|^{2\sigma-2}v_{\mathrm p}^2)$ is a linear combination of the same operators mentioned above as that for $\nabla (\overline{\widetilde W}(|\widetilde W_{[\theta]}|^{2\sigma-2}\widetilde W_{[\theta]}^2-|\widetilde v_1|^{2\sigma-2}\widetilde v_1^2))$ with $\widetilde W_{[\theta]}$ replaced by $\varphi$. Thus, the same argument as above yields
\begin{align*}
&\|\nabla\{e^{-i\gamma \Phi}\overline W(|v_1|^{2\sigma-2}v_1^2-|\varphi|^{2\sigma-2}v_{\mathrm p}^2)\|\\
&\lesssim \||\nabla \widetilde W|(|\widetilde v_1-\varphi|^{2\sigma}+|\varphi|^{2\sigma-1}|\widetilde v_1-\varphi|)\|
+\|\widetilde W|\widetilde v_1-\varphi|^{2\sigma-1}\nabla\widetilde v_1\|\\
&\quad +\|\widetilde W|\widetilde v_1-\varphi|^{2\sigma-1}\nabla(\widetilde v_1-\varphi)\|+\|\widetilde W|\widetilde v_1|^{2\sigma-1}\nabla(\widetilde v_1-\varphi)\|\\
&\lesssim \|\nabla \widetilde W\|\|v_1-v_{\mathrm p}\|_{L^\infty}+\|\widetilde W\|_{L^\infty}\|v_1-v_{\mathrm p}\|_{L^\infty}^{2\sigma-1}\|\nabla(\widetilde v_1-v_{\mathrm p})\|+\| W|v_1-v_{\mathrm p}|^{2\sigma-1}\|_{L^{\frac{1}{\nu_1-1}}}\\
&\lesssim s^{\frac{\sigma}{4}+\beta-}\widetilde Q[W](s)+s^{\frac{\sigma^2}{2}+\frac{(\nu_1-1)\sigma}{2}+(2\sigma-1)\beta-}\widetilde Q[W](s),
\end{align*}
where, in order to bound $\| W|v_1-v_{\mathrm p}|^{2\sigma-1}\|_{L^{\frac{1}{\nu_1-1}}}$ by $s^{\frac{\sigma^2}{2}+\frac{(\nu_1-1)\sigma}{2}+(2\sigma-1)\beta-}\widetilde Q[W](s)$, we interpolate between the following two cases $\nu_1=1,3/2$ obtained by \eqref{theorem_v_1} and \eqref{p_u_p_1_0}: 
\begin{align*}
\| W|v_1-v_{\mathrm p}|^{2\sigma-1}\|_{L^{\infty}}&\lesssim \|W\|_{L^\infty} \|v_1-v_{\mathrm p}\|_{L^\infty}^{2\sigma-1}\lesssim s^{\frac{\sigma^2}{2}+(2\sigma-1)\beta-}\widetilde Q[W](s),\\
\| W|v_1-v_{\mathrm p}|^{2\sigma-1}\|_{L^{2}}&\lesssim \|W\|_{L^2} \|v_1-v_{\mathrm p}\|_{L^\infty}^{2\sigma-1}\lesssim s^{\frac{\sigma^2}{2}+\frac\sigma4+(2\sigma-1)\beta-}\widetilde Q[W](s). 
\end{align*}
This completes the proof of \eqref{p_u_p_1}. 
\end{proof}

\appendix
\section{Some technical materials}
\label{appendix_A}
\subsection{Proof of the identity \eqref{eq_E_ep_1}}
Here we give a rigorous justification of \eqref{eq_E_ep_1}. Let $\psi\in C((0,1],H^1(\R))$ be the solution to \begin{align}
\label{appendix_A_1}
i\partial_s\psi +\frac12\Delta\psi
=\lambda  s^{\sigma-2}\left\{|\varphi|^{2\sigma}\psi+2\sigma  |\varphi|^{2\sigma-2}\Re[\overline{v_{\mathrm p}}\psi]v_{\mathrm p}\right\}.
\end{align}
Let $J_m=(I-\frac1m\Delta)^{-1}$ with $m>0$ and $\psi^m:=J_m\psi \in C^1((0,1],H^1(\R))$. Consider the identity
\begin{align}
\Re \<\text{LHS of }\eqref{appendix_A_1},\partial_s \psi^m\>
&=\Re \<\text{RHS of }\eqref{appendix_A_1},\partial_s \psi^m\>,\label{eq_E_ep_2_0}. 
\end{align}
Here and in what follows, $\<\cdot,\cdot\>$ is regarded as the duality coupling $\<\cdot,\cdot\>_{H^{-1},H^1}$. Since $J_m$ commutes with $\partial_s$ and $\Delta$, we have
\begin{align*}
\Re \<\text{LHS of \eqref{appendix_A_1}},\partial_s \psi^m\>=-\frac14\frac{d}{ds}\|J_m^{1/2}\nabla \psi\|^2=\frac{d}{ds}A_0^m,
\end{align*}
where $A_0^m:=-\frac14\|J_m^{1/2}\nabla \psi\|^2$. Since $\psi\in C((0,1],H^1(\R))$, $A_0^m\to A_0=-\frac14\|\nabla \psi\|^2$ as $m\to \infty$ for all $s\in (0,1]$. Moreover, the term $\Re \<\text{RHS of }\eqref{appendix_A_1},\partial_s \psi^m\>$ can be written as
\begin{align*}
&\lambda s^{\sigma-2}\Re\<|\varphi|^{2\sigma}\psi,\partial_s\psi^m\>
+2\sigma \lambda  s^{\sigma-2}\<|\varphi|^{2\sigma-2}\Re[\overline{v_{\mathrm p}}\psi],\partial_s\Re[\overline{v_{\mathrm p}}\psi^m]\>\\
&\quad -2\sigma\lambda  s^{\sigma-2}\Re \<|\varphi|^{2\sigma-2}\Re[\overline{v_{\mathrm p}}\psi],(\partial_s\overline{v_{\mathrm p}})\psi^m\>\\
&=:A_1^m+A_2^m+A_3^m
\end{align*}
so that 
\begin{align}
\label{eq_E_ep_2_1}
\frac{d}{ds}A_0^m=A_1^m+A_2^m+A_3^m. 
\end{align}
Since $\partial_s \overline{v_{\mathrm p}}=i\lambda  s^{\sigma-2}|\varphi|^{2\sigma}\overline{v_{\mathrm p}}\in C((0,1],H^1(\R))$, we observe that $A_3^m$ converges to 
\begin{align*}
&-2\sigma\lambda  s^{\sigma-2}\Re \<|\varphi|^{2\sigma-2}\Re[\overline{v_{\mathrm p}}\psi],i\lambda  s^{\sigma-2}|\varphi|^{2\sigma}\overline{v_{\mathrm p}}\psi\>\\
&=2\sigma\lambda^2  s^{2\sigma-4} \<|\varphi|^{2\sigma-2}\Re[\overline{v_{\mathrm p}}\psi],|\varphi|^{2\sigma}\Im[\overline{v_{\mathrm p}}\psi]\>=:A_3
\end{align*}
as $m\to \infty$ for all $s\in (0,1]$, where we used the identity $\Re\<f,ig\>=\<f,\Re[ig]\>=-\<f,\Im g\>$ for real-valued $f$. For $A_1^m$, we can rewrite it as
\begin{align}
\nonumber
A_1^m
&=\lambda s^{\sigma-2}\frac{d}{ds}\Re\<|\varphi|^{2\sigma}\psi,\psi^m\>-\lambda s^{\sigma-2}\Re\<|\varphi|^{2\sigma}\partial_s\psi,\psi^m\>\\
\nonumber
&=\frac{d}{ds}\lambda s^{\sigma-2}\Re\<|\varphi|^{2\sigma}\psi,\psi^m\>
-\lambda(\sigma-2) s^{\sigma-3}\Re\<|\varphi|^{2\sigma}\psi,\psi^m\>\\
\nonumber
&\quad-\lambda s^{\sigma-2}\Re\<|\varphi|^{2\sigma}\partial_s\psi,\psi^m\>\\
\label{eq_E_ep_2_2}
&=:\frac{d}{ds}A_{11}^m+A_{12}^m,
\end{align}
where  $A_{11}^m:=\lambda s^{\sigma-2}\Re\<|\varphi|^{2\sigma}\psi,\psi^m\>$ converges to $A_{11}:=\lambda s^{\sigma-2}\||\varphi|^{\sigma}\psi\|^2$ as $m\to \infty$. 
Since $|\varphi|^{2\sigma}\partial_s\psi\in C((0,1],H^{-1}(\R))$, $A_{12}^m$ converges to the function 
\begin{align*}
A_{12}
&=-\lambda(\sigma-2) s^{\sigma-3}\Re\<|\varphi|^{2\sigma}\psi,\psi\>-\lambda s^{\sigma-2}\Re\<|\varphi|^{2\sigma}\partial_s\psi,\psi\>\\
&=-\frac12\left(\frac{d}{ds}\lambda s^{\sigma-2}\||\varphi|^{\sigma}\psi\|^2+\lambda(\sigma-2) s^{\sigma-3}\||\varphi|^{\sigma}\psi\|^2\right)\end{align*}
as $m\to \infty$. Similarly, if we rewrite $A_2^m$ as 
\begin{align}
\nonumber
A_2^m&=\frac{d}{ds}2\sigma\lambda  s^{\sigma-2} \<|\varphi|^{2\sigma-2}\Re[\overline{v_{\mathrm p}}\psi],\Re[\overline{v_{\mathrm p}}\psi^m]\>\\
\nonumber
&\quad -2\sigma(\sigma-2)\lambda  s^{\sigma-3} \<|\varphi|^{2\sigma-2}\Re[\overline{v_{\mathrm p}}\psi],\Re[\overline{v_{\mathrm p}}\psi^m]\>\\
\nonumber
&\quad -2\sigma\lambda  s^{\sigma-2}\<\partial_s\Re[\overline{v_{\mathrm p}}\psi],|\varphi|^{2\sigma-2}\Re[\overline{v_{\mathrm p}}\psi^m]\>\\
\label{eq_E_ep_2_3}
&=:\frac{d}{ds}A_{21}^m+A_{22}^m
\end{align}
with $A_{21}^m=2\sigma\lambda  s^{\sigma-2} \<|\varphi|^{\sigma-1}\Re[\overline{v_{\mathrm p}}\psi],\Re[\overline{v_{\mathrm p}}\psi^m]\>$, then as $m\to \infty$, 
\begin{align*}
A_{21}^m&\to A_{21}:=2\sigma\lambda  s^{\sigma-2} \||\varphi|^{2\sigma-2}\Re[\overline{v_{\mathrm p}}\psi]\|^2,\\
A_{22}^m&\to A_{22}
:=- \frac{d}{ds}\sigma\lambda  s^{\sigma-2}\||\varphi|^{\sigma-1}\Re[\overline{v_{\mathrm p}}\psi]\|^2-\sigma(\sigma-2)\lambda  s^{\sigma-3} \||\varphi|^{\sigma-1}\Re[\overline{v_{\mathrm p}}\psi]\|^2.
\end{align*}
Moreover, identities \eqref{eq_E_ep_2_0}--\eqref{eq_E_ep_2_3} yield
\begin{align}
\label{eq_E_ep_2_4}
\frac{d}{ds}(A_0^m-A_{11}^m-A_{21}^m)=A_{12}^m+A_{22}^m+A_3^m,
\end{align}
where the RHS is continuous in $s\in (0,1]$. Then, integrating both sides of \eqref{eq_E_ep_2_4} over $[s,t]$ for $0<s,t\le1$, taking the limit $m\to \infty$ and using the bounded convergence theorem imply
$$
\left(A_0-A_{11}-A_{21}\right)(t)-\left(A_0-A_{11}-A_{21}\right)(s)=\int_s^{t}\left(A_{12}+A_{22}+A_3\right)(r)dr. 
$$
This shows
$
\frac{d}{ds}(-A_0+A_{11}+A_{21})=-A_{12}-A_{22}-A_3
$, 
or explicitly, 
\begin{align*}
&\frac{d}{ds}\left(\frac14\|\nabla \psi\|^2
+\lambda s^{\sigma-2}\||\varphi|^{\sigma}\psi\|^2
+2\sigma\lambda  s^{\sigma-2} \||\varphi|^{\sigma-1}\Re[\overline{v_{\mathrm p}}\psi]\|^2\right)\\
&=\frac12\left(\frac{d}{ds}\lambda s^{\sigma-2}\||\varphi|^{\sigma}\psi\|^2+\lambda(\sigma-2) s^{\sigma-3}\||\varphi|^{\sigma}\psi\|^2\right)\\
&\quad +\frac{d}{ds}\sigma\lambda  s^{\sigma-2}\||\varphi|^{\sigma-1}\Re[\overline{v_{\mathrm p}}\psi]\|^2+\sigma(\sigma-2)\lambda  s^{\sigma-3} \||\varphi|^{\sigma-1}\Re[\overline{v_{\mathrm p}}\psi]\|^2\\
&\quad -2\sigma\lambda^2  s^{2\sigma-4} \<|\varphi|^{2\sigma-2}\Re[\overline{v_{\mathrm p}}\psi],|\varphi|^{2\sigma}\Im[\overline{v_{\mathrm p}}\psi]\>.
\end{align*}
This is equivalent to \eqref{eq_E_ep_1}. \qed

\subsection{Proof of the pseudo-conformal conservation law}
Here we prove the pseudo-conformal conservation law \eqref{law} for $E(s)$ following the argument by \cite{Ozawa_CVPDE,Fujiwara_Miyazaki}. 
Recall that  for all $\alpha<1$, $$v \in C_{\mathrm w}((0,1],H^1(\R))\cap W_{\mathrm{loc}}^{1,\infty}((0,1],H^{-1}(\R))\cap C((0,1],H^\alpha(\R)).$$ Let 
$
v^m=J_mv$,  $J_m=(I-\frac1m\Delta)^{-1}$, 
$f=\lambda s^{\sigma-2}|v|^{2\sigma}v$ and  $f^m=J_mf
$. Fix $0<s\le\tau\le s_1$ arbitrarily. By  Duhamel's formula \eqref{Duhamel}, we have
\begin{align}
\nabla v^m(s)=e^{i(s-\tau)\Delta/2}\nabla v^m(\tau)-i\int_{\tau}^s e^{i(s-r)\Delta}\nabla f^m(r)dr
\label{ap_B_0}
\end{align}
in $L^2(\R)$ thanks to the smoothing operator $J_m$. 
Then
\begin{align}
\frac12\|\nabla v^m(s)\|^2
&=\frac12\|\nabla v^m(\tau)\|^2-\Im\int_{\tau}^s\<e^{i(r-\tau)\Delta/2}\nabla v^m(\tau),\nabla f^m(r)\>dr\nonumber\\
&\quad +\frac12\left\|\int_{\tau}^s e^{-ir\Delta/2}\nabla f^m(r)dr\right\|^2,
\label{ap_B_1}
\end{align}
where $\|\nabla v^m(r)\|^2\to \|\nabla v(r)\|^2$ as $m\to \infty$ for $r=s,\tau$ since $v(s),v(\tau)\in H^1(\R)$ and $J_m\to I$ strongly on $L^2(\R)$. For the RHS of \eqref{ap_B_1}, 
by using Fubini's theorem, the formula $$\int_{\tau}^s\int_{\tau}^s (\cdots) drdr'=2\int_{\tau}^s\int_{\tau}^r (\cdots) dr'dr$$ and \eqref{ap_B_0}, one can rewrite the last term of \eqref{ap_B_1} as
\begin{align*}
&\frac12\Re\int_{\tau}^s \int_{\tau}^s\<e^{-ir'\Delta/2}\nabla f^m(r'),e^{-ir\Delta/2}\nabla f^m(r)\>drdr'\\
&=\Re\int_{\tau}^s\left\langle\int_{\tau}^r e^{i(r-r')\Delta/2}\nabla f^m(r') dr',\nabla f^m(r)\right\rangle dr\\
&=\Re\int_{\tau}^s\langle i\nabla v^m(r)-ie^{i(r-\tau)\Delta/2}\nabla v^m(\tau),\nabla f^m(r)\rangle dr\\
&=-\Im\int_{\tau}^s\<\nabla v^m(r),\nabla f^m(r)\>dr+\Im\int_{\tau}^s\<e^{i(r-\tau)\Delta/2}\nabla v^m(\tau),\nabla f^m(r)\>dr,
\end{align*}
where  the last term cancels out with the second term in the RHS of \eqref{ap_B_1}. Moreover, by regarding $\<\cdot,\cdot\>$ as the duality coupling $\<\cdot,\cdot\>_{H^{-1},H^1}$, we obtain 
\begin{align*}&-\Im\int_{\tau}^s\<\nabla v^m(r),\nabla f^m(r)\> dr=\Im\int_{\tau}^s\<J_m\Delta v(r), f^m(r)\>dr\\&=2\Im\int_{\tau}^s\<J_m\{-i\dot v(r)+f(r)\},f^m(r)\>dr\\&=-2\Re\int_{\tau}^s\< \dot v(r),f(r)\>dr-2\Re\int_{\tau}^s\<(J_m^2-I)\dot v(r),f(r)\>dr,\end{align*} where $\dot v=\partial_s v$. By \eqref{NLS_v}, the first term of the RHS is equal to $$-\frac{\lambda s^{\sigma-2}}{\sigma+1} \int_\R |v(s)|^{\sigma+2}dx+\frac{\lambda \tau^{\sigma-2}}{\sigma+1} \int_\R  |v(\tau)|^{2\sigma+2}dx+\frac{(\sigma-2)\lambda }{\sigma+1}\int_{\tau}^s r^{\sigma-3}\int_\R|v(r)|^{2\sigma+2}dxdr.$$Besides, the second term of the RHS converges to $0$ as $m\to \infty$ since $J_m^2\to I$ on $H^{-1}(\R)$ and hence $\<(J_m^2-I)\dot v(r),f(r)\>\to 0$ as $m\to \infty$ for a.e $r\in [s,\tau]$, and $$|\<(J_m^2-I)\dot v(r),f(r)\>|\le 2\|\dot v(r)\|_{H^{-1}}\|f(r)\|_{H^1}\lesssim 1$$uniformly in $r\in [\tau,s]$ and $m>0$. 
Hence, by letting $m\to \infty$ in \eqref{ap_B_1}, we obtain \eqref{law}. \qed

\section*{Acknowledgments}
M. K. is partially supported by JSPS KAKENHI Grant Numbers JP24K06796 and JP26K00612. H. M. is partially supported by JSPS KAKENHI Grant Numbers JP21K03325, JP24K00529, JP26K22272 and JP26K00612. This work was supported by the Research Institute for Mathematical Sciences, an International Joint Usage/Research Center located in Kyoto University.


\end{document}